\documentclass[12pt, leqno]{amsart}

\usepackage{esint} 
\usepackage{amsmath}
\usepackage{amssymb}
\usepackage{amsfonts}
\usepackage{enumerate}
\usepackage{enumitem}
\usepackage{mathrsfs}
\usepackage{mathtools}
\usepackage{xcolor}

\usepackage{float}
\usepackage{tikz}
\usetikzlibrary{lindenmayersystems}
\usepackage{pgfplots}
\usepackage{comment}

\usepackage{amsmath}
\usepackage{amssymb}
\usepackage{amsthm} 
\usepackage{epsf}
\usepackage{graphicx}
\usepackage{esint} 
\usepackage{dsfont}
\usepackage{hyperref} 
\hypersetup{backref,  pdfpagemode=FullScreen,  colorlinks=true} 
\usepackage[french,english]{babel}

\numberwithin{equation}{section}

\newcommand{\R}{{\mathbb R}}
\newcommand{\RR}{{\mathbb R}}

\newcommand{\sm}{\setminus}

\newcommand{\cW}{\mathcal W} 
\newcommand{\cF}{\mathcal F} 
\newcommand{\bD}{\mathbb D} 
\newcommand{\bF}{\mathbb F} 
\newcommand{\ol}{\overline} 

\renewcommand{\d}{\partial}
\newcommand{\ms}{\medskip}

\newcommand{\diam}{{\rm diam}}
\newcommand{\dist}{{\rm dist}}
\newcommand{\po}{\partial\Omega}

\newtheorem{theorem}{Theorem}[section]
\newtheorem{lemma}[theorem]{Lemma}

\newtheorem*{claim*}{Claim}
\newtheorem{corollary}[theorem]{Corollary}

\theoremstyle{remark}
\newtheorem{remark}{Remark}
\theoremstyle{definition}
\newtheorem{example}{Example}
\newtheorem*{problem}{Problem}
\newtheorem{definition}{Definition}
\DeclareMathOperator{\Tr}{Tr}
\DeclareMathOperator{\diver}{div}

\DeclareMathOperator{\osc}{osc}

\begin{document}

	\title{Dimension and structure of the Robin Harmonic Measure on Rough Domains}
	\author{Guy David, Stefano Decio, Max Engelstein, Svitlana Mayboroda, Marco Michetti}
	\thanks{All the authors wish to thank Marcel Filoche for many stimulating conversations and providing background on the Robin problem. The authors also wish to thank Tom\'as Merch\'an for extensive notes on the regularity theorems proven in Sections \ref{s:neumannosc} and \ref{sec:density}. Finally, we thank Jill Pipher for the idea of using the representation formula for a simpler proof of absolute continuity of harmonic measure with respect to the Hausdorff measure -- see Section 7.
	G.D. and M.M. were partially supported by the Simons Foundation grant 601941, GD. 
	S.M. and S.D. were supported in part by
 the Simons foundation grant 563916, SM. SM was in addition supported by the NSF grant DMS 1839077.
	M.E. was partially supported by NSF DMS CAREER 2143719.}
	\subjclass[2020]{}
	\address{GD: Universit\'e Paris-Saclay, CNRS, Laboratoire de math\'ematiques d'Orsay,
		91405 Orsay, France}
  \address{SD: Department of Mathematics, ETH Z\"urich, 8092, Switzerland}
 \address{ME: School of Mathematics, University of Minnesota, Minneapolis, MN, 55455, USA.}
  \address{SM: Department of Mathematics, ETH Z\"urich, 8092, Switzerland, and School of Mathematics, University of Minnesota, Minneapolis, MN, 55455, USA.}
  \address{MM: Dipartimento di Scienze di Base e Applicate per l’Ingegneria, Universit\`a di Roma “La
Sapienza”, 00161 Rome, Italy}
	
	\email{guy.david@universite-paris-saclay.fr}
	\email{stefano.decio@math.ethz.ch}
	\email{mengelst@umn.edu} \email{svitlana.mayboroda@math.ethz.ch; svitlana@umn.edu}
	\email{marco.michetti@uniroma1.it }
	
	\begin{abstract}
	The present paper establishes that the Robin harmonic measure is quantitatively mutually absolutely continuous with respect to the surface measure on any Ahlfors regular set in any (quantifiably) connected domain for any elliptic operator. This stands in contrast with analogous results for the Dirichlet boundary value problem and also contradicts the expectation, supported by simulations in the physics literature, that the dimension of the Robin harmonic measure in rough domains exhibits a phase transition as the boundary condition interpolates between completely reflecting and completely absorbing. 
		
	In the adopted traditional language, the corresponding harmonic measure exhibits no dimension drop, and the absolute continuity necessitates neither rectifiability of the boundary nor control of the oscillations of the coefficients of the equation. The expected phase transition is rather exhibited through the detailed non-scale-invariant weight estimates. 
	
	\end{abstract}

        \maketitle
	\tableofcontents
	
	\section{Introduction}\label{s:intro}
	
	The questions of the dimension and structure of the harmonic measure have spurred enormous amounts of activity in the past 20--30 years and led to some far-reaching breakthroughs. The first is still famously open, the second one has been resolved just recently. 
	
	The dimension of the harmonic measure is only understood on a plane. The fundamental results of Makarov \cite{Makarov1, Makarov2} and Jones and Wolff \cite{JonesWolff}, establish that for a planar domain $\Omega\subset \RR^2$ there is an at most one-dimensional subset of $E\subset \po$ which supports the harmonic measure, no matter what is the dimension of $\po$ itself. In $\RR^n$ with $n\geq 3$, the problem is still widely open. The situation must be different from $\RR^2$, for Wolff \cite{Wolff} has constructed a set whose harmonic measure is supported on a portion of the boundary of dimension strictly bigger than $n-1$. The only general positive result, due to Bourgain, merely assures that the dimension of the harmonic measure is smaller than the ambient one \cite{Bourgain}, for instance, smaller than $3-10^{-15}$ in $\RR^3$ \cite{Badger1}. In the presence of extra assumptions on the boundary of the domain, we know a little more, and in general expect a so-called dimension drop, $\dim \omega<\dim \po$ when $\dim \po \neq n-1$, see, e.g. \cite{Volberg1, Volberg2, Azzamdrop, Carleson, Batakis, Tolsadrop}, although even this is challenged by recent counterexamples \cite{NoDrop}.
	
	The question of the structure of the harmonic measure, or rather the structure of the set supporting the measure, is better understood. Starting from the 1916 F.\&M. Riesz theorem on a plane \cite{RR}, the developments of harmonic analysis have brought higher dimensional analogues, the first converse \cite{seven}, and finally, the full geometric description \cite{AHMMT}. Ignoring connectivity issues, harmonic measure is quantifiably absolutely continuous with respect to the Hausdorff measure if and only if the set is uniformly rectifiable. 
		
Classical harmonic measure corresponds to the Dirichlet boundary condition, $u={\bf 1}_E$ on $\partial\Omega$, $E\subset \partial\Omega$. Motivated, in part,  by conjectures in the physics literature and, in part, by a potentially new limiting approach to classical open problems in mathematics, we turned our attention to the Robin boundary condition, $\frac{1}{a}A\nabla u\cdot \nu+u={\bf 1}_E$. Formally speaking, the limit $a \uparrow \infty$ is Dirichlet and the limit  $a\downarrow 0$ is Neumann, and hence one could expect a transition between the Dirichlet ``dimension drop" of the harmonic measure and the Neumann case, when morally the harmonic measure is equal to the surface measure of the boundary. This phase transition was also suggested to us by physicists. 

This paper proves that, quite to the contrary, for any finite $a>0$ both the dimension and the structure of the Robin harmonic measure are dramatically different from the Dirichlet case. There is no dimension drop and rectifiability plays no role: the Robin harmonic measure is absolutely continuous with respect to the Hausdorff measure, for all reasonably accessible sets, of however large dimension. Moreover, again contrary to the Dirichlet scenario, this result holds for all elliptic operators, without any control on the oscillations of the coefficients. The expected phase transition, however, does happen, but in a very different sense: it is seen in the scaling behavior of the Radon-Nikodym derivative of the Robin harmonic measure, which, also contrary to the classical scenario, is not scale invariant in the present circumstances. 

Let us discuss the details. 
	
 Recall that $u\in C^2(\Omega)\cap C^1(\overline{\Omega})$ is a (classical) solution to the Robin problem with data $f\in C(\partial \Omega)$ if \begin{align}
		\label{problem}
		\begin{cases}
			-\mathrm{div}\left(A\nabla u\right)=0 &\text{in} \ \Omega,\\
			\frac{1}{a}A\nabla u\cdot \nu+u=f &\text{on} \ \partial \Omega,
		\end{cases}
	\end{align}
	where $\nu$ is the outward unit normal vector to $\partial \Omega$, $a\geq 0$ is a real number and $A$ is a uniformly elliptic, not necessary symmetric, real matrix valued function. 
These boundary value problems have been treated mathematically from a variety of perspectives, e.g. \cite{BBC, LS, BNNT, kriv}, starting with the foundational work of Fourier.  In fact, the Robin boundary condition is occasionally referred to as the Fourier boundary condition as the invention of the Fourier series for the problem of the heated rod in mathematical terms corresponded to the requirement that the flux through the end points is proportional to the difference of the interior and exterior temperature, thus, the Robin-type boundary data.  Needless to say, Robin-type problems are ubiquitous in physics \cite{GFS, GFS2} and biology \cite{bioref}. Parenthetically, the corresponding eigenvalue problem has also been intensively studied, in physics and in mathematics, e.g., in \cite{B86, D06, BG15, BFK17}

	As the parameter $a \uparrow \infty$ or $a\downarrow 0$, \eqref{problem} recovers the classical Dirichlet or Neumann boundary problems respectively. The well-posedness of the Dirichlet problem,
	\begin{align}
		\label{dproblem}
		\begin{cases}
			-\mathrm{div}\left(A\nabla u\right)=0 &\text{in} \ \Omega,\\
			u=f &\text{on} \ \partial \Omega,
		\end{cases}
	\end{align}
	for data $f\in L^p(\partial \Omega)$,
	can be rephrased in terms of the relationship between the associated elliptic measure $\omega_L$, given by the Riesz Representation theorem, and the surface measure 
(in the most studied case when $\d\Omega$ is $(n-1)$-dimensional, $\mathcal H^{n-1}|_{\partial \Omega}$). As we pointed out above, how this relationship depends on the smoothness of $A$ and the geometry of $\Omega$ has been the focus of intense study for over 100 years.  We cannot adequately cover the literature here, but point the reader to the surveys \cite{JillSurvey, TatianaSurvey}. 
	
	In particular, for the Laplacian ($A \equiv I$), it was recently shown in \cite{AHMMT} that 
a combination of quantitative rectifiability of $\partial \Omega$ and quantitative connectedness of $\Omega$ is equivalent to the quantitative mutual absolute continuity of 
$\omega_{-\Delta}$ and $\mathcal H^{n-1}|_{\partial \Omega}$. A slightly weakened version of this result also holds if $A$ is not the identity but rather a Carleson-type perturbation of the identity, see \cite{KP, HMMTZ} and also \cite{HMM, AGMT}. The Carleson-type restrictions on the oscillations of the coefficients of $A$ are sharp and necessary, as there are well known examples showing that $\omega_L$ and $\mathcal H^{n-1}|_{\partial \Omega}$ can be mutually singular even in the ball (see \cite{MM, CFK}) for general symmetric elliptic  $A$, and examples showing that $\omega_L$ can be comparable to $\mathcal H^{n-1}|_{\partial \Omega}$ in the complement of a purely unrectifiable  Cantor set \cite{DM}. The situation for non-symmetric $A$ can be even more complex, see, e.g. \cite{KKPT}.
	
In this paper, we first develop the elliptic theory necessary to construct 
the Robin elliptic measure, $\omega_{L,R}$, in domains with low regularity, 
in particular, 
one sided NTA domains with an Ahlfors regular boundary 
of co-dimension strictly between $0$ and $2$. 
Then, surprisingly, and in contrast with the theory for the Dirichlet problem outlined above, we show that for arbitrary real-valued, elliptic $A$, and for every domain as above, this Robin measure is quantitatively mutually absolutely continuous with respect to surface measure. The dimension of the underlying boundary, rectifiability, and control of the oscillations of $A$, play no role.  We immediately work in full generality, but let us underline that our results are new even when $\partial\Omega$ has dimension $n-1$ and the underlying operator is the Laplacian. 

At this point, let us turn to the detailed discussion of our geometric assumptions.

\subsection{Geometric assumptions}

We first introduce some definitions. Throughout this document our domains $\Omega \subset \mathbb R^n$ will satisfy the following quantitative connectedness  assumptions 
	
	\begin{enumerate}[label = (C\arabic*)]
		\item \label{CC1} Interior corkscrew condition: there exists a $M > 1$ such that for every $Q \in \partial \Omega$ and $0 < r < \mathrm{diam}(\Omega)/10$,
 there exists an $A_r(Q) \in \Omega \cap B(Q,r)$ such that $\mathrm{dist}(A_{r}(Q), \partial \Omega) \geq M^{-1}r$. 
		\item \label{CC2} Interior Harnack chains: there exists a constant $M > 1$ such that for all $0 < r < \mathrm{diam}(\Omega)/10$, 
$\varepsilon > 0$, $k \geq 1$, 
 and all $x, y\in \Omega$ with $\mathrm{dist}(x,y)=r$ and $\mathrm{dist}(x,\partial \Omega)\geq \epsilon$, $\mathrm{dist}(y,\partial \Omega) \geq \epsilon$ with $2^k \epsilon \geq r$ there exists $Mk$ balls $B_1,..., B_{Mk} \subset \Omega$ with centers $x_1,\ldots, x_{Mk}$ such that $x_1 = x, x_{Mk} = y$, 
 $r(B_i) \leq \frac{1}{10}\mathrm{dist}(x_i, \partial \Omega)$ for all $i$, 
 and $x_{i+1} \in B_i$ for $i < Mk$. 
\end{enumerate}

 In addition, we will add a metric condition on the size of $\partial\Omega$ or rather the measure $\sigma$ supported by $\po$. 
  \begin{definition} \label{d:mixed}
 We will say that the pair $(\Omega,\sigma)$ is a {\bf (one-sided NTA) pair of mixed dimension}
 when $\Omega$ satisfies \ref{CC1} and \ref{CC2}, $\sigma$ is a doubling measure 
 whose support is $\d\Omega$, and in addition there are constants $d > n-2$ and $c_d > 0$ such that
 \begin{equation} \label{1n3}
\sigma(B(Q,r)) \geq c_d (r/s)^{d}\sigma(B(Q,s))
\ \text{ for $Q \in \d\Omega$ and } 0 < s < r < \diam(\Omega).
\end{equation}
Recall that doubling means that there is a constant $C \geq 0$ such that 
\begin{equation} \label{1n4}
\sigma(B(Q,2r)) \leq C \sigma(B(Q,r))
\ \text{ for $Q \in \d\Omega$ and } 0 < r < \diam(\Omega).
\end{equation}
 \end{definition}
 
 For a pair of mixed dimension, we will refer to the constants $d$, $C$ and $c_d$
in the conditions above as the geometric constants of $\Omega$. 

Clearly, domains which satisfy the above conditions include smooth domains (where the boundary is locally given by the graph of a $C^1$ function) and also Lipschitz domains and chord arc domains. 
Much more generally, a particular case of $\sigma$ and $\partial\Omega$ satisfying the conditions above is provided by the Ahlfors regular sets of dimension $n-2<d<n$. We will single out the corresponding assumption as 
  \begin{enumerate}[resume, label = (C\arabic*)]
 		\item  \label{CC3} 
		$\partial \Omega$ is a $d-$Ahlfors regular set, where $n-2<d<n$: there exists a Borel regular measure, $\sigma$, supported on $\partial \Omega$ and a $C > 1$ such that for every $Q\in \partial \Omega$ and $r < \mathrm{diam}(\partial \Omega)/10$, 
		$$C^{-1}r^{d} \leq \sigma(B(Q,r)\cap \partial \Omega)  \leq Cr^{d}.$$
	\end{enumerate}
	
	We note that (it is classical and follows from a covering argument that) 
if $\partial \Omega$ supports {\it a} $d$-Ahlfors regular measure, $\sigma$, then the $d$-dimensional Hausdorff measure $\mathcal H^{d}|_{\partial \Omega}$ is also $d$-Ahlfors regular 
(i.e., satisfies the same condition as $\sigma$ above), with a constant that depends on $C> 0$ above and the dimension, and then $\sigma = h \mathcal H^{d}|_{\partial \Omega}$ for some $h$ such that
$C^{-1} \leq h \leq C$. Our results are still new and of interest if one takes $\sigma \equiv \mathcal H^{d}|_{\partial \Omega}$ throughout, but we emphasize that $\sigma$ can be an arbitrary 
$d$-Ahlfors regular measure. 
	This is  convenient, because replacing $\sigma$ with  an equivalent measure
$h d\sigma$ amounts to multiplying $a$ in \eqref{problem} by $h$;
see  the weak  formulation \eqref{weakform} below, where we now  integrate with respect
to $h d\sigma$ instead of $d\sigma$.
Thus all the theorems here apply (with suitable modifications to the statements) to variable $a$ (as long as $a, a^{-1} \in L^\infty(d\sigma$)).
	
 We will call a domain that satisfies \ref{CC1}, \ref{CC2} and \ref{CC3} a one-sided NTA domain with $d$-Ahlfors regular boundary, and in that case, the constants in \ref{CC1}, \ref{CC2} and \ref{CC3} become the geometric constants of $\Omega$. 
 
 However, for various reasons that will be detailed in Remark \ref{rk1}, we did not want to be restricted to Ahlfors regular sets.  In particular, with an eye on future free boundary results and underlying applications in physics, we want to be able to treat the domains exhibiting different behavior at different scales, such as pre-fractals, which are $n-1$ dimensional microscopically and $d$-dimensional macroscopically, $d$ being the dimension of the ultimate fractal.  Additionally, the proofs here occasionally require us to``hide" bad parts of the boundary via the construction of so-called sawtooth regions, which are of mixed dimension when $\partial \Omega$ is $d$-Ahlfors regular with $d\neq n-1$.

We will assume our domain  $\Omega$ to be bounded to avoid complications, but we 
emphasize that the diameter of $\Omega$ is not a ``geometric constant" in this sense and we will always try to make explicit the dependence of our bounds on $\mathrm{diam}(\Omega)$ (should there be any).

Of course there are many domains satisfying the properties above which do not have 
rectifiable boundaries:
	
	\begin{example}\label{ex:roughdomain}[A rough domain satisfying the above properties] 
		A guiding example of a rough domain covered by our theory is the complement of the four corner Cantor set in the plane.  Such a Cantor set is built starting from a square and erasing everything but the four corners of sidelength four times smaller than the previous generation. To get a bounded domain satisfying the conditions of our paper, take the complement of the four corner Cantor set in a large ball (with carefully chosen radius so the boundary of the ball does not intersect with the fractal). Note that this example is 
$1$-Ahlfors regular in the plane; it would be easy to construct similar $(n-1)$-Ahlfors regular
sets in $\R^n$ (or even of any dimension $d\in (n-2,n)$), but this one is the most striking counterexample.

		\begin{figure}[H]
			\centering
			\begin{tikzpicture}
				\draw[blue, fill=blue!50](0cm, 4cm) circle [radius=17mm];
				\node (K1) at (0,2) {$C_0$};
				\draw[blue, fill=blue!50](4cm, 4cm) circle [radius=17mm];
				\node (K2) at (4,2) {$C_1$};
				\draw[blue, fill=blue!50](8cm, 4cm) circle [radius=17mm];
				\node (K2) at (8,2) {$C_2$};
				\node [draw, blue, fill=white, shape=rectangle, minimum width=2cm, minimum height=2cm, anchor=center] at (0,4) {};
				\node [draw, blue, fill=white, shape=rectangle, minimum width=0.666cm, minimum height=0.666cm, anchor=center] at (4+1-0.333, 4+1-0.333) {};
				\node [draw, blue, fill=white, shape=rectangle, minimum width=0.666cm, minimum height=0.666cm, anchor=center] at (4-1+0.333, 4-1+0.333) {};
				\node [draw, blue, fill=white, shape=rectangle, minimum width=0.666cm, minimum height=0.666cm, anchor=center] at (4+1-0.333, 4-1+0.333) {};
				\node [draw, blue, fill=white, shape=rectangle, minimum width=0.666cm, minimum height=0.666cm, anchor=center] at (4-1+0.333, 4+1-0.333) {};
				
				\node [draw, blue, fill=white, shape=rectangle, minimum width=0.222cm, minimum height=0.222cm, anchor=center] at (8+1-0.333+0.333-0.111, 4+1-0.333+0.333-0.111) {};
				\node [draw, blue, fill=white, shape=rectangle, minimum width=0.222cm, minimum height=0.222cm, anchor=center] at (8+1-0.333-0.333+0.111, 4+1-0.333+0.333-0.111) {};
				\node [draw, blue, fill=white, shape=rectangle, minimum width=0.222cm, minimum height=0.222cm, anchor=center] at (8+1-0.333+0.333-0.111, 4+1-0.333-0.333+0.111) {};
				\node [draw, blue, fill=white, shape=rectangle, minimum width=0.222cm, minimum height=0.222cm, anchor=center] at (8+1-0.333-0.333+0.111, 4+1-0.333-0.333+0.111) {};

				\node [draw, blue, fill=white, shape=rectangle, minimum width=0.222cm, minimum height=0.222cm, anchor=center] at (8-1+0.333+0.333-0.111, 4-1+0.333+0.333-0.111) {};
				\node [draw, blue, fill=white, shape=rectangle, minimum width=0.222cm, minimum height=0.222cm, anchor=center] at (8-1+0.333-0.333+0.111, 4-1+0.333+0.333-0.111) {};
				\node [draw, blue, fill=white, shape=rectangle, minimum width=0.222cm, minimum height=0.222cm, anchor=center] at (8-1+0.333+0.333-0.111, 4-1+0.333-0.333+0.111) {};
				\node [draw, blue, fill=white, shape=rectangle, minimum width=0.222cm, minimum height=0.222cm, anchor=center] at (8-1+0.333-0.333+0.111, 4-1+0.333-0.333+0.111) {};
				
				\node [draw, blue, fill=white, shape=rectangle, minimum width=0.222cm, minimum height=0.222cm, anchor=center] at (8+1-0.333+0.333-0.111, 4-1+0.333+0.333-0.111) {};
				\node [draw, blue, fill=white, shape=rectangle, minimum width=0.222cm, minimum height=0.222cm, anchor=center] at (8+1-0.333-0.333+0.111, 4-1+0.333+0.333-0.111) {};
				\node [draw, blue, fill=white, shape=rectangle, minimum width=0.222cm, minimum height=0.222cm, anchor=center] at (8+1-0.333+0.333-0.111, 4-1+0.333-0.333+0.111) {};
				\node [draw, blue, fill=white, shape=rectangle, minimum width=0.222cm, minimum height=0.222cm, anchor=center] at (8+1-0.333-0.333+0.111, 4-1+0.333-0.333+0.111) {};
				
				\node [draw, blue, fill=white, shape=rectangle, minimum width=0.222cm, minimum height=0.222cm, anchor=center] at (8-1+0.333+0.333-0.111, 4+1-0.333+0.333-0.111) {};
				\node [draw, blue, fill=white, shape=rectangle, minimum width=0.222cm, minimum height=0.222cm, anchor=center] at (8-1+0.333-0.333+0.111, 4+1-0.333+0.333-0.111) {};
				\node [draw, blue, fill=white, shape=rectangle, minimum width=0.222cm, minimum height=0.222cm, anchor=center] at (8-1+0.333+0.333-0.111, 4+1-0.333-0.333+0.111) {};
				\node [draw, blue, fill=white, shape=rectangle, minimum width=0.222cm, minimum height=0.222cm, anchor=center] at (8-1+0.333-0.333+0.111, 4+1-0.333-0.333+0.111) {};
			\end{tikzpicture}
			\caption{First three steps of the construction of the complement of the four corner Cantor set}
		\end{figure}
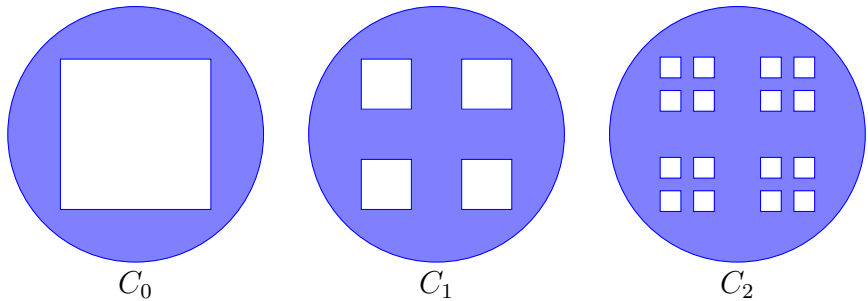
	\end{example}

  \begin{example}\label{ex:koch}[Another rough domain of co-dimension $\neq 1$] Another rough domain for which our results apply is the famous Koch snowflake, which is built starting from an equilateral triangle in $\mathbb{R}^2$, dividing each side into thirds, building another equilateral triangle with the middle third and iterating. The dimension of the boundary of the Koch snowflake is $\log(4)/\log(3)>1$.

\begin{figure}[H]
\centering{ 
  \begin{tikzpicture}
\draw [l-system={rule set={F -> F-F++F-F}, step=2pt, angle=60,
   axiom=F++F++F, order=4}] lindenmayer system -- cycle;
\end{tikzpicture}
\caption{First three steps of the construction of a Koch snowflake}
}
      \end{figure}
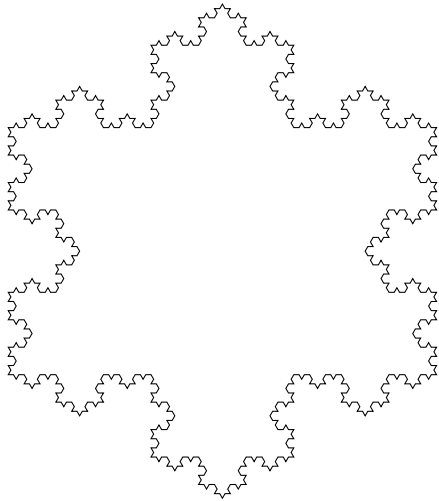
  \end{example}
  
  \subsection{Definition of the Robin elliptic measure and fundamental regularity problems of solutions to the Robin problem}
	
	In domains with this level of regularity we do not expect $\nabla u$ to extend continuously to the boundary. 
Thus we must define solutions of \eqref{problem} weakly. We will say that 
$u\in W^{1,2}(\Omega)$ (see \eqref{2a1} for definition) is a (weak) solution to  \eqref{problem} when 
	\begin{align}
		\label{weakform}
		\frac{1}{a}\int_{\Omega}A\nabla u\nabla \varphi+\int_{\partial \Omega}u\varphi=\int_{\partial \Omega} \varphi f , 
		\qquad \forall \varphi \in C_c^\infty(\mathbb R^n).
	\end{align}
	Here  and below, integrals on $\Omega$ are always with respect to Lebesgue measure and integrals on $\partial \Omega$ are always with respect to $\sigma$.
	
	Our first main results are the local H\"older continuity of weak solutions, up to the boundary,
	and then the existence of a solution to \eqref{weakform} and a representation formula:
	
	\begin{theorem}\label{t:contandhm}
		Let $(\Omega, \sigma)$ be a one-sided NTA pair of mixed dimension, with $\Omega$ bounded,
		and let $A$ be a uniformly elliptic real matrix valued function. 
		Let $u\in W^{1,2}(\Omega)$ satisfy \eqref{weakform} for some $f\in C^{0,\beta}(\partial \Omega)$. There exists an $\alpha \in (0,1)$, which depends only on $\beta$, 
the ellipticity constant  for  $A$, and 
the geometric constants of $\Omega$, such that $u\in C^{0,\alpha}(\overline{\Omega}).$
		
		Furthermore, there exists a family of Radon measures supported on the boundary $\{\omega_{L,\sigma, R}^X\}_{X\in \Omega}$ such that  $u$ solves \eqref{weakform} for some $f\in C(\partial \Omega)$ if and only if 
		$$u(X) = \int_{\partial \Omega} f(Q)d\omega_{L,\sigma, R}^X(Q)
		\qquad \text{ for } X \in \Omega. 
		$$
	\end{theorem}
	
	We call the measures, $\omega_{L,\sigma, R}^X$, the Robin elliptic measures. We often omit the $L,\sigma, R$ from the notation if no confusion is possible. The key steps in constructing these measures is our work in Sections \ref{s:neumannosc} and \ref{sec:density}, where we show that weak solutions to the Robin problem with continuous data are continuous up to the boundary. This strategy follows the general approach of the construction of Dirichlet harmonic measure, but its execution differs substantially in many places. We discuss this in more detail in Subsection \ref{ss:comparison}, but mention here that these boundary estimates seem to be new even for the Neumann problem for the class of domains and operators we consider here.\footnote{While writing this paper we were informed by Steve Hofmann that in joint work with his student, Derek Sparrius, they have independently established a Harnack inequality up to the boundary for solutions to the Neumann problem (our Theorem \ref{th:neumannharnack}). We emphasize that these works were done 
independently and have different end aims (Hofmann-Sparrius is interested in the solvability of the Neumann problem in these domains, as in \cite{KP2}).}
	
	While a natural object, we are unaware of any reference which rigorously constructs Robin harmonic measure, even for the Laplacian in smooth domains. There is a very interesting heuristic construction of Robin harmonic measure in \cite{GFS} as the hitting measure of a partially reflected Brownian motion (the parameter $a$ corresponds to the ratio between the absorption probability and the reflection probability). However, making this construction rigorous here would require more delicate analysis of the ``local time" random variable. This seems to be extremely subtle, even in Lipschitz domains, see, e.g., \cite{BBC} which analyzes the local time variable to study the homogeneous Robin problem with an interior source. Interestingly, our methods (in particular the Harnack inequality up to the boundary) do partially address a problem studied in \cite{BBC} concerning domains for which the whole boundary is ``active" (see Section \ref{sec:density} for further discussion), but we should emphasize that our results in this direction neither imply nor are implied by \cite{BBC}.

\subsection{Absolute continuity of the Robin harmonic measure}	
	We now turn our attention to the behavior of the Robin elliptic measure, $\omega^X$. Here we find that the theory for the Robin problem deviates substantially from the corresponding Dirichlet theory:
	
	\begin{theorem}\label{t:mainac}[Absolute Continuity of the Robin Measure]
	Let $(\Omega, \sigma)$ be a one-sided NTA pair of mixed dimension, with $\Omega$ bounded,
	and let $A$ be a uniformly elliptic real matrix valued function. For any $X\in \Omega$ we have 
$$\sigma \ll \omega_{L, \sigma, R}^X \ll \sigma.$$
	\end{theorem}
	
	We emphasize that the corresponding result for the Dirichlet problem is false for the Laplacian in the class of the domains we consider (see \cite{seven}); more specifically, \cite[Theorem 1.1]{seven} implies that $\sigma \not \ll \omega_{L, \sigma, D}^X \not \ll \sigma$ for $L = -\Delta$ and $\Omega$ as in Example~\ref{ex:roughdomain} above. It is also false even on a ball for an operator as general as ours. Indeed, it is well-known that the absolute continuity of the classical (Dirichlet) elliptic measure requires control on the oscillation of the coefficients, and, e.g.,  in \cite{MM, CFK}, the authors construct uniformly elliptic $A \in C(\overline{B}(0,1))$ such that the corresponding harmonic measure to the operator $-\mathrm{div}(A(x)\nabla u)$ is not mutually absolutely continuous with respect to the surface measure on the ball.

	One of the components of  our proof of Theorem \ref{t:mainac} 
	and its quantitative version below 
	is a particularly easy representation formula relating the harmonic measure with the Green function, see \eqref{e:representationformula} and \eqref{e:repformulae}. Morally, the Poisson kernel in our case is the Green function itself, due to the Robin conditions. This fact makes the absolute continuity much more transparent than in the Dirichlet setting. A completely new challenge, on the other hand, is a lack of scale invariance, and the corresponding quantitative estimates split into different regimes.

	\begin{theorem}\label{t:ainfinitysmall} 
	[Quantitative mutual 
	absolute continuity of $\omega_{R}^X$ with respect to $\sigma$ at small scales] 
Let $(\Omega, \sigma)$ be a one sided NTA pair in $\R^n$, with $\Omega$ bounded.
Let $A$ be a uniformly elliptic real matrix valued function. 
Let $\omega_R^X = \omega_{L, \sigma, R}^X$ be the associated Robin harmonic measure. 
There exists  $C \geq 1$ (depending on the geometric constants for $(\Omega, \sigma)$
and the ellipticity of $A$) such that for all $x_0\in \partial \Omega$ and $0 < r \leq \diam(\Omega)$
such that $0 < a r^{2-n} \sigma(B(x_0,r)) \leq 1$, 
$X\in \Omega \backslash B(x_0, Cr)$ and Borel measurable sets 
$E \subset \partial \Omega \cap B(x_0, r)$, 
	\begin{equation}\label{e:Ainfinitysmall} 
	C^{-1} \frac{\sigma(E)}{\sigma(B(x_0,r))}\leq \frac{\omega_{R}^{X}(E)}{\omega_{R}^{X}(B(x_0,r)\cap \partial \Omega)}
	\leq C \frac{\sigma(E)}{\sigma(B(x_0,r))}.
	\end{equation}	
	\end{theorem}
	
	\ms
Notice the (unexpected) optimal exponent $1$ in this estimate (compare to usual $A_p$ weights conditions).

We said 
``weight-type'' because our estimates only hold when 
$a r^{2-n} \sigma(B(x_0,r)) \leq 1$.  
 This, however, is natural and necessary. 
 Indeed, consider for instance $\Omega$ with a $d$-Ahlfors regular boundary, and 
 $\sigma = \mathcal H^{d}|_{\d\Omega}$. Contrary to what happens with
the Dirichlet problem, \eqref{problem} is not invariant under rescaling, even when 
$\d\Omega$ is Ahlfors regular of dimension $d$, because 
 if $E\mapsto \omega_R^X(E)$ is the Robin harmonic measure on $\Omega$, then
 $E \mapsto\omega_R^{rX}(rE)$ is the Robin harmonic measure of $\Omega/r$ with 
 parameter $ar^{2-n+d}$. This parameter tends to $0$ when $r$ tends to $0$,
 hence we expect $\omega_R^X(E)$ to behave more like the ``Neumann harmonic measure"
 at small scales. We have put the ``Neumann harmonic measure" in the quotation marks as the concept fully analogous to the Dirichlet case does not exist, but exploiting suitable limits of solutions, heuristically at least, the Neumann harmonic measure should be proportional to $\sigma$. Similarly, when $r$ tends to $+\infty$ (or at least gets very large), we expect $\omega_R^X(E)$ to behave more like the Dirichlet harmonic measure.
 
 Said otherwise, in the Ahlfors regular context, $a^{-\frac{1}{2-n+d}}$ scales like a distance,
and generally we expect that for $a$ fixed, solutions of \eqref{problem} will be ``Dirichlet like" at scales larger than $a^{-\frac{1}{2-n+d}}$ and ``Neumann like" at scales smaller than $a^{-\frac{1}{2-n+d}}$. 
Something similar happens in the mixed dimension case, where the analogue of $ar^{2-n+d}$
is $ar^{2-n} \sigma(B(x,r))$, which also tends to $0$ because of the asymptotic dimensional constraint \eqref{1n3}.

When $ar^{2-n+d} \gg 1$ (the Dirichlet regime), getting  results which degenerate is to be expected, 
since we know that the Dirichlet harmonic measure is not absolutely continuous
with respect to $\sigma$ for the full range of operators and domains covered by our theorems. 
We see this is Theorem \ref{t:ainfinitylarge} below. We do not claim our rate of degeneracy is optimal\footnote{For example, for $A$ symmetric and $n\geq 3$ a variational argument can be used to control the ratio of the harmonic measure by the square root of $ar^{2-n+d}$ times the ratio of the surface measures when $\d\Omega$ is Ahlfors regular of dimension $d$.} and indeed, understanding the precise rate of degeneracy as $ar^{2-n+d}\uparrow \infty$ is an interesting open question which should have implications for the modeling of the ``passivation" 
of irregular surfaces, see, e.g. \cite{FGAS}. We plan to address this in future work. 

	\begin{theorem}\label{t:ainfinitylarge}[Quantitative mutual absolute continuity of 
	$\omega_{R}^X$ with respect to $\sigma$ at large scales]
Let $(\Omega, \sigma)$ be a one sided NTA pair in $\R^n$, with $\Omega$ bounded.
Let $A$ be a uniformly elliptic  real matrix valued function. 
Let $\omega_R^X = \omega_{L, \sigma, R}^X$ be the associated Robin harmonic measure. 
There exists  $C \geq 1$ (depending on the geometric constants for $(\Omega, \sigma)$
and the ellipticity of $A$) and a $\gamma > 0$ (depending also on the geometric constants 
for $(\Omega, \sigma)$ 
and the ellipticity of $A$) such that for all $x_0\in \partial \Omega$ and $0 < r \leq \diam(\Omega)$
such that $a r^{2-n} \sigma(B(x_0,r)) \geq 1$, 
$X\in \Omega \backslash B(x_0, Cr)$ and Borel measurable sets 
$E \subset \partial \Omega \cap B(x_0, r)$, 
	\begin{multline}\label{e:Ainfinitylarge} 
	C^{-1} \left(ar^{2-n} \sigma(B(x_0,r))\right)^{-\gamma}\frac{\sigma(E)}{\sigma(B(x_0,r))}
	\\ \leq \frac{\omega_{R}^{X}(E)}{\omega_{R}^{X}(B(x_0,r)\cap \partial \Omega)}
	\\ \leq C \left(ar^{2-n} \sigma(B(x_0,r))\right)^\gamma \frac{\sigma(E)}{\sigma(B(x_0,r))}. 
	\end{multline}	
	\end{theorem}
	
\ms
We end this section with a short discussion of the hypothesis of Theorems \ref{t:ainfinitysmall} and \ref{t:ainfinitylarge}. We do not know if the interior corkscrew condition can be weakened, but the ``trapped domain" examples in \cite{BBC} suggest that the Robin problem can have strange behavior on domains with exterior cusps. 
	
We also believe (suitably local) versions of our theorems may 
hold if the Harnack chain hypothesis is replaced by some sort of quantitative connectivity with the boundary (e.g. weak local John). We chose to work in this simplified setting in order to highlight the main ideas without introducing additional complications. 
	
	This is also why we decided to assume that $\Omega$ is bounded; probably
most results can be generalized to unbounded domains, perhaps at the price of
some complications because the Neumann problem is not well defined for those domains.

	As previously mentioned, by varying $\sigma$ the above theorems extend to variable rather than constant factor $a$, with $a, a^{-1} \in L^\infty(d\sigma)$. 

 We would also like to emphasize (again) that our results are already 
 new and interesting in the case of co-dimension $1$, 
 which is very well studied in the Dirichlet setting.

	\subsection{Outline of the Paper and Comparison with Other Work }\label{ss:comparison}
	
	In Section \ref{s:TEP} we establish the necessary trace/extension theorems and Poincar\'e inequalities in our setting, which we then use to show the existence of weak solutions in the sense of \eqref{weakform} for any $f \in L^2(\partial \Omega)$  via Lax-Milgram (the presence of the normal derivative in the Robin condition allows for lower regularity data, in contrast with the Dirichlet problem). 
	We note that for the Robin problem there is a huge literature concerning Sobolev regularity and solvability (for a small sampling see, e.g. \cite{LS, YYY, DL, GMN}).
	
	In Sections \ref{s:neumannosc} and 
	\ref{sec:density}, we prove basic elliptic estimates (e.g. Moser, oscillation and density estimates) up to the boundary for (local) weak solutions to the Neumann and Robin problems. This culminates in the proof of Theorem \ref{t:HolderContinuity}, that weak solutions to the Robin problem with H\"older continuous data are H\"older continuous up to the boundary. While the general strategy behind these estimates will  be familiar to experts, we do not believe these results are known outside of Lipschitz domains for either the Neumann or Robin problems (see \cite[Section 3]{LS}, for some of these estimates in the context of the Robin problem in Lipschitz domains and also \cite{VS} which proves continuity up to the boundary for continuous data in Lipschitz domains using other methods). It is also important to us to carefully track the dependence of our estimates on the parameter $a$. In these sections 
we prove a Harnack inequality ``up to the boundary" (Theorem \ref{th:robinharnack} and the discussion afterwards) which leads to our partial solution of \cite[Problem 1.2]{BBC}.

	Once we have the continuity up to the boundary, it is then straightforward to construct 
the Robin elliptic measure using the Riesz representation theorem and the maximum principle; 
we do this in Section \ref{s:RRT}, thus completing the proof of Theorem \ref{t:contandhm} above.

	Finally, in Section \ref{s:abscont}, we prove our main Theorems \ref{t:mainac},  \ref{t:ainfinitysmall} and \ref{t:ainfinitylarge}. As mentioned above, our key observation is a representation formula for Robin harmonic measure proven in Section \ref{s:RRT}, which together with the boundary Harnack inequality (Theorem \ref{th:robinharnack}) immediately gives our main theorems. 
	
\section{Functional Analytic Preliminaries and Poincar\'e}\label{s:TEP}
	
In order to find weak solutions to our problem, we need appropriate extension and trace operators. The main hero here is the space $W = W^{1,2}(\Omega)$ of functions
$f\in L^2(\Omega)$ with a derivative in $L^2(\Omega)$. We equip $W$ with the natural
norm
\begin{equation} \label{2a1}
||u||_W = \Big\{ \int_{\Omega} |\nabla u|^2 dx + \int_{\Omega} |u|^2 d\sigma \Big\}^{1/2}. 
\end{equation}
Here is what we need for our basic theorem of existence of weak solutions.

	\begin{theorem}\label{prop:equivalence of norms} 
Let $(\Omega,\sigma)$ be a one-sided NTA pair of mixed dimension, and assume that $\Omega$ is bounded.
Then
\begin{equation} \label{2a2}
\text{the canonical injection $W^{1,2}(\Omega)\rightarrow L^2(\Omega)$ is compact,}
\end{equation}
\begin{equation} \label{2a3}
\text{there is a trace operator  $\Tr: W^{1,2}(\Omega)\rightarrow L^2(\d\Omega, \sigma)$, which is bounded and compact,}
\end{equation}
with a bound for the norm of $\Tr$ that depends only on the geometric constants of 
$\Omega$, $\sigma(\d\Omega)$, and the diameter of $\Omega$,
\begin{equation} \label{2a4}
\text{the Sobolev norm $\|u\|_{W}$ is equivalent to  
$\|u\|_{\Tr}:=\big(\|\nabla u\|_{L^2(\Omega)}^2+\|\Tr u\|_{L^2(\sigma)}^2 \big )^{\frac{1}{2}}$ ,}
\end{equation}		
\begin{equation} \label{2a5}
\text{$\Omega$ is an extension domain,}
\end{equation}		
which means that every $u \in W^{1,2}(\Omega)$ has an extension 
$\widehat u \in W^{1,2}(\R^n)$, and in addition $||\widehat u||_{W^{1,2}(\R^n)} \leq C ||u||_W$,
where $C$ depends only on the geometric constants of $\Omega$, $\sigma(\d\Omega)$, and its diameter, 
\begin{equation} \label{2a6}
\text{$C_c^\infty(\mathbb R^n)$ is dense in $W^{1,2}(\Omega)$,}
\end{equation}		
and finally 
\begin{equation} \label{2a7} 
W^{1,2}(\Omega) \subset L^p(\Omega) \text{ for some } p >2,
\end{equation}
where $p$ depends only on the geometric constants for $\Omega$, and with a bound for the norm
of the injection that depends only on the geometric constants for $\Omega$, $p$, the diameter
of $\Omega$, and $\sigma(\d\Omega)$.
	\end{theorem} 

\begin{proof}
Fortunately, most of the work was done 
before and this theorem will rather easily follow from the literature.
This will also be a nice place to comment on the mixed dimension condition.

First observe that since our domains are $1$-sided NTA domain, this makes them 
$(\epsilon, \delta)$-domain in the sense of \cite{J} (the definitions are equivalent). 
This alone implies that the $\Omega$ 
considered here are Sobolev extension domains.
Since we shall only use the theorem for the existence result below, we do not track down the 
dependence on $\diam(\Omega)$, which would be easy to sort out from homogeneity anyway.
So \eqref{2a5} holds, but also \eqref{2a2} and \eqref{2a6}. Indeed, for \eqref{2a2} we observe 
that the canonical injection $W^{1,2}(\Omega)\rightarrow L^2(\Omega)$ is the composition
of an extension mapping $\varphi : W^{1,2}(\Omega) \to W^{1,2}(\R^n)$, and the restriction
mapping $\psi: W^{1,2}(\R^n) \to L^2(\Omega)$. This last is compact, by the 
Rellich-Kondrachov theorem (we can even restrict compactly to $L^2$ of a large ball that contains 
$\Omega$). Similarly, \eqref{2a6} holds because if $f\in W$, it has an extension in $W^{1,2}(\R^n)$
which is well approximated, even in $W^{1,2}(\R^n)$, by functions in 
$C_c^\infty(\mathbb R^n)$. Of course all this means is that the theorem of \cite{J} is strong.

For the other properties, let us first discuss the special case when $\d\Omega$ is Ahlfors regular of 
dimension $d \in (n-2,n)$. Then we are clearly in the general theory described, e.g., in \cite{HKT}.
Furthermore, because $\partial \Omega$ is Ahlfors regular (what is called a $d$-set in \cite{AR}), the domains considered here are ``admissible" and thus a trace operator is well defined.
We refer to\cite[Proposition 3]{AR}, the case $k=1$, for the proof. 
In this case, we even have an explicit range of exponents $p$ that work in \eqref{2a7}, because actually
\begin{equation} \label{2a7bis}
W^{1,2}(\Omega) \subset L^p(\Omega) \text{ for all } p \in \left[1, \frac{2n}{n-2}\right],
\end{equation}
with again bounds on the norm of the canonical injection which depends only on $p$, the geometric constants of $\Omega$, and its diameter.

\begin{remark} \label{rk1}
 We decided to state the theorem in the more general case of pairs of mixed dimension, so we 
 will need to prove it in this case too. But first let us explain why we think the notion is interesting
 and comment about it.
 
 The definition is not new: it was introduced by two authors of the present paper and J.~Feneuil in 
\cite{DFM20}, with a little more generality than here, because the authors also wanted to allow domains
with boundaries of low dimension (at the price of making our elliptic coefficients unbounded near 
$\d\Omega$).
Here we do not want to do this, so we take the conditions (H1)-(H6) of \cite{DFM20}
(near page 12), but restrict to the special case when the weight $w$ that defines the measure $m$
on $\Omega$ is identically equal to $1$. Thus $m$ here is just the restriction of the Lebesgue measure to $\Omega$, $\sigma$ here is the same as $\mu$ there, 
and the conditions (H1)-(H6) amount to the definition of (one sided NTA) pair of mixed dimension.
More precisely, (H1) and (H2) are the same as \ref{CC1} and \ref{CC2}, (H3) is the doubling condition \eqref{1n4}, (H4) (the existence of a density for $m$ and the fact that it is doubling)
is an easy consequence of the corkscrew condition \ref{CC1}, (H6) is easy because it is implied
by (H6'), which is true here because the weight is constant, and finally (H5) translates as follows.
For $x\in \d\Omega$ and $0 < r \leq \diam(\d\Omega)$, we introduce the quantity 
\begin{equation} \label{2n9}
\rho(x,r) =r^{n-1} \sigma(B(x,r)\cap \partial \Omega)^{-1}
\end{equation}
(which is equivalent to 
$\frac{|B(x,r)\cap\Omega|}{r \sigma(B(x,r)\cap \partial \Omega)}$ in (2.6) of 
\cite{DFM20}, because of \ref{CC1}, and require as in (2.7) of \cite{DFM20} that
$\rho(x,r) \leq C (r/s)^{1-\varepsilon} \rho(x,s)$ for $0 < s < r \leq \diam(\d\Omega)$.
This transforms into $\sigma(B(x,r)) \geq C^{-1} (r/s)^{n-2+\varepsilon} \sigma(B(x,s))$,
which is just \eqref{1n3} with $d = n-2+\varepsilon > n-2$.
So the present definition is just a special case of the main assumptions of \cite{DFM20},
and we can use the results of \cite{DFM20} in our context; the proofs there are also valid for 
$\d\Omega$ of finite diameter.

Before we come to that, let us give some motivation for this additional generality.
The most obvious one is that some interesting examples of $\d\Omega$ are not Ahlfors regular,
but instead behave at different scales like Ahlfors regular sets of different dimensions.
For instance, in Example \ref{ex:roughdomain} we could decide that when we go from 
generation $n$ to generation $n+1$, we replace each small square $Q$ of sidelength $r_k$
with four squares of sidelength $r_{k+1}$ contained in $Q$, and situated at the four corners of
$Q$. We took $r_{k+1} = \frac14 r_k$, but choosing any $r_k \in [ar_k, br_k]$, for some fixed choice of
constants $a$, $b$, with $0 < a < b < 1/2$ would lead to a pair of mixed dimension, and we could even
let $r_k$ depend on the cube $Q$ and not just the generation. In this case, the appropriate
$\sigma$ is just the probability measure that gives a weight of $4^{-n}$ to each square of generation $n$.
 
 A variant of this example is the prefractal that we obtain when we perform $n$ generations
 of the construction above, get a compact set $K_n$ composed of $4^n$ squares
 $Q$, and let $\Omega$ be the complement of the union of the $4^n$ squares. Thus $\d\Omega$
 is composed of $4^n$ square curves, and again $\sigma$ is the obvious probability measure that gives
 the mass $4^{-n}$ to each of these curves. It is good to have a formalism that treats the prefractals
 and the full fractals the same way. Incidentally, notice that if we do this with the ``one quarter Cantor set''
 of Example \ref{ex:roughdomain}, we still get an Ahlfors regular set of dimension $1$, with uniform bounds, but this is not the case in general. The prefractals associated to 
 Example \ref{ex:koch}, for instance, are $1$-dimensional at small scales and 
 $\log(4)/\log(3)>1$-dimensional at large scales (compared to $1$).
 
 The second reason why mixed dimensions were introduced in \cite{DFM20} is that
 the natural saw-tooth constructions, performed on Ahlfors-regular boundaries $\d\Omega$
 of dimensions $d \neq n-1$, yield boundaries of mixed dimensions (typically, $d$ near $\d\Omega$
 and $n-1$ along the sawtooth, far from $\d\Omega$). Thus it was very convenient to have a general theory ready for those domains, and this is still the case here. We will see this when we prove the
 localization lemma \ref{l:tentspaces}.
 
 The reader may be surprised that there is only one inequality like \eqref{1n3}, and nothing in the opposite direction. It turns out that we don't need one, but maybe the reader will find it comforting to observe
 that in the Ahlfors regular case, the existence of corkscrew points implies that the dimension $d$
 of $\d\Omega$ is strictly less than $n$.
 \end{remark}

Return to the results of \cite{DFM20}. The results there are stated for unbounded boundaries 
$\d\Omega$, but are still true when $\Omega$ is bounded as here. However,
\cite{DFM20} uses homogeneous spaces, and we have to go through the appropriate modifications here. 
Let $\dot{W}$ denote the homogeneous space of functions on $\Omega$ with a derivative in $L^2$,
with the norm $||f||_{\dot{W}} = \big\{ \int_{\Omega} |\nabla f|^2\big\}^{1/2}$.
Theorem 6.6 there gives a bounded trace operator $T : \dot{W} \to \dot{H}$, 
a homogeneous space akin to $H^{1/2}$ on $(\d\Omega, \sigma)$. This space is also
a space of functions defined modulo an additive constants, and it is explicit and defined by its norm 
$||g||_{\dot{H}}$, 
where
\begin{equation} \label{2n10}
||g||_{\dot{H}}^2 = \int_{\d\Omega}\int_{\d\Omega} \frac{\rho(x,|x-y|)^2| \,\, g(x)-g(y)|^2}{|x-y|^n} \, d\sigma(x)d\sigma(y)
\end{equation}
(see (6.5) there). Then, in Theorem 8.5 of \cite{DFM20}, a bounded extension operator 
$E : \dot{H} \to \dot{W}$ is built, with $T \circ E = I_{\dot{H}}$. Notice that then 
$\dot{H} = T(\dot{W})$, so the trace of $\dot{W}$ is explicit.

Here we take the trace operator $\Tr$ to be equal to $T$, i.e., we take the same formula
as in \cite{DFM20}, and it happens that $\Tr(f) \in L^2(\sigma)$ when $f \in W$. We know in addition 
that $\Tr$ goes to the quotient, in fact that $\Tr(f+C) = \Tr(f) + C$ when $f\in W$ and $C$
is a constant. The fact that out $\Tr$ maps $W$ boundedly to $L^2(\sigma)$ could be obtained
from the boundedness of $T : \dot{W} \to \dot{H}$ only, but we'll see explicit bounds anyway.
See near Lemma \ref{th:GlobalPoincare3}.

Notice that when we multiply $\sigma$ by $\lambda > 0$, we do not change the mixed 
dimension constants, nor the homogeneous Sobolev norm $||g||_{\dot{W}}$ above. 
However we multiply $\|g\|_{L^2(\sigma)}$ by $\sqrt{\lambda}$.
This explains why the various norms in the statement of Theorem \ref{prop:equivalence of norms}
depend on $\sigma(\d\Omega)$. We prefer not to normalize this out for the moment.

On $\d\Omega$, we use the no longer homogeneous space $H$ defined by the norm
$||\cdot ||_{H}$, where
\begin{equation} \label{2a8}
||g||_H^2 = \int_{\d\Omega} |f|^2 d\sigma + 
\int_{\d\Omega}\int_{\d\Omega} \frac{\rho(x,|x-y|)^2 \, |g(x)-g(y)|^2}{|x-y|^n} \, d\sigma(x)d\sigma(y).
\end{equation}
This norm changes when we multiply $\sigma$ by $\lambda$, but this is all right, we will keep track
of the important constants. 
We now keep the extension operator $E$ the same (i.e., defined with the same formula as in 
\cite{DFM20}). It maps $H$ to $W$ boundedly, and again we will control the norms in
terms of the geometric constants, $\sigma(\d\Omega)$, and $\diam(\Omega)$.

For Theorem \ref{prop:equivalence of norms} still need to prove the equivalence of norms 
in \eqref{2a4}, the better integrability in \eqref{2a7}, and that 
$\Tr : W \mapsto L^2(\d\Omega, \sigma)$ is compact (as announced  in \eqref{2a3}),
and for this some Poincar\'e inequalities will be useful. 
We start with better integrability.

\begin{lemma} \label{th:Poincare1} 
Let $(\Omega,\sigma)$ be a one sided NTA pair of mixed dimension in $\R^n$.  
There exist constants $K \geq 1$ and $p > 2$, that depend only on the geometric constants
for $(\Omega,\sigma)$, such that for $x \in \d\Omega$, $0 < r \leq \diam(\Omega)$, and any 
$u\in W= W^{1,2}(\Omega)$,
		\begin{equation} \label{2n12}
	\left( \fint_{B(x,r)\cap \Omega} |u(y)- \bar u|^{p} \, dy \right)^{1/p} 
	\leq C_c r \left(  \fint_{B(x,Kr)\cap \Omega} |\nabla u(y)|^2 \, dy \right)^{1/2},
		\end{equation}
where $\bar u$ is the average of $u$ on any set $E \subset B(x,2r)\cap \Omega$  
satisfying $|E| \geq c |B(x,r)\cap \Omega|$, and where $C_c>0$ depends only on $c,p$ and on the geometric constants for $(\Omega,\sigma)$.
	\end{lemma}

The story about allowing $E$ to be more general than, say, $E = B(x,r)\cap \Omega$ is
for our later convenience, but is just as easy to prove. 
It is only fair that $C_c$ depends on $c$ (and the dependence is not hard to sort out, comparing with
the main case when $E$ is a corkscrew ball). Also, the (not too) large constant $K$ is needed, 
here and in other similar results below, because
we need to be sure that $E$ is connected inside $\Omega \cap B(0,Kr)$ to the rest of
$B(0,r)$, so that we can use the gradient; since $\Omega$ could look like a hand with long fingers 
and $E$ and $\Omega \cap B(0,r)$ could be near the bulk of the hand but locally separated by fingers, 
$K$ is needed.

The lemma 
follows from \cite[Theorem 5.24]{DFM20}, applied with $p=2$ there and, say,
$D = T_{2Q}$ for a surface cube $Q$ of size $Cr$ that contains $0$. See the definition of
$T_{2Q}$ near (5.3) in \cite{DFM20}, and observe that for $C$ large enough, $T_{2Q}$ contains
$B(0,2r)$ and is contained in $B(0, C^2 r)$ (whence the need for $K$).
\qed

\begin{remark} \label{rk2}
This takes care of \eqref{2a7}.
Note that when $\d\Omega$ is Ahlfors regular,
we said earlier that we are allowed to take $p = \frac{2n}{n-2}$ (any $p < +\infty$ if $n=2$).
Here, if we trust \cite[Remark 5.32]{DFM20} and recall that when $m$ is the Lebesgue measure, the
dimension $d_m$ in \cite[(2.5)]{DFM20} is equal to $n$, 
we are allowed any $p <  \frac{2n}{n-2}$.
\end{remark}

Next we need worry about interactions between $f\in W$ and its trace on $\d\Omega$.
Let us prove a little more than what we need for the moment.

\begin{lemma} \label{th:Poincare3bis}
Let $(\Omega,\sigma)$ be a one sided NTA pair of mixed dimension in $\R^n$.  
There exist a constant $K \geq 1$, that depends only on the geometric constants
for $(\Omega,\sigma)$, and for each $c >0$, a constant $C_c$ that depends only on
the geometric constants for $(\Omega,\sigma)$ and $c$, such that for $x \in \d\Omega$, 
$0 < r \leq K^{-1}\diam(\Omega)$, $u\in W= W^{1,2}(\Omega)$, and $E \subset B(x,2r)\cap \Omega$  
satisfying $|E| \geq c |B(x,r)\cap \Omega|$,
	\begin{equation} \label{2nn}
	\fint_{B(x,r)\cap \d\Omega} |\Tr(u)(y)- \bar u|^{2} \, d\sigma(y) 
	\leq C_c r^2 \fint_{B(x,Kr)\cap \Omega} |\nabla u(z)|^2 \, dz ,
		\end{equation}
where again $\bar u$ is the average of $u$ on $E$.
	\end{lemma}

\begin{proof}
Indeed let $E \subset B(x,2r)$ and $\overline{u}$ be as in the statement,
and consider $v = \varphi (u- \overline u)$,
where $\varphi$ is a normalized cut-off function that is equal to $1$ on $B(x, 2r)$
and is supported in $B(x, 3r)$. Then $v \in W^{1,2}(\Omega)$ too, with
$\nabla v = \varphi \nabla u + (u- \overline u) \nabla\varphi$
(see \cite[Lemma~6.21]{DFM20}) and 
$$
\int_{\Omega} |\nabla v|^2 \leq 2\int_{\Omega \cap B(x,3r)} |\nabla u|^2 + 
C r^{-2}\int_{\Omega \cap B(x,3r)} |u - \overline u|^2
\leq C \int_{\Omega \cap B(x,3Kr)} |\nabla u|^2
$$
by Lemma \ref{th:Poincare1}. Call $g_0 = \Tr(v)$ be the trace of $v$; then $g_0 \in \dot{H}$, with
\begin{equation} \label{2n14}
||g_0||_{\dot{H}}^2 \leq 
C || v ||^2_{\dot{W}}  \leq C \int_{\Omega \cap B(x,Kr)} |\nabla u|^2.
\end{equation}
By construction of the trace, $g_0$ coincides with $\Tr(u) - \overline{u}$ on $\d\Omega \cap B(x,2r)$
and vanishes on $\d\Omega \sm B(x, 3r)$, so \eqref{2n10} yields
\begin{eqnarray} \label{2n15}
||g_0||_{\dot{H}}^2 &\geq& \int_{y\in B(x,2r)}\int_{z\in B(x,Kr) \sm B(x,3r)} 
{\rho(y,|z-y|)^2| g_0(y)-g_0(z)|^2} |y-z|^{-n} \, d\sigma(y)d\sigma(z)
\nonumber\\ 
&\geq& C^{-1}\sigma(\d\Omega \cap B(x,Kr) \sm B(x,3r)) \, [r^{n-1}\sigma(B(x,r)^{-1}]^2 
\int_{y\in B(x,2r)} r^{-n} |g_0(y)|^2 d\sigma(y)
\end{eqnarray}
because $\sigma$ is doubling so $\rho(y,|z-y|) \sim_K r^{n-1}\sigma(B(x,r))^{-1}$
for $y \in \d\Omega \cap B(x,Kr) \sm B(x,3r))$.
Also, we claim that if $K$ is large enough,
\begin{equation} \label{2n16}
\sigma(\d\Omega \cap B(x,Kr) \sm B(x,3r)) \geq C_K^{-1} \sigma(\d\Omega \cap B(x,r)).
\end{equation}
Indeed, if $K \geq 10$ and $\d\Omega$ meets $B(x,Kr/2) \sm B(x,4r)$ at some point $z$, then
$$
\sigma(\d\Omega \cap B(x,Kr) \sm B(x,3r)) \geq \sigma(B(z,r)) 
\geq C_K^{-1} \sigma(B(z, Kr)) \geq C_K^{-1} \sigma(B(x, (K+4)r)),
$$ 
which yields the desired result because $\sigma$ is doubling (and both $x$ and $z$ lie on its support).
If instead  $\d\Omega$ does not meet $B(x,Kr/2) \sm B(x,4r)$, then
$\sigma(B(x,Kr/2) = \sigma(B(x,4r)$, which contradicts the ``large enough dimension condition''
\eqref{1n3} if $c_d (K/8)^d > 1$. So \eqref{2n16} holds, and \eqref{2n14} and \eqref{2n15} yield
\begin{equation} \label{2n17}
\fint_{B(x,2r)} |\Tr(u) - \overline{u}|^2 d\sigma \leq C r^{-n+2} ||g_0||_{\dot{H}}^2
\leq C r^{-n+2} \int_{\Omega \cap B(x,Kr)} |\nabla u|^2.
\end{equation}
Now \eqref{2nn} follows from \eqref{2n17}, which proves the lemma. Parenthetically, we also have
\begin{equation} \label{2n18}
\fint_{B(x,2r)} |\Tr(u)|^2 d\sigma \leq 2|\overline{u}|^2 
+ C r^2 \fint_{\Omega \cap B(x,Kr)} |\nabla u|^2,
\end{equation}
which we will use at a later occasion.
\end{proof}

\ms
We did not state Lemma \ref{th:Poincare3bis} for $r > K^{-1}\diam(\d\Omega)$
because we do not want to confuse issues, 
but here is the global version that we will use for 
Theorem \ref{prop:equivalence of norms}.

\begin{lemma} \label{th:GlobalPoincare3}
Let $(\Omega,\sigma)$ be a bounded one sided NTA pair in $\R^n$.  
There exist a constant $C \geq 1$, that depends only on the geometric constants
for $(\Omega,\sigma)$, such that for $u\in W= W^{1,2}(\Omega)$, 
	\begin{equation} \label{2n19}
	\fint_{\d\Omega} |\Tr(u)(y)- m(u)|^{2} \, d\sigma(y)  + \fint_{\Omega} |u -m(u)|^2
	\leq C \diam(\Omega)^2 \fint_{\Omega} |\nabla u(z)|^2 \, dz ,
		\end{equation}
where $m(u) = \fint_{\Omega} u$ is the average of $u$ on $\Omega$.
	\end{lemma}

\begin{proof}
The estimate for $\fint_{\Omega} |u -m(u)|^2$ is just Lemma \ref{th:Poincare1}, with $p=2$,
so we just need to control the trace.
We cover $\d\Omega$ by less than $C$ balls $B_i = B(x_i, 10^{-1}\diam(\d\Omega))$ 
centered on $\d\Omega$, select a corkscrew ball
$D_i= B(\xi_i, C^{-1}\diam(\d\Omega)) \subset \Omega \cap B_i$ for each $B_i$, 
and set $\overline{u}_i = \fint_{D_i} u$. Notice that 
\begin{equation}\label{2n20}
|\overline{u}_i - m(u)|^2 \leq C \diam(\Omega)^2 \fint_{\Omega} |\nabla u(z)|^2 \, dz
\end{equation}
by Lemma \ref{th:Poincare1} again, and now 
\begin{eqnarray} \label{2n21}
\fint_{\d\Omega} |\Tr(u)(y)- m(u)|^{2} \, d\sigma(y)  
&\leq& \sigma(\d\Omega)^{-1} \sum_i \int_{\d\Omega} |\Tr(u)(y)- m(u)|^{2} d\sigma
\nonumber\\
&\leq& 2\sigma(\d\Omega)^{-1} \sum_i  
\int_{\d\Omega \cap B_i} |\overline{u}_i - m(u)|^2 + |\Tr(u)(y)- \overline{u}_i |^{2} d\sigma(y)
\nonumber\\
&\leq& C  \diam(\Omega)^2 \fint_{\Omega} |\nabla u|^2 
\end{eqnarray}
by  \eqref{2n20}, Lemma \ref{th:Poincare3bis} applied to each $B_i$, and the fact that
we have a bounded number of balls, all of volume and $\sigma$-measure comparable to those of
$\Omega$ and $\d\Omega$.
\end{proof}

We may now return to the proof of Theorem \ref{prop:equivalence of norms}. 
Observe first that \eqref{2a4} follows from this lemma. As for the compactness of 
$\Tr : W \to L^2(\sigma)$ announced in \eqref{2a3}, we need to know that the space $H$ of
\eqref{2a8} embeds compactly in $L^2(\sigma)$. We will spare the details to the reader, 
and instead list the ingredients, which are the same as for the theorem of Rellich-Kondrachov.

Call $B_W$ the unit ball of $W$, and observe that $\Tr(B_W) \subset H \subset L^2(\sigma)$,
which is a complete space. It is enough to show that $\Tr(B_W)$ is completely bounded,
i.e., that for each $\varepsilon > 0$ we only need a finite number of balls $B_j$ of radius $\varepsilon$
in $L^2$ to cover $\Tr(B_W)$. 

Cover $\d\Omega$ by a finite number of small balls $B_j$, 
$1 \leq j \leq N(\varepsilon)$, of radius $\eta(0) > 0$ (very small, to be chosen) and consider the finite 
set $X$ of functions $h \in L^2(\Sigma)$ that are constant on each 
$\widetilde B_j = B_j \sm \big(\cup_{i < j} B_i \big)$, and take values that
are less than $C(\varepsilon)$ (large) and integer multiples of  $\eta$.
We just need to know that for $f \in B_W$ we can find $h \in X$ such that $||h- \Tr(f)||_2 < \varepsilon$.
We take $h$ so that its constant value on $\widetilde B_j$ is within $\tau$ of 
$\fint_{\Omega \cap B_j} g$, for instance, and use our local Poincar\'e estimate 
Lemma \ref{th:Poincare3bis} to control  
\begin{equation}\label{2n22}
\int_{\widetilde B_j} |u-h|^2 \leq \int_{B_j} |u-h|^2 \leq C \tau^2 \int_{\Omega \cap CB_j} |\nabla u|^2.
\end{equation}
We then sum over $j$, use the finite cover property of our covering, and get that
$\int_{\d\Omega} |u-h|^2 \leq C \tau^2 \int_{\Omega} |\nabla u|^2$, with a constant $C$
that does not depend on $\tau$ or $f \in B_W$. Finally we take $\tau = \tau(\varepsilon)$ small enough 
and conclude.

This finally ends our verification of Theorem \ref{prop:equivalence of norms}.
\end{proof}

\subsection{An existence result for Robin boundary data}
	Theorem \ref{prop:equivalence of norms} is just what we need to apply the Stampacchia method to 
	find weak solutions, as in \eqref{weakform}, to our problem \eqref{problem}
$- {\rm div} A \nabla u = 0$ with Robin boundary data $\psi$, 
with $\psi \in H = \Tr(W)$ our space of traces of $W^{1,2}(\Omega)$ functions. 
Recall that, due to the rough boundary, we decided to only consider weak solutions.

	\begin{theorem}\label{th:existencerobin}  [Existence of Robin solutions]
	
Let $(\Omega, \sigma)$ be a one-sided NTA pair of mixed dimension, as in Definition \ref{d:mixed}, 
with $\Omega$ bounded.
Let $A$ be a uniformly elliptic real-matrix valued function. 
For any $a > 0$ and  $\psi\in L^2(\partial \Omega, d\sigma)$ (more generally, for $\psi \in H^*$, 
the dual of $H$ in \eqref{2a8}) there exists a unique function $u\in W^{1,2}(\Omega)$ such that 
\begin{equation}\label{2n23}
\frac{1}{a}\int_{\Omega}A\nabla u\nabla \varphi +\int_{\partial \Omega}\Tr(u) \varphi d\sigma
			=\int_{\partial \Omega}\psi\varphi d\sigma
			\ \text{ for all } \varphi\in C_c^\infty(\mathbb R^n).
\end{equation}
Furthermore, $\|u\|_{W} : = \|u\|_{W^{1,2}(\Omega)} \leq C \|\psi\|_{H^\ast} 
\leq C\|\psi\|_{L^2(\partial \Omega)}.$
	\end{theorem}

Here $Tr(u) \in H \subset L^2(\sigma)$ is the trace of $u \in W$; rapidly, we will take the habit
of  abusing notation and writing $u$ instead of $\Tr(u)$ on $\d\Omega$.
In this paper, we will only apply Theorem~\ref{th:existencerobin} with data $\psi\in L^2(\d\Omega)$, but the proof is the same for $\psi \in H^*$.
In the latter case, it could happen that $\psi$ is not given by a measurable function, and then
$\int_{\partial \Omega}\psi\varphi$ is an abuse of notation for $\psi(\Tr(\varphi))$.

\begin{proof}
We consider the continuous bilinear form $F$ on $W$ given by $F(u,\varphi) 
= \frac{1}{a}\int_{\Omega} A \nabla u \nabla \varphi + \int_{\d\Omega} \Tr(u) \Tr(\varphi)$. 
The form $F$ is accretive in $W$ because of \eqref{2a4}. We also consider the 
linear form $\varphi \mapsto \int_{\d\Omega} \psi \Tr(\varphi) d\sigma$, which is 
continuous on $W$ because our trace operator maps $W$ to the Hilbert space $H$. 
The Stampacchia theorem says that there is a unique $u \in W$ such that 
$F(u,\varphi) = L(\varphi)$ for all $\varphi$. The theorem follows. 
\end{proof}

\subsection{A localization lemma}

For Lemma \ref{th:neumannsub} we will need to localize and apply Theorem~\ref{th:existencerobin}
to intermediate domains that contain a given $\Omega \cap B(y, r)$, $y \in \d\Omega$, 
and have a diameter comparable to $r$ (we will call this a tent domain). Here is the statement.

\begin{lemma}\label{l:tentspaces}
Let $(\Omega, \sigma)$ be a one-sided NTA pair of mixed dimension, as in Definition \ref{d:mixed}.
For each $y \in \partial\Omega$ 
and $0 < r \leq \diam(\Omega)$, we can find a one-sided NTA pair of mixed dimension 
$(T(y,r), \sigma_\star)$, with
\begin{equation} \label{2a9}
B(y, r)\cap \Omega \subset T(y,r) \subset B(y,Kr)\cap \Omega,
\end{equation} 
and the restriction of  $\sigma_\star$ to $\d\Omega$ is $\sigma$. Here $K \geq 1$ and the geometric
constants for $(T(y,r), \sigma_T)$ can be chosen to depend only on $n$ and the geometric constants for
$(\Omega,\sigma)$.
\end{lemma}

We shall call $T(y,r)$ a tent domain. As before, a large $K$ may be needed here, 
because if $\Omega$ looks like a hand with long fingers and $B(y,r)$ meets two fingers, we may need to include faraway pieces of $\Omega$ that connect the two fingers.

The construction of tent domains has been a very useful tool for the study of elliptic PDEs.
Often the geometric situation is more subtle than here, as one wants to hide certain (possibly more complicated)
parts of the boundary where the control is less good, while the tent domain is nicer in some respects and better estimates are available on it. Here we just want to localize so this aspect is simpler.

We'll use a construction of tent domains that originates in \cite[Lemma 3.61]{HM}, and fortunately 
the delicate part of the construction, the verification of the NTA property for $T(y,r)$, was done there
and in subsequent papers. We cannot quote \cite{HM} directly, because it was only interested of domains 
with a boundary of co-dimension $1$, while here we'll have to pay some attention to the measure 
on the boundary because the added pieces of boundary in the construction are $(n-1)$-dimensional.

The construction of \cite{HM} was also used in \cite[Sections 4-5]{MP} and \cite[Section 9.2]{DM23}, 
where it was adapted to higher co-dimension boundaries.
The advantage is that we can do this with Ahlfors regular boundaries $\d\Omega$ of any dimension
$d \neq n-1$, and we get boundaries of mixed dimensions for $T(y,r)$.
In fact one of the reasons for the introduction of the mixed boundary concept was to be able to do this
construction and apply our results to $T(y,r)$.

In the present case we still have to say something because we want to check that we can also start the construction in the general mixed context, not necessarily Ahlfors regular. 
Fortunately the necessary adaptations will be rather easy. 



\begin{proof}
In order to avoid unnecessary (and long) repetitions, 
we shall follow the construction of \cite[Sections 4-5]{MP}, which has the advantage of being quite detailed
and already adapted to fractal boundaries $\d\Omega$ and a mixed dimensional conclusion. 
The verification that the domain  $T = T(y,r)$ constructed there (and, in fact, in \cite{HM} too)
is one-sided NTA is the same in the present situation: the proof only uses the NTA properties of $\Omega$,
and not precisions on the measure $\sigma$.
However we shall remind the reader of how the construction goes, because it  will be needed
when we construct a measure $\sigma_\star$ on $\d T$, and refer to \cite{MP} for details.

There could be a way to avoid the proof that follows by restricting our attention
to Ahlfors regular boundaries (and still using mixed dimensions when we use the localization lemma), but 
a priori we would not recommend this because mixed dimensions are interesting anyway.

First we use a collection $\bD$ of ``dyadic cubes on $\d\Omega$''. The cubes $Q \in \bD$
are in fact merely measurable subsets of $\Gamma = \d\Omega$, with the usual properties
of the usual dyadic cubes in $\R^n$, regarding size and nesting properties in particular.
See \cite[Section 3]{MP} for a review and notations.

Then we choose a scale $k_0$ of cubes, so that the cubes in $\bD_{k_0}$ (cubes of generation $k_0$ in $\bD$)
all have diameters comparable to $r$, but are significantly smaller (to make sure that each cube is contained in $B(x,2r)$, say),
and start from the union $\Delta_0$ of all the cubes of generation $k_0$ that meet $\Gamma \cap \ol B(x,r)$.
Also call $\cF$ the collection of all the other cubes of $\bD_0$ (those that do not meet $\Delta_0$)
and $\bD_{\cF}$, the family of cubes that are contained in $\Delta_0$.
This corresponds to (4.26) in \cite{MP}, but we simplify a bit because all the cubes of
$\cF$ are of the same generation: to make this text readable, we do not claim here that the general 
construction of \cite{MP} goes through in the general mixed context, but only the specific
case needed here). Also we localize slightly differently here, as if all the cubes that compose
$\Delta_0$ were contained in a single cube $Q_0$ of generation $k_0 - 1$. 
The only effect is to make some constants worse at the scales $k_0$ and $k_0 -1$.

We also use a collection $\cW$ of Whitney cubes in $\Omega$. These are (true) dyadic cubes $I$
contained in $\Omega$, with 
\begin{equation}\label{2n25}
4 \diam(I) \leq \dist(I,\Gamma) \leq 40 \diam(I)
\end{equation}
(as in (4.1) in \cite{MP}) that are maximal with this property, and they form a disjoint (except for boundaries) family that covers $\Omega$. 
For each $Q \in \bD_\cF$, we construct the Whitney region $U_Q$ of \cite[(4.24)]{MP}),
which is (the interior of) a finite union of Whitney cubes, all of roughly the same size as $\diam(Q)$, all within
$C \diam(Q)$ from $Q$, and all at distances at least $C^{-1}\diam(Q)$ from $\Gamma$.
They are chosen large so that $U_Q$ connects Whitney cubes near $Q$ to each other through thick paths 
(Harnack chains); for this it is very useful to use Whitney cubes, because they have a simple shape,
are locally boundedly many, and are well connected as soon as they are connected (i.e., through the full interior of a $(n-1)$-dimensional face). 

Then we let $T = \Omega_\cF$ be the interior of the union of all the $U_Q$, $Q \in \bD_\cF$.
The verification that $T $ is a one-sided NTA domain that satisfies \eqref{2a9} is done in \cite{MP}
(see Propositions~5.3 and 5.56 there), and does not use $\sigma$. Maybe our (unneeded) simplification here
where we say that all the cubes that compose $\Delta_0$ have just one parent $Q_0 = \Delta_0$ makes the corkscrew and 
Harnack chain constants a little bigger than needed, but this does not matter.

The boundary $\d T$ can be written as the disjoint union $\d T = \Gamma_0 \cup \Sigma$, where 
\begin{equation}\label{2n26}
\Gamma_0 = \Gamma \cap \d T \ \text{ and } 
\Sigma = \d T \sm \Gamma = \bigcup_{F} F
\end{equation}
is a union of some $(n-1)$-faces of Whitney cubes $I$, where $I$ runs along a subcollection of all the faces 
of Whitney cubes that compose the 
$U_Q$, $Q \in \bD_\cF$. Of course not all the faces occur because two adjacent cubes $I$ and $J$
often lie in the collection of Whitney cubes that compose $T$, and then their common face is not in the list. 
Note that the union in \eqref{2n26} is locally finite, because by \eqref{2n25}, $\delta(z) := \dist(z, \Gamma)$ 
is roughly constant on each face $F$ and its neighbors.
It will be useful to know that 
\begin{equation}\label{2m27}
\dist(z,\Gamma) \leq \dist(z, \Gamma_0) \leq C \dist(z,\Gamma)
\ \text{ for } z \in \Sigma.
\end{equation}
The first part is obvious, and for the second part we observe that if $z \in \Sigma$, and $F$ is one of the faces
that contain $z$, and $Q \in \bD_{\cF}$ is a pseudocube in $\Gamma$ such that $F$ is a face of one of
the  Whitney cube that compose $U_Q$, then $\dist(z, \Gamma) \sim \diam(F) \sim \diam(Q)$ 
and $\dist(F, \Gamma_0) \leq  \dist(F, Q) \leq C \diam(Q)$
by construction, and in particular the fact that each $Q \in \bD_\cF$ is contained in $\Gamma_0$; in fact,
(4.36) in \cite{MP} even says that $\Gamma_0$ is roughly equal to the union $\Delta_0$ of the 
cubes of $\bD_{\cF}$). More precisely, it says that
\begin{equation}\label{2m28}
\Gamma \sm \cup_{Q \in \cF} Q \subset \Gamma_0 = \Gamma \cap \d T \
\subset \Gamma \sm \cup_{Q \in \cF} \,  {\rm int}(Q),
\end{equation}
where the interiors are taken in $\Gamma$. 
The difference between the cubes of $\bD$ and their interior ${\rm int}(Q)$ is negligible, because
$\sigma(\ol Q \sm {\rm int}(Q)) = 0$ for all pseudocubes, by the small boundary condition.

Now we need to construct a measure $\sigma_\star$ on $\d T$, and here we need to modify the 
definition of (5.1) in \cite{MP} because maybe $\sigma$ is no longer Ahlfors regular. 
We take
\begin{equation}\label{2n29}
\sigma_\star = \sigma|_{\Gamma_0} 
+ \frac{\sigma(B(z,2\delta(z))}{\delta(z)^{n-1}} \, \mathcal H^{n-1}|_{\Sigma} 
= \sigma|_{\d T \cap \d\Omega} 
+ \frac{\sigma(B(z,2\delta(z))}{\delta(z)^{n-1}} \, \mathcal H^{n-1}|_{\Sigma} \, ,
\end{equation}
where we recall that $\delta(z) := \dist(z, \Gamma)$.
That is, we keep $\sigma$ as it was on $\Gamma_0$, and on the faces that compose $\Sigma$,
 we use the Hausdorff measure, multiplied by a weight which will help $\sigma_\star$ to be doubling.
 Notice that if $F$ is one of those faces, the weight
 $w(z) = \frac{\sigma(B(z,2\delta(z))}{\delta(z)^{n-1}}$ is roughly constant on $F$, and
 \begin{equation} \label{2n30}
\sigma_\star(F) \sim \sigma(B(z,2\delta(z)).
\end{equation}
In (5.1) in \cite{MP}, the weight was $\delta(z)^{d+1-n}$, which was roughly equivalent because
$\d\Omega$ is $d$-Ahlfors regular there.

Now we need to check that $\sigma_\star$ is doubling and satisfies \eqref{1n3}, and so we evaluate
$\sigma_\star(B(w,t))$ when $w \in \d T$ and $0 < t \leq C \diam(T) \sim r$. We claim that 
\begin{equation} \label{2n31}
\sigma_\star(B(w,t)) \sim \frac{\sigma(B(w,2\delta(w))}{\delta(w)^{n-1}} \, t^{n-1}
\ \text{ when  $t < \delta(w)/2$}
\end{equation}
(which forces $w \in \Sigma$ because $\delta(w) = 0$ on $\Gamma$) and 
\begin{equation} \label{2n32}
\sigma_\star(B(w,t)) \sim \sigma(B(w, 3t)) 
\ \text{ when $t \geq \delta(w)/2$.}
\end{equation}

The first case is the easiest, because then $B(w,t)$ does not meet $\Gamma = \d\Omega$, 
and even meets only a bounded number of faces $F$ as above, all with diameters comparable to $\delta(w)$ 
and total masses comparable to $\sigma(B(w,2\delta(w))$ by \eqref{2n30}.
The estimate follows because each face has a nice shape and $\sigma_\star$ is proportional to 
$\mathcal H^{n-1}$ there. In particular the lower bound comes because if a face $F$ shows up, it is entirely 
contained in $\Sigma$.

We may now focus on \eqref{2n32}, and we start with the upper bound. Call $B = B(w,t)$ and write 
$\sigma_\star(B) \leq \sigma(B) + \sigma_\star(B \cap \Sigma) \leq \sigma(B) 
+ \sum_{F \in \bF} \sigma_\star(B \cap F)$, 
where $\bF$ denotes the set of faces in the description of $\Sigma$ that meet $B$.
The first term $\sigma(B)$ is obviously under control, so we concentrate on the sum.
Let $F \in \bF$ be given, and let $z = z_F$ be any point of $F \cap B$. 
Recall that $\diam(F) \sim \delta(z)$, so $F \subset B(z, C \delta(z)) \subset B(w, Ct)$ because $z \in B$.
Thus, by \eqref{2n30}, it is enough to control $S = \sum_{F \in \bF} \sigma(z_F, C \delta(z))$.
We do not know for sure that these balls have bounded overlap, but fortunately we can use the definition of
$F = \Omega_\cF$ to associate cubes in $\Gamma \sm \Delta_0$ which have a bounded overlap.

Indeed $F$ lies in the boundary $\Sigma$, and this means that at least one of the Whitney cubes $I$ that touch $F$
does not belong to the family that composes $T = \cup_{Q \in \bD_F} U_Q$; since the whole series of Whitney cubes covers
$\Omega$, this means that $I$ lies in one of the $U_Q$, $Q$ contained in a cube of  $\cF$, i.e., for some 
$Q = Q(F)$ that does not meet $\Delta_0$.
In addition, $\diam(Q) \sim \diam(F)$ and $Q$ lies within $C \diam(F)$ from $F$ and also from $\Delta_0$
(this time because $F$ also lies on the boundary of some $U_R$, $R \in \bD$). 
And near the center of this cube $Q$, by the small boundary condition, there is a small surface ball 
$\Delta(Q)$ that lies at distance at least  $C^{-1}\diam(Q)$ from the exterior of $Q$ in $\Gamma$, 
and hence also from all the cubes of $\Delta_0$ (i.e., from the cubes of $\bD_\cF$). 
Altogether, $C^{-1} \diam(Q) \leq \dist(\xi, \Delta_0) \leq C \diam(Q)$ for 
$\xi \in \Delta(Q)$. Thus $\diam(F)$ essentially determines $\dist(\xi, \Delta_0)$ for $\xi \in \Delta(Q(F))$, and because of this the 
$\Delta(Q(F))$, $F \in \bF$, have bounded overlap. Since they are also contained in $B(w,Ct)$, we get that
$$
\sum_{F \in \bF} \sigma_\star(B \cap F) \leq C \sum_{F \in \bF} \sigma(\Delta(Q(F))) \leq C \sigma(B(w,Ct))
\leq C \sigma(B(w,3t))
$$
by \eqref{2n30}, the doubling property of $\sigma$, and bounded overlap (we made sure that $\ol B(w, 2t)$
meets $\Gamma$ so that we can use doubling). This proves the upper bound in \eqref{2n32}.

For the lower bound, first assume that $B(w, t/2)$ meets $\Gamma_0$.
Recall from \eqref{2m28} that this forces $B(w,t/2)$ to meet one of the cubes of $\Delta_0$, say, at
some point $z$. Let $Q$ be the largest cube of $\bD_\cF$ that contains $z$ and is contained in
$B(z,t/2) \subset B(w,t)$. Then $\diam(Q) \geq C^{-1} t$, hence
$\sigma(Q) \geq C^{-1} \sigma(B(z, t/2) \geq C^{-1} B(w,3t)$ because $\sigma$ is doubling, and 
then $\sigma_\star(B(w,t)) \geq \sigma(B(w,t) \cap \Gamma_0) \geq \sigma(Q) \geq C^{-1} B(w,3t)$
(because $Q \subset \Gamma_0$ by \eqref{2m28}), as needed.

So we may assume that $B(w, t/2)$ does not meet $\Gamma_0$, hence that $\dist(w,\Gamma_0) \geq t/2$
and, by \eqref{2m28}, $\dist(w,\Gamma) \geq t/C$. Then $w$ lies in a face $F$ of diameter larger than $t/C$.
In fact, $\diam(F) \sim t$, because for \eqref{2n32} we are in the case when $\delta(w) \leq 2t$
and $\diam(F) \leq C \delta(w)$ if $F$ contains $w$. So in fact  
$\sigma_\star(B(w,t)) \geq \sigma_\star(F \cap B(w,t)) \geq C^{-1} \sigma_\star(F) \geq C^{-1} \sigma(B(w,3t))$
by \eqref{2n30} (or the definition of $\sigma_\star$) and the doubling property.
This completes our proof of  \eqref{2n31} and \eqref{2n32}.

We may return to or main goal, the doubling properties of $\sigma$. 
The fact that $\sigma_\star$ is doubling is a direct consequence of \eqref{2n31} and \eqref{2n32},
because $\sigma$ is itself doubling, and the two formulas give roughly equivalent bounds when 
$t \sim \delta(z)$. For \eqref{1n3}, the argument is nearly as simple: we pick $w \in \d T$
and $0 < s < r  < \diam(T)$, and use \eqref{2n31} and \eqref{2n32} to compare
$\sigma_\star(B(w,r))$ with $\sigma_\star(B(w,s))$. The interesting case is when 
$s < \delta(w) < r$ (the other cases are simpler), and then 
$$
\frac{\sigma_\star(B(w,r))}{\sigma_\star(B(w,s))} \sim s^{n-1}\delta(w)^{1-n} 
\frac{\sigma(B(w,2\delta(w)))}{\sigma(B(w,3t))}
\geq C^{-1} \Big(\frac{s}{\delta(w)}\Big)^{n-1} \Big(\frac{\delta(w)}{t}\Big)^{d}
\geq C^{-1} \Big(\frac{s}{t}\Big)^{\min(n-1,d)}
$$
where $d > n-2$ is as in \eqref{1n3}. That is, $\sigma_\star$ satisfies 
\eqref{1n3} with $\widetilde d = \min(n-1,d)$.
This completes our proof of Lemma~\ref{l:tentspaces}.
	\end{proof}

\begin{remark}		
	The domains $T(y,r)$ have various names in the literature;  here we will call them tent domains.
We will later use the following  {\bf notation} concerning the two parts that compose $\d T(y,r)$: 
for $T(y,r)$ a tent domain as above, we denote 
\begin{equation}\label{2m33}
S(y,r):= \partial T(y,r)\cap \Omega \  \text{ and } \  \Gamma(y,r):= \partial T(y,r)\cap \partial \Omega.
\end{equation}
	\end{remark}

	
\subsection{More Poincar\'e inequalities}  

Here we record simple consequences of the Poincar\'e inequalities above that will be used later.
For the next one, notice that the trace condition is on the intersection of $\Omega$ with a sphere, not on $\d\Omega$.
	
A direct consequence of the above is the following:
	\begin{lemma} \label{lem:Poincare2}
Let $(\Omega, \sigma)$ be a one sided NTA pair of mixed dimension in $\R^n$.
Assume $0 \in \d\Omega$ and $0 < r < \diam(\Omega)/4$. 
There exists $p > 2$ such that for any 
$u\in W^{1,2}(\Omega)$ with $\Tr(u)=0$ on $\partial B(0,r)\cap \Omega$ the following inequality holds: 
		\begin{equation}
\left( \fint_{B(0,r)\cap \Omega} |u|^{p} \, dm \right)^{1/p} 
\leq C r \left(\fint_{B(0,r)\cap \Omega} |\nabla u|^2 \, dm \right)^{1/2}, 
		\end{equation}
where $C>0$ depends only on $p$ and the  geometric constants for $(\Omega, \sigma)$.
	\end{lemma}
	
Note that by Remark \ref{rk2} (and modulo a little bit of checking in the non Ahlfors regular case), 
we can take any $p < \frac{2n}{n-2}$.
	
	\begin{proof}
We consider the function $w$ defined by 
$w=u$ on $B(0,r)\cap \Omega$ and $w=0$ on $\Omega\setminus B(0,r)$. 
Note that $w \in W^{1,2}(\Omega)$,  
with $\nabla w = 0$ on $\Omega \sm B(0,r)$; the only dangerous places 
for the existence of a weak derivative 
would be points of $\Omega\cap \d B(0,r)$, but $w\in W^{1,2}_{loc}$ across those points 
because of our assumption on the trace of $u$. 
We apply Lemma \ref{th:Poincare1} 
to the function $w$ on the set $B(0,Kr)\cap \Omega$ 
and using $E=\big (B(0,2r) \cap \Omega \big)\setminus \big (B(0,r)\cap\Omega\big)$.
This set is not too small because it contains a piece of Harnack chain that connects
a corkscrew point for $B(0,r)$ to a point of $\Omega\sm B(0,2r)$. Also,
we don't need to integrate on a large ball because $\nabla w = 0$ on $\Omega\sm B(0,r)$.
	\end{proof}

We will also use the following direct consequence of Lemma \ref{th:Poincare3bis}.

\begin{lemma}\label{lem:Poincare3}
Let $(\Omega, \sigma)$ be a one sided NTA pair of mixed dimension in $\R^n$.
There exists $K \geq 1$,  and $C \geq 1$,
depending only on the geometric constants for $(\Omega,\sigma)$ 
such that if $0 \in \d\Omega$, $0 < r \leq \diam(\d \Omega)$, and $E \subset \partial \Omega \cap B(0,r)$ 
are such that $\sigma(E) > 0$ 
and $u\in W^{1,2}(\Omega\cap B(0, Kr))$, then
		\begin{equation} \label{e:poincaretrace}
\fint_{B(0,r)\cap \Omega} u(y)^2 \, dy 
	\leq C \frac{\sigma(B(0,r))}{\sigma(E)}  \,   r^2 \fint_{B(0,K r)\cap \Omega} |\nabla u(y)|^2 \, dy
	+ 2   \fint_{E} u^2 d\sigma .
		\end{equation}
\end{lemma}

 \begin{proof}
Let $(\Omega, \sigma)$, $r$, and $u$ be as in the statement. 
Set $\ol u = \fint_{\Omega \cap B(x,r)} u$. Since Lemma \ref{th:Poincare1}, with $p=2$, 
gives a compatible bound on $\fint_{B(0,r)\cap \Omega} |u-\ol u|^2$, we just need to show that
$|\ol u|^2$ is bounded by the right-hand side of \eqref{e:poincaretrace}.
If in addition $r \leq K^{-1}\diam(\Omega)$, we can apply Lemmma \ref{th:Poincare3bis}, 
and \eqref{2nn} says that 
\begin{equation}\label{2m36}
\fint_{B(x,r)\cap \d\Omega} |\Tr(u)(y)- \bar u|^{2} \, d\sigma(y) 
	\leq C r^2 \fint_{B(x,Kr)\cap \Omega} |\nabla u(z)|^2 \, dz .
\end{equation}
In \eqref{e:poincaretrace} we simply  denoted $\Tr(u)$ by $u$. With this notation, since
 $|u|^2 \leq (|u| + |u-\ol u|)^2 \leq 2 u^2 + 2 |u-\ol u|^2$ on $\d\Omega$, we get that
 \begin{eqnarray}\label{2m37}
|\ol u|^2 &=& \fint_E (2 u^2 + 2 |u-\ol u|^2) d\sigma
\leq 2 \frac{\sigma(B(0,r))}{\sigma(E)} \fint_{B(0,r)} |u-\ol u|^2 d\sigma + 2  \fint_{E} u^2 d\sigma,
\end{eqnarray}
and \eqref{e:poincaretrace} follows from \eqref{2m36}. We are left with the case when $K^{-1}\diam(\Omega) \leq r \leq \diam(\Omega)$, 
but in this case we can replace Lemmma \ref{th:Poincare3bis} with Lemma \ref{th:GlobalPoincare3}.
\end{proof}

		\begin{corollary}\label{th:Poincare0}
Let $(\Omega, \sigma)$ be a one sided NTA pair of mixed dimension in $\R^n$, with $\Omega$ bounded.
Let $\Gamma \subset \partial \Omega$ be a measurable set 
such that $\sigma(\Gamma) > 0$. 
Then 
for $u\in W^{1,2}(\Omega)$ such that $\mathrm{Tr}(u)|_{\Gamma} = 0$, 
$$  
\fint_{ \Omega} |u|^2 \leq C \frac{\sigma(\d\Omega)}{\sigma(\Gamma)} \diam(\Omega)^2  
\fint_{\Omega} |\nabla u|^2,
$$
with a constant $C$ that depends only on the geometric constants for $(\Omega, \sigma)$.
\end{corollary}

\begin{proof}
Apply Lemma \ref{lem:Poincare3} to $\Omega$, with $r= \diam(\Omega)$ and $E=\Gamma$, and observe that 
$\int_E u^2 d\sigma$ vanishes.
\end{proof}

\subsection{Partial Neumann Data} Later, in order to build barriers, it will be useful to 
find weak solutions with Neumann data on some part of $\d \Omega$ and Dirichlet 
condition on the rest of $\d \Omega$.

\begin{theorem}\label{th:existencerobin2} [Existence of partial Neumann-data solutions]
	Let $(\Omega, \sigma)$ be a one sided NTA pair of mixed dimension in $\R^n$, with $\Omega$ bounded, and let $A$ be a uniformly elliptic real matrix valued function. 
Given any  $\Sigma \subset \partial \Omega$ such that $\sigma(\Sigma) > 0$ 
and any $\psi\in L^2(\partial \Omega, d\sigma)$, there exists a unique 
function $u\in W^{1,2}(\Omega)$ with $\mathrm{Tr}(u)|_{\Sigma} = 0$ such that
		\begin{equation}\label{2a30}
			\int_{\Omega}A\nabla u\nabla \varphi=\int_{\partial \Omega}\psi\varphi d\sigma
\ \quad \text{ for all $\varphi \in W^{1,2}(\Omega)$ such that $\Tr(\varphi) |_{\Sigma} = 0$.}
		\end{equation}
	\end{theorem}
	
	\begin{proof}
	For the existence and uniqueness of a weak solution, we can proceed as for Theorem~\ref{th:existencerobin},
with the subspace $W_0 = \big\{ u \in W^{1,2}(\Omega) \, ; \, \Tr(u) = 0 \text{ on } \Sigma \big\}$,
which is closed in $W$ because $\Tr : W \to L^2$ is continuous. We use the bilinear form 
$F(u,\varphi) = \int_{\Omega} A \nabla u \nabla \varphi$, which is accretive because
$F(u,u) = \int_\Omega |\nabla u|^2 \simeq ||u||_W^2$ for $u \in W_0$ by Corollary \ref{th:Poincare0}
(and we don't care that the constant depends on $\diam(\Omega)$).
We keep the same continuous linear form $\varphi \mapsto \int_{\d\Omega} \psi \Tr(\varphi) d\sigma$,
and the Stampacchia theorem says that there is a unique $u \in W_0$ such that 
$F(u,\varphi) = L(\varphi)$ for all $\varphi$, as needed. 
	\end{proof}

\begin{remark} \label{rk4}
Here, logically, the values of $\psi$ on $\Sigma$ do not matter. 
As before, we could also have taken any $\psi$ on $\d \Omega$ that defines a continuous linear form on 
the trace of $W_0$, with $\int_{\partial \Omega}\psi\varphi d\sigma$ replaced by
$\psi(\Tr(\varphi))$ for $\varphi \in W_0$.  On nice domains, $\psi=\partial_{\nu_A} u= \nu\cdot A \nabla u$ on $\partial\Omega\setminus\Sigma$, a traditional conormal derivative, and $u=0$ on $\Sigma$.
\end{remark}

	\section{H\"older Continuity at the Boundary for the Neumann problem}\label{s:neumannosc}
	Our goal over the next two sections is to prove boundedness and H\"older continuity for functions with H\"older continuous Robin boundary data. Throughout the section we assume that $0\in\partial\Omega$.
	
In this section we prove a Harnack inequality at the boundary for the Neumann problem, 
which, in domains of this generality, we do not believe to appear elsewhere in the literature. 
	
	\begin{theorem}
		\label{th:neumannharnack}
	Let $(\Omega, \sigma)$ be a one sided NTA pair of mixed dimension in $\R^n$, with $\Omega$ bounded. 
	There exists $K \geq 1$, that depends only on the geometric constants for $(\Omega, \sigma)$,
and $\theta >0$, depending on those constants and the ellipticity constant for $A$, with the following property.
		Let $w\in W^{1,2}(\Omega\cap B(0,Kr))$, $0\leq w$, satisfy 
		\begin{align}
			\label{e:neumann}
			\int_{\Omega\cap B(0,Kr)}A\nabla w\nabla \varphi=0
		\end{align}
		for all $\varphi\in W^{1,2}(\Omega\cap B(0,Kr))$ such that 
$\varphi \equiv 0$ on $\Omega \sm B(0,\rho)$ for some $\rho < Kr$
(see Remark~\ref{rem:tracevszero}). 
 
	Then 
	 \begin{equation}\label{e:neumannharnack}
			\inf_{\Omega \cap B(0,r)} w \geq \theta \sup_{\Omega \cap B(0,r)} w.
		\end{equation} 
	\end{theorem}
 
 \begin{remark}\label{rem:tracevszero}
We have not defined yet 
what it means to be locally a weak solution of Robin or Neumann problem on the domain $\Omega\cap B$, where $B$ is a ball centered on the boundary. 
There are a number of (likely equivalent) options,
but we want to avoid having to discuss the values, or the trace, of $W^{1,2}$ functions on 
$\Omega \cap \d B$, because $\Omega \cap B$ may not be a ``good" domain (for the trace/extension).
Here we decided to choose the one with 
test function space to be $\varphi \in W^{1,2}(\Omega \cap B)$ such that $\varphi \equiv 0$ on $\Omega \sm \mathcal{K}$ for some compact $\mathcal{K}\subset \subset B$.
It may sound a little awkward, but it is a completely secure way to ensure that $\varphi$ vanishes on 
$\Omega \cap \d B$, and we will never need to use the subtle difference between $B(x,Kr)$ in the theorem
and a slightly smaller ball. That is, $\varphi$ looks like something that we may add to $w$, and
want the sum to have the same values as $w$ outside of the ball.

In practice, we often deal with situations where we can actually define traces for our functions $w$ and 
$\varphi$, for instance when $B$ is replaced by one of our tent domains $T(x,R)$, or more brutally when
$\Omega \subset B$, but the definition above will not harm.
 
 Now we could also have required \eqref{e:neumann} only for all smooth functions 
 $\varphi \in C^{\infty}(\R^n)$, with compact support in $B(x,Kr)$, but this would have been equivalent,
 because if $\varphi\in W^{1,2}(\Omega\cap B(0,Kr))$ vanishes
 on $\Omega \sm B(0,\rho)$ for some $\rho < Kr$, we can use the fact that $\Omega$ is an extension
 domain to extend $\varphi$ to $\R^n$, then approximate the extension by smooth functions 
 in $W^{1,2}(\R^n)$, and even cut off the approximations so that they vanish outside of a compact 
 subset of $B(0,Kr)$. 
 Then we can deduce \eqref{e:neumann} for $\varphi$ from the same property for the approximations,
 because $w \in W^{1,2}(\Omega\cap B(0,Kr))$.
 
 \end{remark}

	The first step of the proof is a Moser type estimate for positive Neumann subsolutions:

\begin{lemma}[Moser for Neumann Subsolutions] 

		\label{lem:neumannMoser}  
	Let $w\in W^{1,2}(\Omega\cap B(0,2r))$, $0\leq w$, satisfy
		\begin{align}
			\label{e:neumannsub}
			\int_{\Omega\cap B(0,2r)}A\nabla w\nabla \varphi
			\leq \tau \int_{\partial\Omega\cap B(0,2r)}\varphi d\sigma
		\end{align}
		for all $\varphi\in W^{1,2}(\Omega\cap B(0,2r))$ 
	such that $\varphi\geq 0$ and $\varphi \equiv 0$ on $\Omega \sm B(0,\rho)$ for some $\rho < 2r$
(see Remark \ref{rem:tracevszero}), 	and some constant $\tau \geq 0$.
	Then 
		\begin{align}
			\label{e:moser}
			\sup_{\Omega\cap B(0,r)}w\leq C \tau r^{2-n} \sigma(B(0,r))  
			+ C\left(\fint_{\Omega\cap B(0,2r)} w^2\right)^{\frac{1}{2}},
		\end{align}
		where $C$ depends only on the geometric constants of $\Omega$ and on the ellipticity of $A$.
	\end{lemma}
\begin{remark}\label{rem:robinisneumannsub}
We added a right-hand side in \eqref{e:neumannsub}, which will be convenient in Section \ref{sec:density}. This is the weak formulation of the condition $\partial_\nu u \leq \tau$.
Notice that if $u$ is a non-negative weak solution to the Robin problem in the sense of \eqref{weakform} 
then $u$ is also a weak subsolution to the Neumann problem in the sense of Lemma \ref{lem:neumannMoser}, 
with $\tau = a\sup f^+$. Note that since we used smooth functions in \eqref{weakform} and
$W^{1,2}$ functions in \eqref{e:neumannsub}, we need to know that smooth functions are dense in
$W^{1,2}(\Omega)$, so that the analogue of \eqref{weakform} also holds with  $W^{1,2}$ test functions. 
Fortunately this is true, because $\Omega$ is an extension domain.
So we will be able to apply the lemma to Robin weak solutions.
	\end{remark}

	\begin{proof}
	Note that our estimate is nicely invariant. 
	For one thing, multiplying $\sigma$ by $\lambda > 0$ (and $\tau$ by $\lambda^{-1}$)
	does not change the estimate. As for dilations, if 
	$w$ satisfy \eqref{e:neumannsub} - an easy computation 
	shows that the function $w_r$ defined by $w_r(y)=w(ry)$ satisfies
  $$
  \int_{\Omega\cap B(0,2)}A\nabla w_r\nabla \varphi
			\leq C \tau r^{2-n}  
			\int_{\partial\Omega\cap B(0,2)}\varphi d\sigma^\sharp
  $$
  for test functions $\varphi$ compactly supported in $B(0,2)$, and $\sigma^\sharp$
 the pushforward measure (with the same total mass $\sigma(B(0,2r)$).
  In particular, if we can prove \eqref{e:moser} for $r=1$ and a measure $\sigma$ with total mass comparable to $1$, this scaling argument shows that it holds for all $r$. Inequality \eqref{e:moser} was proved in \cite{Kim} under the assumption that $\Omega$ supports endpoint 
Sobolev and trace embeddings. Kim's proof can be modified to fit our setting, but we choose to give an alternative argument which is (heavily) inspired by the proof of \cite[Lemma~3.1]{LS}, using our Lemma \ref{lem:Poincare2} instead of the Friedrichs inequality and Sobolev embedding.

We will proceed as often done for this type of results and intend to iterate some basic estimate. 
Write $B = B(0,2)$, $D = \Omega \cap B$, and $\Gamma = \partial \Omega\cap B$. 
We may as well assume that $\tau > 0$ (otherwise use $\tau > 0$ and let it tend to $0$ in \eqref{e:moser}).
For $L > 0$ large, define $v = v_L$ by
\begin{equation} \label{3a5}
		v(x) = v_L(x)=\begin{cases}
				w(x) + k \,\,\,\text{ if }0<w(x)<L, \\
				L+ k \quad\quad\text{ otherwise,}
			\end{cases}
		\end{equation}
where the constant $k > 0$ will be chosen later.
Notice that $0 < v \leq L+k$. Let $\varphi \in C_0^\infty (B(0,2))$, $\varphi \geq 0$. 
For every $\alpha\geq 1$, consider  $h = v^{\alpha-1}(w+k)\varphi^2$.
Notice that $h \in W^{1,2}(D)$ because $v$ is bounded, and
$\nabla h = (\alpha-1) [v^{\alpha-2} \nabla v] (w+k) \varphi^2 
+ v^{\alpha-1}  \nabla w \varphi^2
+ 2 \varphi v^{\alpha-1}(w+k) \nabla \varphi$
lies in $L^2(D)$ too, because the first term vanishes when $w > L$.
So $h$ is a valid test function in \eqref{e:neumannsub} and therefore 
\begin{equation} \label{3a6}
\int_D A\nabla w \nabla h \leq \tau \int_\Gamma h
= \tau\int_{\Gamma} v^{\alpha-1}(w+k) \varphi^2.
\end{equation} 
Notice that $A\nabla w \nabla v = A\nabla v \nabla v$ because 
$\nabla v = 0$ almost everywhere on $\{ w > L \}$, so
\begin{equation*}
\begin{aligned}
A\nabla w \nabla h - 2 \varphi v^{\alpha-1}(w+k) A\nabla w \nabla \varphi
	=  (\alpha -1)v^{\alpha-2} (w+k) \varphi^2 A\nabla v \nabla v 
	+  v^{\alpha-1}  \varphi^2 A\nabla w \nabla w
\\  \geq  (\alpha -1)v^{\alpha-1}\varphi^2 A\nabla v \nabla v
+ v^{\alpha-1}  \varphi^2 A\nabla w \nabla w
\end{aligned}
\end{equation*}
because $w+k \geq v$ pointwise. Then by ellipticity and \eqref{3a6},
\begin{multline*}
(\alpha-1) \int_{D} v^{\alpha-1}\varphi^2|\nabla v|^2 +\int_{D}v^{\alpha-1}\varphi^2 |\nabla w|^2 \\
\leq C (\alpha -1)\int_D v^{\alpha-2} (w+k) \varphi^2 A\nabla v \nabla v 
	+  C v^{\alpha-1}  \varphi^2 A\nabla w \nabla w
	\\
\leq C \int_D \varphi v^{\alpha-1}(w+k)  |\nabla w| \, |\nabla \varphi| 
	+ C \tau \int_{\Gamma} v^{\alpha-1}(w+k) \varphi^2.
\end{multline*}
Here and below, it is important that $C$ does not depend on $\alpha$, which will take large values at the end of the argument.

Next $\int_D \varphi v^{\alpha-1}(w+k)  |\nabla w| \, |\nabla \varphi| 
\leq \big\{\int_D \varphi^2 v^{\alpha-1}  |\nabla w|^2  \big\}^{1/2}
 \big\{\int_D v^{\alpha-1}(w+k)^2   |\nabla \varphi|^2 \big\}^{1/2}$; 
using $ab \leq \epsilon a^2 + \frac{1}{\epsilon}b^2$ we can hide the first integral in the
right-hand side and get that
		\begin{equation}\label{3a7}
	(\alpha-1) \int_{D} v^{\alpha-1}\varphi^2|\nabla v|^2 
		+\frac{1}{2}\int_{D}v^{\alpha-1}\varphi^2 |\nabla w|^2  
	\leq C \int_{D} v^{\alpha-1} (w+k)^2 |\nabla \varphi|^2
	+ C \tau \int_{\Gamma} v^{\alpha-1}(w+k)\varphi^2 .
		\end{equation}  
		
		Let $U = v^{\frac{\alpha -1}{2}}(w+k)$ ; 
then we can compute that on $D$,
\begin{equation*}
\begin{aligned}
|\nabla (U\varphi)| &\leq U |\nabla \varphi| + 
\varphi \frac{\alpha-1}{2} \,   v^{\frac{\alpha -3}{2}} (w+k) |\nabla v| 
+ \varphi v^{\frac{\alpha -1}{2}}|\nabla w|
\\
&\leq U |\nabla \varphi| + 
\varphi \frac{\alpha-1}{2} \,   v^{\frac{\alpha -1}{2}} |\nabla v| 
+ \varphi v^{\frac{\alpha -1}{2}}|\nabla w|,
\end{aligned} 
\end{equation*}  
because $w+k \leq v$. Then by \eqref{3a7}
$$\begin{aligned}
\int_{D}|\nabla (U\varphi)|^2 &\leq C \int_{D}U^2 |\nabla \varphi|^2 
+  (\alpha-1)^2 \varphi^2 v^{\alpha-1}|\nabla v|^2 +  \varphi^2 v^{\alpha-1} |\nabla w|^2
	\\
	&\leq  C(\alpha+1)
	\left(\int_{D}U^2 |\nabla \varphi|^2  + \tau \int_{\Gamma} (w+k)^{-1}(U\varphi)^2 d\sigma
	\right).
\end{aligned}$$ 
Recall that $w\geq 0$; we obtain 
	\begin{equation}\label{3a8}
		\int_{D}|\nabla (U\varphi)|^2
	\leq C (\alpha + 1)\int_{D}(U |\nabla \varphi|)^2 
	+ C k^{-1} \tau \int_\Gamma (U\varphi)^2 d\sigma .
		\end{equation}
Next we wish to replace $\int_\Gamma (U\varphi)^2 d\sigma$ with an integral on $D$.
Set $f = U\varphi$; this function lies in $W^{1,2}(D)$ (because $v$ is bounded),
and it is compactly supported in $B$ (because of $\varphi)$; the trace inequality and the Poincar\'e inequality in Lemma \ref{lem:Poincare2} (which we can apply because $f$ is compactly supported in $B$) give 
\begin{equation}\label{3b9}
\fint_\Gamma  f^2 d\sigma \leq C \sigma(B)  \fint_D |\nabla f|^2
\leq C \fint_D |\nabla f|^2
\end{equation}
(recall that we normalized $\sigma(B)$ out). 
We return to \eqref{3a8} and get that 
\begin{equation*}
\int_{D}|\nabla (U\varphi)|^2
	\leq C (\alpha + 1)\int_{D}(U |\nabla \varphi|)^2 
	+ C k^{-1} \tau \int_D |\nabla (U\varphi)|^2 d\sigma .
\end{equation*}
We may now choose $k = C_1 \tau$, where $C_1$ is so large that we can absorb the
last term in the left-hand side, and get that 
\begin{equation}\label{3b12}
\int_{D}|\nabla (U\varphi)|^2 \leq C (\alpha + 1)\int_{D}(U |\nabla \varphi|)^2.
\end{equation}
Using now the Poincar\'e inequality given in Lemma \ref{lem:Poincare2}, 
we obtain that there is some $p>2$ such that:
		\begin{equation*}
	\left ( \fint_{D}|U\varphi|^p\right)^\frac{2}{p}
	\leq C \fint_D |\nabla (U\varphi)|^2 
		\leq C (\alpha+1) \, \fint_{D} U^2 \, |\nabla \varphi|^2,
		\end{equation*}
where $C$ is a constant that now depends on the geometric constants of $\Omega$ and the ellipticity of $A$ (but not $\alpha$).  Again Lemma \ref{lem:Poincare2} tells us that we 
should integrate the gradient on a larger domain, but here $U \varphi$ is compactly supported in $D$.

We continue now as in the standard Moser iteration argument.
Choose an exponent $p > 2$ that works in the Poincar\'e inequality,  set 
$\gamma = \frac{p}{2}$, and choose the successive exponents 
$\alpha_m = 2 \gamma^m -1$ and $\beta_m = 1+ \alpha_m = 2 \gamma^m$, $m \geq 0$. 
Also define the radii $r_m = (1 + 2^{-m}) $, $m \geq 0$, and
set $B_m = B(0, r_m)$, $D_m = \Omega \cap B_m$, and $\Gamma_m = \d \Omega \cap B_m$.
Note that $\alpha_0 = 1$, $D_0 = D$, $\Gamma_0 = \Gamma$, and $r_m > 1$ for all $m$.

Choose a test function $\varphi_m$ compactly supported in $B_m$, such that $\varphi_m = 1$ 
on $B_{m+1}$, $\varphi_m\geq 0$, and $|\nabla \varphi_m| \leq C 2^m $.
Then apply the argument above, but in the slightly smaller ball $B_m$ and
with the exponent $\alpha_m$. The same proof as above, with
$U_m = v^{\frac{\alpha_m-1}{2}} (w+k)$,  yields
\begin{equation*}
\left ( \fint_{D_{m+1}}U_m^p\right)^\frac{2}{p}
	\leq C \left ( \fint_{D_m}|U_m \varphi_m|^p\right)^\frac{2}{p}
	\leq C (\alpha_m+1) 2^{2m} \fint_{D_m} U_m^2.
		\end{equation*}
These functions depend on the large cut-off parameter $L$ from \eqref{3a5}, 
but we may now let $L$ tend to $+\infty$.
Then $U_m$ tends to $(w+k)^{\frac{(\alpha_m+1)}{2}}$ and by
the monotone convergence theorem,
\begin{equation*}
\left ( \int_{D_{m+1}}(w+k)^{(\alpha_m+1) \frac{p}{2}}\right)^\frac{2}{p}
		\leq C (\alpha_m+1) 2^{2m} \int_{D_m} (w+k)^{\alpha_m+1},
		\end{equation*}
	where it is implied that the left-hand side is finite when the right-hand side is finite.
Set $a_m = \Big\{\fint_{D_m} (w+k)^{\beta_m}\Big\}^{1/\beta_m}$ and 
recall that $\beta_m = \alpha_m+1 = 2 \gamma^m$ and 
$\frac{p}{2}(\alpha_m+1) = \beta_{m+1}$; we just proved that
\begin{equation}\label{3b13}
a_{m+1}^{\beta_m} \leq C (\alpha_m+1) 2^{2m} a_m^{\beta_m}.
\end{equation}
Since we start with $a_0 =\Big\{ \fint_D (w+k)^2 \Big\}^{1/2} < +\infty$,
all these norms are finite. We take logarithms and get that 
\begin{equation}\label{3b14}
\log(a_{m+1}) \leq \log(a_m) +  \frac{1}{\beta_m} \big[\log(C) + \log(\beta_m) + 2m \log(2) \big].
\end{equation}
Notice that $\sum_{m} \frac{1}{\beta_m} \big[C + \log(\beta_m) + 2m \log(2) \big] \leq C$ because 
$\beta_m = 2 \gamma^m$, so $\limsup_{m \to +\infty} \log(a_m) \leq \log(a_0) + C$ and 
\begin{equation}\label{3b15}
||w+k||_{L^{\infty}(\Omega \cap B(0,r))} \leq \limsup_{m \to +\infty} a_m \leq C a_0 
= C\Big\{ \fint_D (w+k)^2 \Big\}^{1/2} \leq  C \Big\{ \fint_D (w+k)^2 \Big\}^{1/2} + C k.
\end{equation}
But we took $k = C_1 \tau$, so \eqref{e:moser} and Lemma \ref{lem:neumannMoser} follow from this.
\end{proof}

	To prove our Harnack inequality it will be useful to show that for positive $A$-harmonic functions with non-positive Neumann data, the value of the function at a corkscrew point is comparable to the supremum of the function on the whole ball.  We remind the reader that this proof is substantially different from the existing proofs in the Dirichlet boundary case. 
	\begin{lemma}\label{l:CSmax}
		There exist constants $K$ (depending the geometric constants of $ \Omega$) and $C$ (depending also on the ellipticity of $A$) such that if $u\geq 0$ satisfies $-\mathrm{div}\left(A\nabla u\right)=0$ in $\Omega\cap B(0,Kr)$ and 
		\begin{align}\label{3b16}
			\int_{\Omega\cap B(0,Kr)}A\nabla u\nabla \varphi\leq 0
		\end{align}
for all $\varphi\in W^{1,2}(\Omega\cap B(0,Kr))$ such that $\varphi \geq 0$ and 
$\varphi \equiv 0$ on $\Omega \sm B(0,\rho)$ for some $\rho < Kr$
(see Remark \ref{rem:tracevszero}), then \begin{equation}\label{e:CSmax}
			u(Y) \leq Cu(A_{r/4}(0)),\qquad \text{ for all } Y \in \Omega \cap B(0,r/4).
		\end{equation}
	\end{lemma}
	
	Here $A_{r/4}(0)$ is a corkscrew point for $B(0, r/4)$, i.e., a point of $\Omega \cap B(0,r/4)$
such that $B(A_{r/4}(0), C^{-1}r) \subset \Omega$, as in \ref{CC1}; the precise choice does not matter, since the values of $u$ at all such points are comparable by the (interior) Harnack inequality.

The reader may be surprised that $K$ and $C$ also depend on $\sigma$, through the doubling and
asymptotic dimension $> n-2$ estimates, even though the estimate we prove do not involve $\sigma$.
But the existence of an appropriate $\sigma$ is an information that we believe we use earlier. This will happen again, although often $\sigma$ appears in the weak definition of the Robin boundary condition.

	\begin{proof}
	By dilating and scaling as above, 
we can assume that $r= 1$ and $u(A_{1/4}(0)) = 1$. 
We want to prove:
		\begin{equation}\label{e:aux1}
			u(X) \leq M \text{ for } X \in \Omega \cap B(0,1/4).
		\end{equation}
		
		Here $M$ is a potentially large constant whose value may vary from line to line but depends only on the geometric constants of $\Omega$. By Lemma \ref{lem:neumannMoser} with $\tau=0$, we only need to prove that 
		\begin{equation} \label{e:aux2}
			\fint_{B(0,1/2)} |u(X)|^2 \leq M^2.
		\end{equation}
		
		The claim below follows easily by repeated applications of the (interior) Harnack inequality, using the Harnack chain condition \ref{CC2}. 
		\begin{claim*}
			With the assumptions above, 
			\begin{equation}\label{e:aux3}
	c_1 \delta(X)^{A} \leq u(X) \leq c_2 \delta(X)^{-A}  \text{ for } X \in \Omega \cap B(0,1).
			\end{equation}
Here and below, $\delta(X) = dist(X, \partial\Omega)$, and $A \geq 1$ is a large constant that depends on the geometry.
		\end{claim*}
		
		We want to prove \eqref{e:aux1} by contradiction. To this end, consider the quantity
		\begin{equation} \label{e:mass}
			m(r) = r^{-n} \int_{B(0,r) \cap \Omega} u^2.
		\end{equation} 
		We will show that if it is very large for $r=1/2$, say, then it is much larger for slightly bigger $r$, so that, after some iterations, $\int_{B(0,1)} u^2$ will be infinite, contradicting $u\in W^{1,2}$. We fix $M$ large (to be chosen later),  
let $1/2 \leq r \leq 1$, and assume $m(r) \geq C_0 M^2$, where $C_0$ will chosen later, 
depending only on $n$. 
		We first consider the region 
		\begin{equation} \label{e:largeregion}
			E(r) = \big\{ X \in \Omega \cap B(0,r) \, ; \, u(X) \geq M \big\}.
		\end{equation} 
		We take $X \in E(r)$; by \eqref{e:aux3},  $M\leq u(X) \leq c_2 \delta(X)^{-A} $, hence
		\begin{equation} \label{2c6}
			\delta(X) \leq c_2^{-1/A} M^{-1/A} =: C_1 M^{-1/A}.
		\end{equation}
		That is to say that $u$ can be very large only very close to the boundary, at a scale roughly $M^{-1/A}$. We now choose $\rho =  C_2 M^{-1/A}$, with $C_2$ large enough (to be chosen soon), and consider a corkscrew point $\xi_\rho = \xi_\rho(X)$ for $B(X,\rho)$. We notice that
		$u(z) \leq M/2$ for $z\in B(\xi_\rho, \rho/C)$, again by \eqref{e:aux3}. Because of this,
		\begin{equation} \label{2c7}
			\begin{split}
				\int_{B(X,\rho)} \int_{B(X,\rho)} |u(z)-u(x)|^2 dxdz
				\geq \int_{E(r) \cap B(X,\rho)} \int_{B(\xi_\rho,\rho/C)} |u(z)-u(x)|^2 \,dxdz
				\\
				\geq C^{-1} \rho^{n} \int_{E(r) \cap B(X,\rho)} u^2\,dx
			\end{split}
		\end{equation}
		because $|u(z)-u(x)|^2 \geq (|u(x)|-M/2)^2 \geq u^2(x)/4$, $x\in E(r) \cap B(X,\rho)$,
		by the definition \eqref{e:largeregion}  of $E(r)$. And then, by this and an application of the Poincar\'e inequality Lemma~\ref{th:Poincare1} on our domain (sketch of the proof: connect $z$ and $x$ by a Harnack chain, write $u(z)-u(x)$ as a telescoping sum of differences of averages, use Lemma~\ref{th:Poincare1} on each term, integrate over $x$ and $z$ and sum up),
		\begin{equation} \label{2c8}
			\int_{E(r) \cap B(X,\rho)}  u^2 
			\leq C \rho^{-n} \int_{B(X,\rho)} \int_{B(X,\rho)} |u(z)-u(x)|^2 dxdz
			\leq C \rho^2 \int_{B(X,C\rho)} |\nabla u|^2.
		\end{equation}
		Now we cover $E(r) \subset B(0,r)$ by balls of radius $\rho$, with a bounded covering number, 
		and find that
		\begin{equation}\label{e:aux4}
			\int_{E(r)} |u|^2 \leq C \rho^{2} \int_{B(0,r+C\rho)} |\nabla u|^2.
		\end{equation}
		We now need a Caccioppoli inequality. Pick $\rho_1>C\rho$, to be determined later.
		We claim that 
		\begin{equation}\label{e:aux5}
			\int_{B(0,r+C\rho)} |\nabla u|^2 \leq C \rho_1^{-2} \int_{B(0,r+C\rho + \rho_1)} u^2,
		\end{equation}
		where as usual $C$ depend only on the geometry. As usual for Caccioppoli-type arguments, to see this, we take a cut-off function $\theta$,
		with compact support in $B_1 = B(0, r+C\rho + \rho_1)$, so that
		$0 \leq \theta \leq 1$ everywhere, $\theta = 1$ on $B_0 = B(0, r+C\rho)$, and
		$|\nabla \theta| \leq 2 \rho_1^{-1}$. Then we apply \eqref{3b16} to the function
		$\varphi = \theta^2 u$; we get that
		\begin{align*}
			\int_{\Omega\cap B_1}A\nabla u [\theta^2 \nabla u + 2 \theta u \nabla \varphi] \leq 0,
		\end{align*}
		hence by ellipticity 
		\begin{align*}
		\int_{\Omega\cap B_1} \theta^2 |\nabla u|^2 
		& \leq C\int_{\Omega\cap B_1}\theta^2 A\nabla u \nabla u 
		\leq 2C \int_{\Omega\cap B_1} \theta u |\nabla u| |\nabla \varphi| 
		\leq 4C \rho_1^{-1} \int_{\Omega \cap B_1} \theta u |\nabla u|
		\\
		& \leq 4C \rho_1^{-1} \Big\{ \int_{\Omega \cap B_1 \sm B_0} u^2 \Big\}^{1/2}
		\Big\{ \int_{\Omega \cap B_1} \theta^2 |\nabla u|^2  \Big\}^{1/2};
		\end{align*}
		\eqref{e:aux5} follows after simplifying, since $\theta = 1$ on $B_0$.
		Combining with the previous inequality, we obtain:
		\begin{equation}\label{e:aux6}
			\int_{E(r)}  u^2 \leq C \rho^2 \rho_1^{-2} \int_{B(0,r+2 \rho_1)} u^2.
		\end{equation}
		Recall that $m(r) \geq C_0 M^2$,
	so $\int_{B(0,r)}  u^2 \geq C_0 r^{n} M^2$. The contribution of $B(0,r) \sm E(r)$ to this integral
		is at most $M^2 |B(0,r)| = c_n r^n M^2$, where $c_n$ is the volume of the unit ball. We 
		choose $C_0 = 2c_n$. Then we also get that $\int_{E(r)} u^2 \geq \frac12 \int_{B(0,r)} u^2$, 
		and \eqref{e:aux6} yields
		\begin{equation}\label{e:aux7}
			\int_{B(0,r)}  u^2 \leq C \rho^2 \rho_1^{-2} \int_{B(0,r+2 \rho_1)} u^2.
		\end{equation}
		Recall that $\rho =  C_2 M^{-1/A}$, and we now choose
	$\rho_1 =  M^{-1/(2A)} \gg \rho$ (if $M$ is large enough) and observe that $r\approx r+2\rho_1$. This yields
		\begin{equation}\label{e:aux8}
	m(r+2\rho_1)  \geq C^{-1} \rho^{-2} \rho_1^{2} \,\frac{r^2}{(r+\rho_1)^2} \, m(r)
			\geq \widetilde{C}^{-1} M^{1/A} m(r) \geq 2 m(r)
		\end{equation}
		if $M$ is large enough. We started with the assumption that 
		$m(1/2) \geq C_0 M^2$, say, and we want to apply the computation above repeatedly.
		This leads to choosing $r_0 = 1/2$, and by induction
		$r_{k+1} = r_{k} (1+2 M_k^{-1/(2A)})$, where $M_k = 2^k M$.
		The estimate above  works as long as we keep $r_k < 1$. 
		Since $r_{k+1} \leq r_k + 2 M_k^{-1/(2A)} r_k
		\leq r_k + 2 M_k^{-1/(2A)} \leq  r_k + 2 M^{-1/(2A)} 2^{-k/(2A)}$, this is ensured by taking $M$ large enough (depending on the value of $\sum_k 2^{-k/(2A)}$). Iterating \eqref{e:aux8}, we finally obtain $m(1)=+\infty$, which is the contradiction we were seeking.
	\end{proof}

 For the sake of being self-contained, we remind the reader that the positive part of a weak Robin or Neumann subsolution is also a weak Neumann subsolution, even for boundary value problems in rough geometry.

	\begin{lemma}\label{l:subsolformax}
Let $B = B(0,r)$ be centered on $\d\Omega$ and let $w \in W^{1,2}(\Omega \cap B)$
a Robin subsolution in the sense that 
\begin{align}
	\label{3c31}
\int_{\Omega\cap B}A\nabla w\nabla \varphi 
+ \beta \int_{\partial \Omega \cap B} w\varphi d\sigma \leq 0 
		\end{align}
for all $\varphi\in W^{1,2}(\Omega\cap B)$ such that $\varphi\geq 0$ on $B$ and $\varphi \equiv 0$ 
on $\Omega \sm \mathcal{K}$ 
for some compact $\mathcal{K}\subset \subset B$ (see Remark \ref{rem:tracevszero}).
	If $\beta \geq 0$, then the positive part, $w^+$,  is also in $W^{1,2}(\Omega\cap B)$ and satisfies $$\int_{\Omega\cap B} A\nabla w^+ \nabla \varphi \leq 0,$$ for all $\varphi$ as above. 
\end{lemma}

\begin{proof}
 Define $F_\epsilon(t) := \sqrt{t^2 + \epsilon^2} - \epsilon$ when $t\geq 0$ and $F_\epsilon(t) = 0$ when $t < 0$. Note that $F_\epsilon(t) \rightarrow t^+$ pointwise as $\epsilon \downarrow 0$ and so $\nabla F_\epsilon(w) \rightarrow \nabla w^+$ in $L^2$ by the dominated convergence theorem (we are using that $F_\epsilon'(t) < 1$ and that $F'_\epsilon(t) \uparrow 1\chi_{(0, \infty)}(t)$ as $\epsilon\downarrow 0$). 

So if we can show that $$\int_{\Omega\cap B}A\nabla F_\epsilon(w)\nabla \varphi \leq 0,$$ for all $\epsilon > 0$ and all $\varphi$ as above then we are done.  If $\varphi$ is a valid test function, then so is $\varphi^M = \min\{M, \varphi\}$ for any $M \geq 0$. Furthermore $\nabla \varphi^M \rightarrow \nabla \varphi$ in $L^2$ as $M \uparrow \infty$ again by the dominated convergence theorem. So it suffices to show that for any $M > 0$ and $\epsilon > 0$ we have  $$\int_{\Omega\cap B}A\nabla F_\epsilon(w)\nabla \varphi^M \leq 0.$$

We compute $$\begin{aligned} \int_{\Omega \cap B} A\nabla F_\epsilon(w)\nabla \varphi^M =& \int_{\Omega \cap B} F_\epsilon'(w) A\nabla w\nabla \varphi^M\\ =& \int_{\Omega\cap B} A\nabla w \nabla (F_\epsilon'(w)\varphi^M) - \int_{\Omega \cap B} \varphi^M F_\epsilon''(w)A\nabla w \nabla w\\
 \leq& -\beta \int_{\partial \Omega \cap B} wF_\epsilon'(w) \varphi^M \leq  0.\end{aligned}$$
We need to justify these last two inequalities a bit more. First note that $F_\epsilon'(w)\varphi^M\in W^{1,2}$ because $F_\epsilon'(w), \varphi^M \in W^{1,2}\cap L^\infty$. Next observe that $F_\epsilon'(w)\varphi^M \geq 0$ (recall that $F' \geq 0$), and so $F_\epsilon'(w)\varphi^M$ is a valid test function. The second integral on the penultimate line is $\geq 0$ because of ellipticity and the fact that $F_\epsilon''(w) \geq 0$.  Finally, the last integral is $\geq 0$ because $\beta\geq 0$ and $F'_\epsilon(w) = 0$ whenever $w < 0$.
\end{proof}

	The last ingredient for the Harnack inequality is a density lemma for Neumann super-solutions. Before proving the Lemma it is worth reminding ourselves that if $F: [a,b]\rightarrow \mathbb R$ 
is $C^1$ and $u\in W^{1,2}(\Omega)$, such that $a \leq u \leq b$
almost everywhere, then $F \circ u \in W^{1,2}(\Omega)$. 
This fact will also be used in subsequent proofs.
	
	\begin{lemma}[Density Lemma for Neumann supersolutions]
		\label{lem:neummdensity}     
	There exist a constant $K > 1$, that depends on the geometric constants
	for $(\Omega,\sigma)$ and, for each $\eta \in (0,1)$, a constant $c_\eta > 0$, that also depends on the ellipticity constant for $A$, with the following property.
	
	 Let $w\in W^{1,2}(\Omega\cap B(0,Kr))$ be a 
	 non-negative function satisfying
		\begin{align*}
			\int_{\Omega\cap B(0,Kr)}A\nabla w\nabla \varphi \geq 0
		\end{align*}
		for all $0 \leq \varphi\in W^{1,2}(\Omega\cap B(0,Kr))$ 
		such that $\varphi \equiv 0$ on $\Omega \sm B(0,\rho)$ for some $\rho < Kr$
(see Remark \ref{rem:tracevszero}). 
		If  
		\begin{equation*}
			|\{x\in \Omega\cap B(0,2 r)\,:\, w(x)\geq 1\}|\geq \eta|\Omega \cap B(0, 2r)|,
		\end{equation*} 
		then,
		\begin{align}
			\inf_{\Omega\cap B(0,r)}w\geq c_\eta.
		\end{align}
	\end{lemma}

	\begin{proof}
		We will prove this for $w_{\delta}=w+\delta$ with $\delta>0$, which obviously satisfies the assumptions, and with constants that do not depend on $\delta$; the result for $w$ will follow at once. 
	
		We want to show that $v^+_{\delta}:= (\log(w_{\delta}))^-$ is a Neumann subsolution. 
		By Lemma \ref{l:subsolformax}, it suffices to show that 
		$v_\delta:= -(\log(w_{\delta}))$ is a subsolution, 
	and indeed, this is straightforward to check: 
		\begin{align*}
			\int_{\Omega\cap B(0,Kr)}A\nabla v_\delta \nabla \varphi&=\int_{\Omega\cap B(0,Kr)}-\frac{1}{w_\delta}A\nabla w_\delta\nabla \varphi \\
			&= -\int_{\Omega\cap B(0,Kr)}A\nabla w_\delta \nabla \left(\frac{1}{w_\delta}\varphi\right)-\int_{\Omega\cap B(0,Kr)}\left(\frac{1}{w_\delta^2}\varphi\right) A\nabla w_\delta \nabla w_\delta.
		\end{align*}
	By hypothesis, if $\varphi \geq 0$ we have that $\int_{\Omega\cap B(0,Kr)}A\nabla w_\delta \nabla \big(\frac{1}{w_\delta}\varphi\big)\geq 0$, so we conclude that 
		\begin{equation*}
			\int_{\Omega\cap B(0,Kr)}A\nabla v^+_{\delta}\nabla \varphi\leq 0
		\end{equation*}
		for all $\varphi\in W^{1,2}(\Omega\cap B(0,Kr))$ such that $\varphi \equiv 0$ on $\mathcal{K}^c$. 
		
		We apply Lemma~\ref{lem:neumannMoser} (the Moser inequality) with $\tau=0$ and obtain:
		\begin{align*}
			\sup_{\Omega\cap B(0,r)}v^+_{\delta}\leq C\left(\fint_{\Omega\cap B(0,2r)}(v^+_{\delta})^2\right)^{\frac{1}{2}}.
		\end{align*}
		We want to bound the right hand side of the above equation. By hypothesis we know that 
		\begin{equation*}
	|\{x\in \Omega\cap B(0,2r)\,:\, v^+_{\delta}(x)=0\}|\geq \eta|\Omega\cap B(0,2r)|, 
		\end{equation*}
		so by the Poincar\'e inequality in Theorem \ref{th:Poincare1}
		we have that   
		\begin{equation}\label{eq:upperboundsup}
			\sup_{\Omega\cap B(0,r)}v^+_{\delta}
		\leq C\left(\fint_{\Omega\cap B(0,2r)}(v^+_{\delta})^2\right)^{\frac{1}{2}}
	\leq Cr\left(\fint_{\Omega\cap B(0,Kr)}|\nabla v^+_{\delta}|^2\right)^{\frac{1}{2}}.
		\end{equation}
		We recall that the constant $C$ in the above equation depends on the geometry of the domain and on $\eta$. 
		
		We use the test function $\phi=\frac{\xi^2}{w_\delta}$, where $\xi$ will be chosen soon
		(compactly supported in $B(0,4r)$), and we compute 
		\begin{equation}\label{extra}
			0\leq \int_{\Omega\cap B(0,Kr)}A\nabla w_\delta\nabla \phi=-\int_{\Omega\cap B(0,Kr)} \xi^2\frac{A\nabla w_\delta \nabla w_\delta}{w_\delta^2}+2\int_{\Omega\cap B(0,Kr)} \xi\frac{A\nabla w_\delta\nabla \xi}{w_\delta}.
		\end{equation}
		By ellipticity and the Cauchy-Schwarz inequality, 
	\begin{equation*}\begin{aligned}
	\fint_{\Omega\cap B(0,Kr)}\xi^2|\nabla v_{\delta}|^2 
	& = \fint_{\Omega\cap B(0,Kr)}\xi^2 \, \frac{|\nabla w_\delta |^2}{w_\delta^2}
	\leq C \fint_{\Omega\cap B(0,Kr)}\xi^2\frac{A\nabla w_\delta \nabla w_\delta}{w_\delta^2}
	\\&\leq C \int_{\Omega\cap B(0,Kr)} \xi \, \frac{|\nabla w_\delta|}{w_\delta} \, |\nabla \xi|
	 \leq C \left(\fint_{\Omega\cap B(0,Kr)}\xi^2|\nabla v_{\delta}|^2\right)^{\frac{1}{2}}
	\left(\fint_{\Omega\cap B(0,Kr)}|\nabla \xi|^2\right)^{\frac{1}{2}},
	\end{aligned}\end{equation*}
	where we used \eqref{extra} in the second inequality.
		We choose $\xi\in \mathcal{C}^1_0(B(0,4r))$ such that $\xi=1$ on $B(0,2r)$ and 
		$|\nabla \xi|\leq \frac{C}{r}$, and we obtain
		\begin{equation*}
			\left(\fint_{\Omega\cap B(0,2r)}|\nabla v^+_{\delta}|^2\right)^{\frac{1}{2}} \leq \left(\fint_{\Omega\cap B(0,2r)}|\nabla v_{\delta}|^2\right)^{\frac{1}{2}}\leq \frac{C}{r}.
		\end{equation*}
		Hence, from \eqref{eq:upperboundsup}, $\displaystyle \sup_{\Omega\cap B(0,r)}v^+_{\delta}\leq C$,
	and so we have 	$\displaystyle \inf_{\Omega\cap B(0,r)}w_{\delta}\geq e^{-C}$, 
		where $C$ does not depend on $\delta$.
	\end{proof}

	With the aid of the density estimate we are able to prove oscillation estimates, which may be of independent interest.

	\begin{lemma}[Oscillation Estimate for Neumann solutions] 
		\label{lem:neummOsc} 
There exist a constant $K > 1$, that depends on the geometric constants for $(\Omega, \sigma)$,
and a constant $\gamma \in (0,1)$, that depends also on the ellipticity constant for $A$, so that
if  $w\in W^{1,2}(\Omega\cap B(0,Kr))$ is non-negative, bounded, and such that 
		\begin{align}
			\label{neumann}
			\int_{\Omega\cap B(0,Kr)}A\nabla w\nabla \varphi=0
		\end{align}
for all $\varphi\in W^{1,2}(\Omega\cap B(0,Kr))$ such that $\varphi \equiv 0$ on $\Omega \sm B(0,\rho)$ for some $\rho < Kr$
(see Remark \ref{rem:tracevszero}), then 
		\begin{equation}
			\osc_{\Omega\cap B(0,r)}w\leq \gamma \osc_{\Omega\cap B(0,2r)}w.
		\end{equation}
	\end{lemma}

	\begin{proof} 
		By assumption, we know that $w$ is bounded in $\Omega\cap B(0,Kr)$. We define the numbers:
		\begin{equation*}
			\alpha_2=\sup_{\Omega\cap B(0,2r)}w\,\,\, \text{ and  }\,\,\, \beta_2=\inf_{\Omega\cap B(0,2r)}w
		\end{equation*}
		and also 
		\begin{equation*}
\alpha_1=\sup_{\Omega\cap B(0,r)}w\,\,\, \text{ and  }\,\,\, \beta_1=\inf_{\Omega\cap B(0,r)}w.
		\end{equation*}
		The two functions
		\begin{equation*}
\frac{w-\beta_2}{\alpha_2-\beta_2}\,\,\, \text{ and  }\,\,\, \frac{\alpha_2-w}{\alpha_2-\beta_2}
		\end{equation*}
		are positive solutions in $\Omega\cap B(0,2r)$. We need to distinguish two cases. The first is when $w$ is ``generally big" and second is when $w$ is ``generally small". Note that $2\frac{w-\beta_2}{\alpha_2-\beta_2} + 2\frac{\alpha_2-w}{\alpha_2-\beta_2} = 2$, and so, at every point, one of the two summands must be larger than $1$. 
		Thus the following two cases are exhaustive:
		
		\textbf{Case 1.} Suppose that:
		\begin{equation*}
			\Big |\Big \{x\in \Omega\cap B(0,2r)\,:\, 2\frac{w(x)-\beta_2}{\alpha_2-\beta_2}\geq 1\Big \}\Big |\geq \frac{1}{2}|B(0,2r)|.
		\end{equation*} 
		We apply Lemma \ref{lem:neummdensity} to the function $2\frac{w-\beta_2}{\alpha_2-\beta_2}$ and $\eta=\frac{1}{2}$ and we obtain that there exists a constant $0<c<1$ 
		depending only on the geometry of $\Omega$ and the ellipticity constant such that 
		\begin{equation*}
			\beta_1=\inf_{\Omega\cap B(0,r)}w\geq \beta_2+\frac{c}{2}(\alpha_2-\beta_2).
		\end{equation*}
		
		\textbf{Case 2.} Suppose that:
		\begin{equation*}
\Big |\Big \{x\in \Omega\cap B(0,2r)\,:\, 2\frac{\alpha_2-w}{\alpha_2-\beta_2}\geq 1\Big \}\Big |
\geq \frac{1}{2}|B(0,2r)|.
		\end{equation*}
		Using the same argument as in Case $1$, we obtain:
		\begin{equation*}
			\alpha_1=\sup_{\Omega\cap B(0,r)}w\leq \alpha_2-\frac{c}{2}(\alpha_2-\beta_2),
		\end{equation*}
		where $0<c<1$ is the same constant in the Case $1$.
		
		By definition $\beta_1\geq \beta_2$ and $\alpha_1\leq \alpha_2$, so in either case
		\begin{equation*}
			\alpha_1-\beta_2\leq (1-\frac{c}{2})(\alpha_2-\beta_2),
		\end{equation*}
		which concludes the proof.
	\end{proof}

	We are now ready to prove Theorem \ref{th:neumannharnack} using 
	the density Lemma.

	\begin{proof}[Proof of Theorem \ref{th:neumannharnack}] 
Without loss of generality, assume that $r=1$ and $K\geq 2$.
		By Lemma~\ref{lem:neumannMoser} with $\tau = 0$, 
$w$ is bounded in $B(0,K/2)$. Set $S = \sup_{B(0,1)} w$; we can assume that $S > 0$.
		
Notice that $w$ is a positive $A$-harmonic function, so we can apply Lemma \ref{l:CSmax} and get that 
	$w(A_{1/4}(0)) \geq C^{-1}S$, where $A_{1/4}(0)$ is a corkscrew point for $B(0,1/4)$. 
	Also, the Harnack principle says that we still have $w \geq C^{-1}S$ (with a larger $C$)
		in a corkscrew ball of size $\simeq 1$ centered at $A_{1/4}(0)$ and  contained in $\Omega$. Then the density Lemma \ref{lem:neummdensity} 
		says that $w \geq c_\eta C^{-1}S$ on $\Omega \cap B(0,1)$ (for some $\eta \simeq 1$ with the comparability depending only on the geometric constants of $\Omega$). Thus 
		$\inf_{\Omega \cap B(0,1)} w \geq c\, C^{-1} S \geq  c_\eta C^{-1} \sup_{\Omega \cap B(0,1)} w$,
		as announced in \eqref{e:neumannharnack}.
		\end{proof}

	\section{H\"older continuity for the Robin problem}
	\label{sec:density}
	We continue our efforts to prove H\"older continuity for solutions to the Robin problem. 
	
	Recall that if $Q\in \partial \Omega$ and $B \equiv B(Q,R)$ for some $R > 0$ we say that $u\in W^{1,2}(\Omega\cap B)$ solves the Robin problem with data $f$ if:
	
	\begin{align}
		\label{weakformloc}
		\frac{1}{a}\int_{\Omega\cap B}A\nabla u\nabla \varphi
	+\int_{\partial \Omega\cap B}u\varphi d\sigma =\int_{\partial \Omega \cap B} f\varphi d\sigma
	\end{align}
	for all $\varphi\in W^{1,2}(\Omega\cap B)$ such that $\varphi \equiv 0$ 
	on $\Omega \sm \mathcal{K}$ 
	for some compact $\mathcal{K}\subset \subset B$. We recall that we choose this definition for the test function because we want to avoid to refer to $\mathrm{Tr}(\varphi)|_{\partial B\cap \Omega}$ since $B\cap \Omega$ may not be a ``good" domain (for the trace/extension)  
(see Remark~\ref{rem:tracevszero}).
	
	The main goal of this section is to prove the following result:
	
	\begin{theorem}\label{t:HolderContinuity}[H\"older continuity of Weak Solutions] Let $(\Omega, \sigma)$ be a one sided NTA pair of mixed dimension, with $\Omega$ bounded.
	Let $K>1$ be large enough, depending on the geometric constants for the pair $(\Omega, \sigma)$.
Let $u$ be a weak solution to the Robin problem (in the sense of \eqref{weakformloc}) in $B(Q, KR)\cap \Omega$. Define the H{\"o}lder norms 
		\begin{align*}
	\|u\|_{C^{0,\alpha}(B(Q,R)\cap \Omega)} \vcentcolon
=\sup_{B(Q,R)\cap \Omega} |u| +\sup_{x\neq y \in B(Q,R)\cap \Omega} \frac{|u(x)-u(y)|}{|x-y|^\alpha} .
		\end{align*}
		There exist constants $C > 0$ and $\alpha_0 \in (0,1)$, depending on the geometric 
constants of $(\Omega, \sigma)$ 
and the ellipticity of $A$, but not on $a$, 
such that for $0 < \alpha \leq \alpha_0$ and $u$ as above, 
$$ \|u\|_{C^{0,\alpha}(B(Q,R)\cap \Omega)} \leq C\left(\fint_{B(Q,2R)\cap \Omega} u^2\right)^{1/2}+C \|f\|_{C^{0,\alpha}(B(Q,KR))}.$$
	\end{theorem}
	
	This theorem will follow from the results in the previous section, in particular from Lemma~\ref{lem:neumannMoser} (and the remark below that), and from Theorem \ref{t:RobinOsc} below.

	As before, it first behooves us to prove a density lemma. Note that we must be careful tracking the dependence of the constants on the parameter $a$. We would like to thank Tom\'as Merch\'an, from whose typed notes we first learned of this density lemma and the proof. 
	
	\begin{lemma}\label{lem:density}
Let $K>1$ be large enough, depending on the geometric constants,  and let $u\in W^{1,2}(\Omega\cap B(0,Kr))$, $u\geq 0$, satisfy 
		\begin{align}
			\label{supersolution}
			\frac{1}{a}\int_{\Omega\cap B(0,Kr)}A\nabla u\nabla \varphi
			+\int_{\partial \Omega\cap B(0,Kr)}u\varphi d\sigma
			\geq \int_{\partial \Omega\cap B(0,Kr)}\gamma\varphi d\sigma
		\end{align}
		for all $0\leq\varphi\in W^{1,2}(\Omega\cap B(0,Kr))$such that $\varphi \equiv 0$ on $\mathcal{K}^c$ for some compact $\mathcal{K}\subset \subset B(0,Kr)$, and some number $\gamma\geq 1$. If $u\leq \gamma$ on $\partial\Omega\cap B(0,Kr)$, then 
		\begin{align}
			\label{density}
	\inf_{B(0,r)}u\geq c_1\exp\left(-\frac{c_2}{\gamma a \sigma(B(0,r)) r^{2-n}} 
			\right),
		\end{align}
		where $c_1$ and $c_2$ depend on the geometric constants of $\Omega$.\end{lemma}
	
	All in all, we are assuming that the Robin data is greater than or equal to $\gamma$ and the Dirichlet data is smaller than or equal to $\gamma$ on $\partial\Omega\cap B(0, Kr)$. Notice that $\gamma = 1$ already gives a reasonable estimate; then if 
we can take $\gamma$ larger, the estimate \eqref{density} becomes better. 
And if the assumption holds only for some small $\gamma > 0$, we can always apply the result to $\lambda u$. 
Note also that when $\sigma$ is Ahlfors regular of dimension $d$, 
then $\sigma(B(0,r)) r^{2-n} \sim r^{2-n+d}$, with an exponent on $r$ 
which is strictly between $0$ and $2$ by our restriction on $d$, and is equal to $1$ when $d=n-1$.

	\begin{proof} 
		Let $\delta>0$ and define $u_{\delta}=u+\delta$ (we do this merely to make sure we do not divide by $0$, we will send $\delta$ to $0$ at the end). Note that 
		\begin{align}
			\label{superdelta}
	\frac{1}{a}\int_{\Omega\cap B(0,Kr)}&A\nabla u_{\delta}\nabla \varphi
		+\int_{\partial \Omega\cap B(0,Kr)}u_{\delta}\varphi  
		\nonumber
		\\ 
		=\frac{1}{a}&\int_{\Omega\cap B(0,Kr)}A\nabla u\nabla \varphi
			+\int_{\partial \Omega\cap B(0,Kr)}u\varphi
			+\int_{\partial \Omega\cap B(0,Kr)}\delta\varphi
			\geq \int_{\partial \Omega\cap B(0,Kr)}(\gamma+\delta)\varphi.
		\end{align}
		Call $\gamma_{\delta}:=\gamma+\delta$. We let now $v_{\delta}:=-\log{u_{\delta}}$. We want to prove that $v_{\delta}$ satisfies 
		\begin{align}
			\label{subsolution}
			\frac{1}{a}\int_{\Omega\cap B(0,2r)}A\nabla v_{\delta}\nabla \varphi+\int_{\partial \Omega\cap B(0,2r)}\gamma_{\delta}v_{\delta}\varphi\leq 0 
		\end{align}
		for all $0\leq \varphi\in W^{1,2}(\Omega\cap B(0,2r))$ with $\varphi=0$ on $\mathcal K^c$ where $\mathcal K \subset \subset B(0,2r)$. Since $\varphi$ vanishes outside of a compact subset of $B(0,2r)$ we can extend $\varphi$ by $0$ on $B(0,Kr)\setminus B(0,2r)$. We compute
		\begin{align*}
			\frac{1}{a}\int_{\Omega\cap B(0,2r)}A\nabla v_{\delta}\nabla \varphi+\int_{\partial \Omega\cap B(0,2r)}\gamma_{\delta}v_{\delta}\varphi=-\frac{1}{a}\int_{\Omega\cap B(0,Kr)}A\frac{\nabla u_{\delta}}{u_\delta}\nabla \varphi+\int_{\partial \Omega\cap B(0,Kr)}\gamma_{\delta}v_{\delta}\varphi\\
			=-\frac{1}{a}\int_{\Omega\cap B(0,Kr)}A\nabla u_{\delta}\nabla \left(\frac{\varphi}{u_\delta}\right)-\frac{1}{a}\int_{\Omega\cap B(0,Kr)}\frac{A\nabla u_{\delta}\nabla u_\delta}{|u_\delta|^2}\varphi+\int_{\partial \Omega\cap B(0,Kr)}\gamma_{\delta}v_{\delta}\varphi.
		\end{align*}
		Now, by \eqref{superdelta} applied with $\varphi/u_{\delta}$ instead of $\varphi$, we have that 
		\begin{align*}
			\frac{1}{a}\int_{\Omega\cap B(0,Kr)}A\nabla u_{\delta}\nabla \left(\frac{\varphi}{u_\delta}\right)\geq -\int_{\partial \Omega\cap B(0,Kr)}\varphi+\int_{\partial \Omega\cap B(0,Kr)}\gamma_{\delta}\frac{\varphi}{u_{\delta}}.
		\end{align*}
		Plugging it in the above estimate and forgetting the term with 
		$A\nabla u_{\delta}\nabla u_\delta \geq 0$,  we obtain 
		\begin{align*}
			\frac{1}{a} \int_{\Omega\cap B(0,2r)}A\nabla v_{\delta}\nabla \varphi+\int_{\partial \Omega\cap B(0,2r)}\gamma_{\delta}v_{\delta}\varphi\leq \int_{\partial \Omega\cap B(0,Kr)}\varphi\left(1-\frac{\gamma_{\delta}}{u_{\delta}}+\gamma_{\delta}v_{\delta}\right).
		\end{align*}
		Observe that $L\log x + L/x \geq 1$  
		for $x>0$ and $L>1$. Since $\gamma_{\delta}>1$ by hypothesis, we have that 
		\begin{align*}
			1-\frac{\gamma_{\delta}}{u_{\delta}}+\gamma_{\delta}v_{\delta}\leq 0,
		\end{align*}
		completing the proof of \eqref{subsolution}. 
	Lemma \ref{l:subsolformax} implies that $v_{\delta}^+$ is a (Neumann-)subsolution and the Moser inequality \eqref{e:moser}, with $\tau = 0$, yields 
			\begin{align}
			\label{moser}
			\sup_{\Omega\cap B(0,r)}v_{\delta}^+\leq C\left(\fint_{\Omega\cap B(0,2r)}(v_{\delta}^+)^2\right)^{\frac{1}{2}},
		\end{align}
		where $C$ depends on the geometric constants of $\Omega$ and the ellipticity of $A$. We will need the Poincar\'e inequality \eqref{e:poincaretrace}, which we apply to $v_{\delta}^+$:
		\begin{align}
			\label{e:poincaredelta}
	\fint_{\Omega\cap B(0,2r)}(v_{\delta}^+)^2\leq  \frac{C}{\sigma(B(0,r))}   
	\int_{\partial\Omega\cap B(0,r)}(v_{\delta}^+)^2+Cr^{2-n}\int_{\Omega\cap B(0,(K/2)r)}|\nabla v_{\delta}^+|^2.
		\end{align}
An attentive reader will have noticed (but should not be concerned) that we need to take $K$ four times larger than usual when applying  \eqref{e:poincaretrace} to make sure we integrate over $B(0, (K/2)r)$ in \eqref{e:poincaredelta}.

		Let $\beta\in C^{\infty}_c(B(0,Kr))$, with $0\leq \beta \leq 1$, $\beta=1$ on $B(0,(K/2)r)$, $|\nabla\beta|\leq C/r$. Using \eqref{superdelta} with $\varphi=\beta^2/u_{\delta}$, we estimate 
		\begin{align*}
			0\leq \int_{\Omega\cap B(0,Kr)}A\nabla u_{\delta}\nabla\left(\frac{\beta^2}{u_{\delta}}\right)+\int_{\partial\Omega\cap B(0,Kr)}a(u_{\delta}-\gamma_{\delta})\frac{\beta^2}{u_{\delta}},
		\end{align*}
and since $u\leq \gamma$ on $\partial\Omega\cap B(0,Kr)$, $u_{\delta}\leq\gamma_{\delta}$. 
Hence 
\begin{align*}
0 \leq \int_{\Omega\cap B(0,Kr)}A\nabla u_{\delta}\nabla\left(\frac{\beta^2}{u_{\delta}}\right)
= - \int_{\Omega\cap B(0,Kr)}\frac{\beta^2}{u_{\delta}^2} A\nabla u_{\delta} \nabla u_\delta
+ 2  \int_{\Omega\cap B(0,Kr)}\frac{\beta}{u_{\delta}} A\nabla u_{\delta} \nabla \beta
		\end{align*}
and, by ellipticity and Cauchy-Schwarz,
\begin{multline*}
	\int_{\Omega\cap B(0,Kr)}\frac{\beta^2}{u_{\delta}^2}|\nabla u_{\delta}|^2
	\leq C \int_{\Omega\cap B(0,Kr)}\frac {\beta}{u_{\delta}} |\nabla u_{\delta}| \, |\nabla \beta|
	\\ \leq C \Big\{\int_{\Omega\cap B(0,Kr)}\frac{\beta^2}{u_{\delta}^2}|\nabla u_{\delta}|^2 \Big\}^{1/2}
\left(\int_{\Omega\cap B(0,Kr)}|\nabla\beta|^2\right)^{\frac{1}{2}}.
		\end{multline*}
We simplify and get, by definition of $\beta$, 
\begin{align}\label{4a8}
	\int_{\Omega\cap B(0,(K/2)r)} |\nabla v_{\delta}^+|^2
	= \int_{\Omega\cap B(0,(K/2)r)} \frac{|\nabla u_{\delta}|^2}{u_{\delta}^2}
	\leq \int_{\Omega\cap B(0,Kr)}\frac{\beta^2}{u_{\delta}^2}|\nabla u_{\delta}|^2
	\nonumber\\
	\leq C \int_{\Omega\cap B(0,Kr)}|\nabla\beta|^2
	\leq C r^{n-2}.
\end{align}
For the boundary term in \eqref{e:poincaredelta}, consider $\varphi\in {C}^{\infty}_c(B(0,2r))$ such that $0\leq \varphi\leq 1$, $\varphi=1$ on $B(0,r)$, $|\nabla\varphi|\leq C/r$. We use \eqref{subsolution} to get 
\begin{align*}
	\int_{\partial\Omega\cap B(0,r)}a\gamma_{\delta}(v_{\delta}^+)^2
	\leq \int_{\partial\Omega\cap B(0,2r)}a\gamma_{\delta}v_{\delta}v_{\delta}^+\varphi^2
	\leq -\int_{\Omega\cap B(0,2r)}A\nabla v_{\delta}\nabla(v_{\delta}^+\varphi^2)
\end{align*}
and since 
$\varphi^2 A\nabla v_{\delta}\nabla(v_{\delta}^+) 
= \varphi^2 A\nabla v_{\delta}^+\nabla v_{\delta}^+ \geq 0$, 
\begin{eqnarray} 
	\int_{\partial\Omega\cap B(0,r)}a\gamma_{\delta}(v_{\delta}^+)^2
	&\leq& - 2 \int_{\Omega\cap B(0,2r)} v_{\delta}^+ \varphi A\nabla v_{\delta}\nabla\varphi 
\nonumber\\
&\leq& C\left(\int_{\Omega\cap B(0,2r)}|\nabla v_{\delta}^+|^2\varphi^2\right)^{\frac{1}{2}}\left(\int_{\Omega\cap B(0,2r)}| v_{\delta}^+|^2|\nabla\varphi|^2\right)^{\frac{1}{2}}\\
\nonumber\\	
&\leq& \frac{C}{r}\left(\int_{\Omega\cap B(0,2r)}|\nabla v_{\delta}^+|^2\right)^{\frac{1}{2}}\left(\int_{\Omega\cap B(0,2r)}(v_{\delta}^+)^2\right)^{\frac{1}{2}}
\nonumber
		\end{eqnarray}
by Cauchy-Schwarz and the definition of $\varphi$. Return to \eqref{e:poincaredelta}; we obtain
		\begin{align*}
	\fint_{\Omega\cap B(0,2r)}(v_{\delta}^+)^2\leq \frac{C|\Omega\cap B(0,2r)|^{\frac{1}{2}}}{\gamma_{\delta} a r \sigma(B(0,r))}  
	\left(\int_{\Omega\cap B(0,(K/2)r)}|\nabla v_{\delta}^+|^2\right)^{\frac{1}{2}}\left(\fint_{\Omega\cap B(0,2r)}(v_{\delta}^+)^2\right)^{\frac{1}{2}}\\
			+Cr^{2-n}\int_{\Omega\cap B(0,(K/2)r)}|\nabla v_{\delta}^+|^2.
		\end{align*}
		Suppose first that 
		\begin{align*}
	r^{2-n}\int_{\Omega\cap B(0,(K/2)r)}|\nabla v_{\delta}^+|^2
	\geq\frac{|\Omega\cap B(0,2r)|^{\frac{1}{2}}}{\gamma_{\delta} a r \sigma(B(0,r))} 
	\left(\int_{\Omega\cap B(0,(K/2)r)}|\nabla v_{\delta}^+|^2\right)^{\frac{1}{2}}\left(\fint_{\Omega\cap B(0,2r)}(v_{\delta}^+)^2\right)^{\frac{1}{2}};
		\end{align*}
		then 
		\begin{align*}
			\fint_{\Omega\cap B(0,2r)}(v_{\delta}^+)^2\leq 2Cr^{2-n}\int_{\Omega\cap B(0,(K/2)r)}|\nabla v_{\delta}^+|^2\leq \widetilde{C}
		\end{align*}
		by \eqref{4a8}. Otherwise, we have that 
		\begin{align*}
	\fint_{\Omega\cap B(0,2r)}(v_{\delta}^+)^2\leq \frac{C|\Omega\cap B(0,2r)|^{\frac{1}{2}}}{\gamma_{\delta} a r \sigma(B(0,r))} 
	\left(\int_{\Omega\cap B(0,(K/2)r)}|\nabla v_{\delta}^+|^2\right)^{\frac{1}{2}}\left(\fint_{\Omega\cap B(0,2r)}(v_{\delta}^+)^2\right)^{\frac{1}{2}},
		\end{align*}
		from which we obtain 
		\begin{align*}
			\fint_{\Omega\cap B(0,2r)}(v_{\delta}^+)^2
			\leq \frac{C} {(\gamma_{\delta} a)^2 r^{2-n} \sigma(B(0,r))^2} 
		\int_{\Omega\cap B(0,(K/2)r)}|\nabla v_{\delta}^+|^2
		\leq \frac{C}{(\gamma_{\delta} a r^{2-n} \sigma(B(0,r)))^2} 
		\end{align*}
		by \eqref{4a8} again. In any case we get 
		\begin{align*}
			\fint_{\Omega\cap B(0,2r)}(v_{\delta}^+)^2
			\leq C\left(\frac{1}{\gamma_{\delta} a r^{2-n}\sigma(B(0,r))}\right)^2+C. %
		\end{align*}
		Then by \eqref{moser}
		\begin{align*}
			\sup_{\Omega\cap B(0,r)}(-\log u_{\delta})^+
			\leq C'\left(\frac{1}{\gamma_{\delta} a r^{2-n} \sigma(B(0,r))}\right)+C' %
		\end{align*}
		and finally 
		\begin{align*}
			\inf_{\Omega\cap B(0,r)}u_{\delta}
			\geq \exp\left(-\frac{C}{\gamma_{\delta} a r^{2-n} \sigma(B(0,r))}-C\right). %
		\end{align*}
		Sending $\delta$ to $0$ we get \eqref{density}. 
		Note that if $1<u_\delta<\gamma,$ then the estimate on $\sup_{\Omega\cap B(0,r)}(-\log u_{\delta})^+$ is not informative, but the theorem is of course true automatically in this case.
	\end{proof}

	We now state a result that we need, together with Theorem \ref{th:neumannharnack}, in order to prove oscillation estimates in the small scales regime $ar^{2-n} \sigma(B(0,r))\lesssim 1$. 
We will use the tent domain $T = T(0,2r)$ from Lemma \ref{l:tentspaces}. 
We remind the reader of the notation $S(0,2r) = \d T(0,2r) \cap \Omega$
and $\Gamma(0,2r) = \d T(0,2r) \cap \d\Omega$ for the two pieces of boundary. 

	\begin{lemma}
		\label{th:neumannsub}
		Assume $0 \in \d\Omega$ and $r < \diam(\d\Omega)/4$.
		 Let $u\in W^{1,2}(T(0,2r))$ be bounded, with $\Tr(u)=0$ on 
		$S(0,2r) = \d T(0,2r) \cap \Omega$. Assume that for some number $\tau \geq 0$, $u$ satisfies the inequality 
		\begin{align}
			\label{e:neumannsub2}
	\int_{T(0,2r)}A\nabla u\nabla \varphi \leq  \tau \int_{\Gamma(0,2r)} \varphi d\sigma
		\end{align}
		for every $\varphi \in W^{1,2}(\Omega)$ such that $\varphi \geq 0$ on $T(0,2r)$ and 
		$\Tr(\varphi) = 0$ on $S(0,2r)$.
		Then 
		\begin{align}
			\label{e:neumannmoser}
			u \leq C  \tau  r^{2-n} \sigma(B(0,r)) \quad \text{ on } B(0,r), 
		\end{align}
where $C$ depends only on the geometric constants of $\Omega$ and the ellipticity of $A$. 
	\end{lemma}
	
	\begin{proof} 
	By the same scaling argument we used at the beginning of the proof of Lemma \ref{lem:neumannMoser}, we can assume that $r=1$, and even $\sigma(B(0,1)) = 1$.
Set $T = T(0,2)$, $\Gamma = \Gamma(0,2)$, and $S = S(0,2)$ to save notation.
	Let us prove the result first when $u \geq 0$. Using ellipticity we get 
		\begin{align*}
			\int_{T} |\nabla u|^2\leq C \int_{T}A\nabla u\nabla u
			\leq C \tau \int_{\Gamma} u d\sigma
		\end{align*} 
because 
$\varphi = u$ is an allowable test function in \eqref{e:neumannsub2}.
Call $\sigma_\star$ the doubling measure constructed on $T \cap B(0,2)$, and 
recall that its restriction to $\Gamma$ is $\sigma$.
Observe that $\sigma_\star(S) \geq C^{-1} \sigma_\star(\d T \cap B(0,2)) 
\geq C^{-1}(\Gamma)$ because $\sigma_\star$ is doubling 
and the Harnack chain condition gives a large ball in $S \sm \Gamma$; this is why we assumed
that $r < \diam(\d\Omega)/4$, to make sure that $\Omega$ connects a corkscrew ball for $B(0,r)$
to the exterior of $B(0,2r)$. 
Then by Cauchy-Schwarz and Corollary \ref{th:Poincare0} we get that
  $$ 
\tau \int_{\Gamma}u d\sigma \leq C \tau \left(\int_{\Gamma} |u|^2 d\sigma \right)^{1/2}
\leq C \tau \left(\int_{T}|\nabla u|^2\right)^{1/2}. 	
$$ 
	Combined with the above, this gives the energy estimate 
		\begin{equation}\label{e:energyest}
	\left(\int_{T} |\nabla u|^2\right)^{1/2} \leq C \tau. 		
	\end{equation} 
		By Lemma \ref{lem:neumannMoser}, 
		\begin{align}
			\label{e:kim}
			\sup_{\Omega\cap B(0,1)} u
		\leq C\left(\fint_{T} u^{2}\right)^{\frac{1}{2}}+C \tau. 		
		\end{align}
	Call $\ol u$ the average of $u$ on $T \cap B(0,2)$, 
	and apply our various Poincar\'e estimates to $\Omega' = T = T(0,2)$. 
	Lemma \ref{th:Poincare1} (with $p=2$ and $E = T \cap B(0,2))$ 
	says that
	\begin{align*}
	\fint_{T} u^{2}\leq 2 \ol u^2 + 2 \fint_{T} |u-\ol u|^{2}
	\leq 2 \ol u^2 + C\fint_{T} |\nabla u|^{2}
	\end{align*} 
	(notice that $K$ disappears because $T\cap B(0,2K) \subset T(0,2)$ trivially),
	while by Lemma \ref{th:Poincare3bis} (applied to a larger ball),
	\begin{align*}
	\ol u^2 = \fint_{S} |\Tr(u)- \ol u|^2 d\sigma_\star
	\leq C \fint_{\d T} |\Tr(u)- \ol u|^2 d\sigma_\star
	\leq C\fint_{T} |\nabla u|^{2}, 
	\end{align*} 
	so that altogether
	\begin{align*}
\fint_{T} u^{2} d\sigma \leq C\fint_{T} |\nabla u|^{2} \leq C \tau^2
	\end{align*} 
	by \eqref{e:energyest}, and the desired inequality \eqref{e:neumannmoser} follows from \eqref{e:kim}.
		This finishes the proof for $u\geq 0$.
	
	For general $u$, our strategy is to build a solution (rather than a subsolution) to \eqref{e:neumannsub2} and apply the maximum principle to conclude. 
Apply Theorem \ref{th:existencerobin2} to the pair $(T(0,2), \sigma_\star)$, 
with $\Sigma = S$  
(recall that $\sigma_\star(\Sigma) \geq C^{-1} \sigma_\star(\d T) > 0$
because $\sigma_\star$ is doubling), and choose the data $\psi = \tau {\bf {1}}_{\Gamma(0,2)}$.
	We get a function $u_0 \in W^{1,2}(T)$ such that $\Tr(u_0) = 0$ on $S$ and 
	\begin{align*}
		\int_{T} A\nabla u_0\nabla \varphi = \tau \int_{\Gamma} \varphi d\sigma
		\end{align*} 
for every $\varphi \in W^{1,2}(\Omega)$ such that $\varphi \geq 0$ on $T$ and $\Tr(\varphi) = 0$ on $S$.  (Observe that all the integrals are actually over $\Gamma$, 
where $\sigma_\star = \sigma$.) 
	Notice that $u_0$ is non-negative by a simple accretivity argument (i.e. testing against $\varphi = u_0^-$, see the argument below). So \eqref{e:neumannmoser} holds for $u_0$.
	
	Now let $u$ be as in the statement (potentially changing sign), and set $v = u_0-u$. If $u_0, u$ are classical (sub)solutions the maximum principle would immediately imply $v \geq 0$ and thus the desired bound for $u$. For completeness, we show $v\geq 0$ for these weak (sub)solutions.
	By \eqref{e:neumannsub2} and the definition of $u_0$,
we get that $\int_{T(0,2)}A\nabla v\nabla \varphi \geq 0$ for every $\varphi \in W^{1,2}(\Omega)$ 
such that $\varphi \geq 0$ on $T$ 
and $\Tr(\varphi) = 0$ on $S$ 
Taking $\varphi = v^-$, we get
	\begin{align*}
	0 \leq \int_{T} A\nabla (u_0-u) \nabla (u_0 -u)^-  
	= - \int_{T} A\nabla (u_0-u)^- \nabla (u_0 -u)^-
		\end{align*} 
	because we can integrate only where $u_0 - u < 0$. By ellipticity, $\nabla (u_0-u)^- = 0$, 
	and then $(u_0-u)^- = 0$ because we control the values on $S$ and 
	$\sigma_\star(S) > 0$  
	In this case, $u \leq u_0$ and \eqref{e:neumannmoser} for $u$ follows 
	from \eqref{e:neumannmoser} for $u_0$.
	\end{proof}

	Using the above lemma and Theorem \ref{th:neumannharnack} we want to prove a  Harnack-type inequality for Robin solutions at small scales.
	
		\begin{theorem}
		\label{th:robinharnack}
	Let $K > 1$ be large enough, depending only on the geometric constants 
	for $(\Omega, \sigma)$, and assume that $0 \in \d\Omega$ and 
	$4Kr < \diam(\d\Omega)$. 
	Let 
	$u\in W^{1,2}(\Omega \cap B(0,K^2r))$  be a nonnegative weak solution of the Robin problem 
	(see \eqref{weakformloc} and Remark \ref{rem:tracevszero})
		\begin{align} \label{4a13}
			\begin{cases}
		-\diver\left(A\nabla u\right)=0 &\text{in} \ \Omega\cap B(0,K^2r),\\
		\frac{1}{a}A\nabla u\cdot \nu+u=\beta &\text{on} \ \partial \Omega\cap B(0, K^2r),
			\end{cases}
		\end{align} 
	for some nonnegative $\beta\in L^{\infty}(\partial \Omega\cap B(0,K^2r))$. 
	There exists $c_0 \in (0,1)$ and $\theta \in (0,1)$, that depend on the geometric constants 
	of $(\Omega, \sigma)$     
	and the ellipticity constants for $A$, such that if 
\begin{equation} \label{e:notgenuine}
a r^{2-n} \sigma(B(0,r)) \leq c_0 
\end{equation}
and  
\begin{equation}\label{e:notgenuine2}
 0 \leq \beta \leq \frac{c_0}{ar^{2-n} \sigma(B(0,r))} 
		\, \sup_{\Omega\cap B(0,K^2r)}u
		\ \ \text{ a.e. on } \partial \Omega\cap B(0,K^2r), 
\end{equation}
		then 
		$$\inf_{\Omega \cap B(0,r)} u \geq \theta \sup_{\Omega \cap B(0,r)} u.$$
		\end{theorem}

	\begin{proof}	
		We pick $K > 1$ so large that $B(0,K^2r) \supset T(0,Kr)$, where $T(0,Kr)$ is the tent domain of Lemma \ref{l:tentspaces}. 
		We assumed that $4Kr < \diam(\d\Omega)$ to make sure that 
		$S(0,Kr) = \d T(0,Kr) \cap \Omega$ contains a reasonable ball, 
		hence has positive $\sigma_\star$ measure.
		By Theorem \ref{th:existencerobin2} 
		we can find $h\in W^{1,2}(T(0,Kr))$ such that $\mathrm{Tr}(h) = 0$ on $S(0,Kr)$
		and 
		\begin{equation} \label{4a16}
\frac{1}{a}\int_{T(0,Kr)} A\nabla h \nabla \varphi = \int_{\Gamma(0,Kr)} (\beta - u)\varphi
 	\end{equation} 	
	for all $\varphi \in W^{1,2}(T)$ such that $\Tr(\varphi) = 0$ on $S(0,Kr)$.
		
		We apply Lemma \ref{th:neumannsub} to $h$, with $\tau = a \sup_{T(0,Kr)}(\beta -u) \leq a \sup_{T(0,Kr)}\beta$, and get that 
		\begin{align*}
			\sup_{\Omega\cap B(0,r)} h\leq Cr^{2-n}\sigma(B(0,r))\tau \leq Cr^{2-n}\sigma(B(0,r))a\sup_{T(0,Kr)}\beta  \leq C c_0 \, \sup_{\Omega\cap B(0,K^2r)}u. 
		\end{align*}
		Notice that $-h$ satisfies the same equation, but with $u-\beta \leq u$, so the same argument, with $\tau = a \sup_{T(0,Kr)} u$,
		yields
		\begin{align*}
			-\inf_{\Omega\cap B(0,r)}h 
			\leq Car^{2-n} \sigma(B(0,r))\sup_{T(0, Kr)} u \leq C ar^{2-n} \sigma(B(0,r))\sup_{B(0,K^2r)\cap \Omega} u
			\leq C c_0 \sup_{B(0,K^2r)\cap \Omega} u
		\end{align*}
	by \eqref{e:notgenuine}. Thus, if we set $M = \sup_{B(0,K^2r)\cap \Omega} u$ and 
	$m = ||h||_{\infty}$, we get that $m \leq C c_0 M$  (we shall soon take $c_0$
	so small that $m \ll M$). 
		
	Our function $h$ was chosen to have the same Neumann data as $u$ on 
	$\Gamma(0,Kr) = \d \Omega \cap \d T(0,Kr)$ (see \eqref{weakformloc}). 
	Notice that $u-h$ lies in $W^{1,2}(\Omega \cap B(0,Kr))$ because 
	$\Omega \cap B(0,Kr) \subset T(0,Kr)$ and it satisfies the Neumann condition \eqref{e:neumann}
	(by \eqref{weakformloc} and \eqref{4a16}). 
	Apply Theorem~\ref{th:neumannharnack} to $w = u-h + m$ (which is non-negative), to obtain 
\begin{equation} \label{4a17}
\inf_{\Omega\cap B(0,r)}w\geq \eta \sup_{\Omega \cap B(0,r)}w.
\end{equation}
But $\sup_{\Omega \cap B(0,r)}w \geq M -m+m = M$,
and therefore 
\begin{align*}
		\inf_{\Omega \cap B(0,r)} u \geq \inf_{\Omega \cap B(0,r)} w - 2m
		\geq \eta M - 2m \geq \frac{\eta M}{2}
\end{align*}		
	if $c_0$ is small enough (depending on $\eta$ which itself depends only on the usual constants). This gives the desired result, with $\theta = \eta/3$.
	\end{proof}

\begin{remark}\label{r:harnackandproblem}
	If in Theorem \ref{th:robinharnack} $\beta=0$, i.e. if $u$ is a solution with locally homogeneous 
Robin data, then the above proof shows that $\inf_{\Omega\cap B(0,r)}u\geq c \sup_{\Omega\cap B(0,r)}u$, as long as $ar^{2-n} \sigma(B(0,r)) \leq c_0$ 
(i.e. without the need for \eqref{e:notgenuine2}), 
which is a genuine Harnack inequality at small scales. To see this, note that in this case, 
$h^+ = 0$ 
and by Lemma \ref{l:CSmax}, $\sup_{\Omega\cap B(0,r)}u\simeq \sup_{\Omega\cap B(0,K^2r)}u$.
\end{remark} 

Recall from the introduction a problem raised in \cite{BBC}:

\begin{problem}[c.f. Problem 1.2 in \cite{BBC}]\label{probproblem}
	Let $B\subset \Omega$ and let $u$ solve \begin{align}
		\label{e:probproblem}
		\begin{cases}
			-\mathrm{div}\left(A\nabla u\right)=0 &\text{in} \ \Omega\backslash B,\\
			\frac{1}{a}A\nabla u\cdot \nu +u=0 &\text{on} \ \partial \Omega,\\
			u = 1 &\text{on} \ B.
		\end{cases}
	\end{align}
	Characterize the $\Omega$ for which $\inf_{\Omega} u(x) > 0$ (in which case it is said the whole surface is ``active"). 
\end{problem}

\begin{proof}[Partial Solution of Problem \ref{probproblem}]
Theorem \ref{th:robinharnack} allows us to partially address this question. In particular, it follows immediately from that theorem (and Remark \ref{r:harnackandproblem}) 
that for any domain considered here, the whole boundary is active. In fact, we can show that $\inf_{\Omega} u > \delta > 0$ where $\delta > 0$ depends only on the geometric constants of $\Omega$, the ellipticity of $A$, the radius of $B$, the diameter of $\Omega$ and the distance from $B$ to $\partial \Omega$.  

In \cite{BBC}, they show that the whole boundary is ``active" for all domains which are ``well approximated by Lipschitz domains" (see \cite[Definition 2.1]{BBC}) and also give some examples of cusp domains for which the whole boundary fails to be active.  We emphasize that the domains considered here and those studied in \cite{BBC} are overlapping but distinct classes, and so our Theorem \ref{th:robinharnack} neither implies nor is implied by the main theorems in \cite{BBC}.
\end{proof}

We prove now oscillation estimates -- in the small scales regime we use Theorem \ref{th:robinharnack} while in the large scales regime we use Lemma \ref{lem:density}.

\begin{theorem}\label{t:RobinOsc}
Kor $K > 1$ large enough, depending on the geometric constants of $(\Omega, \sigma)$ 
there is a constant $\eta \in (0,1)$, that depends also 
 on the ellipticity constant (but not on $a$), such that if the data $f$ is bounded on $\d\Omega$, and
$u\in W^{1,2}(\Omega\cap B(0,Kr))$ is a bounded weak solution of the Robin problem (see \eqref{weakformloc} and Remark \ref{rem:tracevszero})
	\begin{align*}
		\begin{cases}
			-\diver A\nabla u=0 &\text{in} \ \Omega\cap B(0,Kr),\\
			\frac{1}{a}A\nabla u\cdot \nu+u=f &\text{on} \ \partial \Omega\cap B(0,Kr),
		\end{cases}
	\end{align*}
	then 
	\begin{align}
		\label{oscillation}
		\osc_{\Omega\cap B(0,r)}u\leq \eta \osc_{\Omega\cap B(0,Kr)}u 
		+ 2\osc_{\partial \Omega \cap B(0,Kr)} f. 
	\end{align}
\end{theorem}

\begin{proof}[Proof of Theorem \ref{t:RobinOsc}]
	We define $M_{K}=\sup_{\Omega\cap B(0,Kr)}u$, $m_{K}=\inf_{\Omega\cap B(0,Kr)}u$, $M=\sup_{\Omega\cap B(0,r)}u$, $m=\inf_{\Omega\cap B(0,r)}u$, $\overline{f}=\sup_{\partial \Omega \cap B(0,Kr)} f$, $\underline{f}=\inf_{\partial \Omega \cap B(0,Kr)} f$ and finally $\overline{M}=\max\{M_{K},\overline{f}\}$, $\underline{m}=\min\{m_{K},\underline{f}\}$. We need to prove that
	\begin{equation} \label{4a20}
M-m\leq \eta (M_{K}-m_{K})+ 
2(\overline{f}-\underline{f}).
\end{equation}
To save notation, we set $\sigma_r = \sigma(B(0,r))$ in the computations below.
	
\smallskip\noindent	\textbf{Small scales.} 
Suppose $2ar^{2-n} \sigma_r \leq c_0$,  
where $c_0$ comes from Theorem \ref{th:robinharnack}. We define 
	\begin{align*}
		v=\frac{u-\underline{m}}{\overline{M}-\underline{m}}.
	\end{align*}
	Clearly $v$ satisfies $-\diver A\nabla v=0$ and $0\leq v \leq 1$. Note that 
	\begin{align*}
		\frac{1}{a}A\nabla v\cdot \nu=\frac{f-u}{\overline{M}-\underline{m}}=\frac{f-(\overline{M}-\underline{m})v-\underline{m}}{\overline{M}-\underline{m}},
	\end{align*}
	(weakly), and therefore 
	\begin{align*}
		\frac{1}{a}A\nabla v\cdot \nu+v=\frac{f-\underline{m}}{\overline{M}-\underline{m}}=\vcentcolon \beta.
	\end{align*}	
	Note that $0\leq\beta\leq 1$ and hence by the assumption on $r$, 
	$\beta \leq c_0 [ar^{2-n} \sigma_r]^{-1}$.  
	We want to apply Theorem~\ref{th:robinharnack},
	so we assume that the present $K$ is larger than $K^2$ there. 
	
	First assume that $c_0 [ar^{2-n} \sigma_r]^{-1} \sup_{\Omega\cap B(0,r)}v \geq 1$.
	\\ 
	In this case $0 \leq \beta \leq 1 \leq c_0 [ar^{2-n} \sigma_r]^{-1} 
	\sup_{\Omega\cap B(0,Kr)}v$, 
	Theorem~\ref{th:robinharnack} applies to $v$, and we get that 
	\begin{align} \label{4a21}
		\inf_{\Omega\cap B(0,r)}v \geq \theta  \sup_{\Omega\cap B(0,r)}v 
		\geq  \theta \, \frac{ar^{2-n} \sigma_r}{c_0} 
	\end{align}
	and therefore 
	$u - \underline{m} = (\overline{M}-\underline{m})  v \geq \theta\frac{ar^{2-n} \sigma_r}{c_0} (\overline{M}-\underline{m})$
	on $\Omega\cap B(0,r)$, 
	whence $m \geq \underline{m} +  \theta \, \frac{ar^{2-n} \sigma_r} 
	{c_0}(\overline{M}-\underline{m})$ and 
	(since obviously $\underline{m}\leq m$)
	\begin{align}  \label{4a22}
		M-m\leq M-\underline{m}
		\leq \left(1-\theta \, \frac{ar^{2-n} \sigma_r} 
		{c_0}\right)(\overline{M}-\underline{m}).
	\end{align} 
	In the other case when $c_0 [ar^{2-n} \sigma_r]^{-1}  
	\sup_{\Omega\cap B(0,r)}v \leq 1$, or equivalently
	$c_0 [ar^{2-n} \sigma_r]^{-1} 
	(M-\underline{m}) \leq \overline{M}-\underline{m}$, we directly get that
	\begin{align}  \label{4a23}
		M-m\leq M-\underline{m}\leq  \frac{ar^{2-n} \sigma_r} 
		{c_0} \, (\overline{M}-\underline{m}) \leq \frac12 (\overline{M}-\underline{m}).
	\end{align}
	In both cases, we proved that $M-m\leq  c \, (\overline{M}-\underline{m})$ for some $c \in (0,1)$.

	Suppose now first that the intervals $I_{K}=[m_{K}, M_{K}]$ and $I_f=[\underline{f},\overline{f}]$ intersect. 
	If either of them contains the other, we clearly have that 
	$\overline{M}-\underline{m}\leq M_{K}-m_{K}+\overline{f}-\underline{f}$ 
	and \eqref{4a20} holds. 
	Otherwise, suppose for instance that $m_{K}\leq \underline{f}\leq M_{K}\leq \overline{f}$; then again 
	$\overline{M}-\underline{m}=\overline{f}-m_{K}\leq M_{K}-m_{K}+\overline{f}-\underline{f}$. 
	The same is true in all the other case in which $I_{K}$ and $I_f$ intersect, so \eqref{4a20} holds then.
	
	So we may assume that $I_{K}\cap I_f=\emptyset$, and for instance suppose that $I_f$ lies to the right of $I_{K}$. 
	We define another auxiliary function: 
	\begin{align*}
		\widetilde{v}=\frac{u-m_{K}}{M_{K}-m_{K}}.
	\end{align*}
	Note that 
	\begin{align*}
		\frac{1}{a}A\nabla \widetilde{v}\cdot \nu+\widetilde{v}=\frac{f-m_{K}}{M_{K}-m_{K}}=\vcentcolon\widetilde{\beta},
	\end{align*}
	and since by assumption $f\geq M_{K}$, we have that $\widetilde{\beta}\geq1$ and moreover $\widetilde{v}\leq\widetilde{\beta}$. 
	We distinguish two cases:\\
	\textit{a)} $\widetilde{\beta}\geq c_0 [ar^{2-n} \sigma_r]^{-1}$. 
	We apply the Density Lemma \ref{lem:density} with $\gamma = 1$, 
	\begin{align*}
\inf_{\Omega\cap B(0,r)}v
\geq c_1\exp\left(-\frac{c_2}{\widetilde\beta ar^{2-n} \sigma_r}\right) 
	\geq c_1 \exp\left(-\frac{c_2}{c_0}\right) = c_3 > 0. 
	\end{align*}
		From this and the definition of $\widetilde{v}$ it follows easily that 
	$M-m\leq \eta(M_{K}-m_{K})$ for some $\eta \in (0,1)$, and we are done.
		\\
	\textit{b)} $\widetilde{\beta}< c_0 [ar^{2-n} \sigma_r]^{-1}$.
	 In this case we will apply again Theorem \ref{th:robinharnack}. 
	\\ 
	If  $c_0 [ar^{2-n} \sigma_r]^{-1} 
	\sup_{\Omega\cap B(0,r)}v < 1$, and  as in \eqref{4a23}, 
	we have $M-m\leq M-m_{K}\leq \frac{ar^{2-n} \sigma_r}{c_0}\,(M_{K}-m_{K}) 
	\leq  \frac12 \,(M_{K}-m_{K})$, which is what we want. 
		\\
	If instead $I_f$ lies to the left of $I_{K}$, we repeat the argument substituting $\overline{f}$ for $M_{K}$ and $\underline{f}$ for $m_{K}$ in the definition of $\widetilde{v}$.

	\medskip\noindent
	\textbf{Large scales.} We now assume that $2ar^{2-n} \sigma_r> c_0$. 
	Set $\Delta = M_K - m_k$. We may also assume that 
	$\overline{f} - \underline{f} \leq \Delta/2$ because otherwise 
	\eqref{4a20} follows directly since $M-m \leq \Delta$.
	Let us first assume that $\overline{f} \leq M_K - \Delta/2$,
	and use Lemma \ref{lem:density} to show that 
	\begin{align}
		\label{upper}
		M_{K} - M \geq \eta_1 (M_{K}-\overline{f})
	\end{align}
	for some $\eta_1 \in (0,1)$. Then \eqref{4a20} will follow, with $\eta  = 1-\eta_1$
	because $m \geq m_K$. Define 
		\begin{align*}
		v_1&=\frac{M_{K}-u}{M_{K}-\overline{f}} 
	\end{align*}
	and observe that
	\begin{align*}
		\frac{1}{a}A\nabla v_1\cdot \nu+v_1=\frac{M_{K}-f}{M_{K}-\overline{f}}=\vcentcolon \beta_1,
	\end{align*}
	with $\beta_1\geq 1$ (because $f \leq \overline{f}$ and the denominator is positive), 
and $0 \leq v_1 \leq 2$ because $M_{K}-f \leq M_k - m_K = \Delta \leq 2 (M_{K}-\overline{f})$.
So $v_1$ satisfies the hypotheses of Lemma \ref{lem:density} with $\gamma=2$, hence 
	\begin{align*}
	\inf_{\Omega\cap B(0,r)}v_1\geq c_1e^{-c_2/(2ar^{2-n} \sigma_r)} 
	\geq c_1e^{-c_2/(4c_0)}=: \eta'.
	\end{align*}
	and $M_K - M = M_K - \sup_{\Omega\cap B(0,r)} u = (M_K - \overline{f}) \inf_{\Omega\cap B(0,r)} v_1 = (M_K - \overline{f}) \eta' \geq \Delta \eta'/2$. So
\eqref{upper} holds with $\eta_1 = \eta'/2$.
	
	We may now assume that $\overline{f} > M_K - \Delta/2$, and 
	since $\overline{f} - \underline{f} \leq \Delta/2$, this implies that
	$\underline{f} \geq m_K + \Delta/2$. This time we do the symmetric construction. Consider 
	\begin{align*}
		v_2&=\frac{u-m_{K}}{\underline{f}-m_{K}},
	\end{align*}
	and observe that 
	\begin{align*}
		\frac{1}{a}\nu \cdot A\nabla v_2+v_2=\frac{f-m_{K}}{\underline{f}-m_{K}}=\vcentcolon \beta_2 \, , 
		\end{align*}
with $\beta_2 \geq 1$, that $v_2 \geq 0$ by definition of $\underline{f}$, 
and that $v_2 \leq 2$ because $u-m_K \leq \Delta$.
Lemma \ref{lem:density} 
applies, again with $\gamma=2$, and yields 
$\inf_{\Omega\cap B(0,r)}v_2\geq \eta'$, and we get that 
$m-m_K \geq \eta' (\underline{f}-m_{K}) \geq \eta' \Delta/2$, 
from which we can deduce \eqref{4a20}.
	So \eqref{4a20} holds in both cases, and Theorem \ref{t:RobinOsc} follows.
\end{proof}

As was suggested above, Theorem \ref{t:HolderContinuity}
follows from Lemma \ref{lem:neumannMoser}, which we can use to get a first $L^{\infty}$ bound,
and Theorem \ref{t:RobinOsc}, that gives the desired decay for the oscillation
of $u$ on small balls.

\section{Existence of Robin Harmonic Measure }\label{s:RRT}

As before, all the domains in this section are bounded, one-sided NTA domains $\Omega$,
for which there is a measure $\sigma$ on $\d\Omega$ such that 
$(\Omega, \sigma)$ is a pair of mixed dimension.
 Our goal in this section is to prove the existence of the 
 Robin harmonic measure, the existence of the Robin Green function and the relation between them.

The oscillation estimates of the previous sections first allow us to find a continuous {\it up to the boundary} solution to the Robin problem for any continuous data in such domains. 

\begin{lemma}\label{l:continuoussolution}
	For any $f\in C(\partial \Omega)$ there exists 
	$u\in W^{1,2}(\Omega)\cap C(\overline{\Omega})$ such that 
	$$\frac{1}{a}\int_{\Omega} A\nabla u\nabla \varphi +\int_{\partial \Omega} u\varphi = \int_{\partial \Omega} f\varphi, \qquad \forall \varphi \in C_c^\infty(\mathbb R^n).$$
\end{lemma}

\begin{proof}
	The existence of a weak solution, even with data $f \in L^2(\d\Omega)$, follows from Theorem~\ref{th:existencerobin}. That this solution is continuous up to the boundary follows immediately from Theorem \ref{t:RobinOsc} and, as before,  Lemma \ref{lem:neumannMoser}.
\end{proof}

Before we can construct the Robin harmonic measure via the Riesz Representation theorem, we need a weak maximum principle:

\begin{lemma}\label{l:maxprinc}[Weak Maximum Principle]
	Let $u \in W^{1,2}(\Omega)$ solve 
	$$\frac{1}{a}\int_{\Omega} A\nabla u \nabla \varphi + \int_{\partial \Omega} u\varphi = \int_{\partial \Omega} f\varphi, 
	\quad\forall \varphi \in C_c^\infty(\mathbb R^n),
	$$ 
	where $L^2(\partial \Omega)\ni f\geq 0$ almost everywhere. Then $u\geq 0$ in $\Omega$.
\end{lemma}

\begin{proof}
	Note that $u^- \in W^{1,2}(\Omega)$ (as it is the maximum of two Sobolev functions). Since $C_c^\infty(\mathbb R^n)$ is dense in $W^{1,2}(\Omega)$ (by the extension property
	of $\Omega$), we can conclude that 
	
	$$\frac{1}{a}\int_{\Omega} A\nabla u \nabla u^- + \int_{\partial \Omega} uu^- = \int_{\partial \Omega} fu^-.$$ Note, however, that the left hand side of the above equation is $\leq 0$ (by definition and ellipticity) whereas the right hand side is $\geq 0$ because $f\geq 0$. This is only possible if 
$\nabla u^-$ on $\Omega$ and $u^- = 0$ on $\d\Omega$, hence 
	$u^- \equiv 0$, which, of course, is the desired result. 
\end{proof}

An easy corollary of the above two lemmas is a strong maximum principle for continuous solutions:

\begin{corollary}\label{c:strongmax}
	Let $u \in W^{1,2}(\Omega)$ solve 
	$$\frac{1}{a}\int_{\Omega} A\nabla u \nabla \varphi + \int_{\partial \Omega} u\varphi = \int_{\partial \Omega} f\varphi, 
	\quad \forall \varphi \in C_c^\infty(\mathbb R^n),$$ 
	where $f\in C(\partial \Omega)$. Then for every $x\in \Omega$, $u(x)\leq \sup_{Q\in \partial \Omega} f(Q)$.
\end{corollary}

\begin{proof}
	We can assume by scaling that $\sup_{Q\in \partial \Omega} f = 1$. Applying Lemma \ref{l:maxprinc} to $1-u$ we see that $1-u \geq 0$ almost everywhere in $\Omega$. But by Lemma \ref{l:continuoussolution}, $1-u\in C(\overline{\Omega})$, so $1-u \geq 0$ in all of $\Omega$.
\end{proof}

From here, we can show that the Robin elliptic measure exists: 

\begin{theorem}\label{t:hmexistence}
Let $(\Omega, \sigma)$ be a 1-sided NTA pair of mixed dimension, with $\Omega$ bounded.
There exists a family of Radon probability measures $\{\omega_{\Omega, R}^x\}_{x\in \Omega}$ supported on $\partial \Omega$, such that for every $f\in C(\partial \Omega)$ the unique weak solution of the Robin problem in $\Omega$ with data $f$, call it $u_f$, is given by 
$$
u_f(x) = \int_{\partial \Omega} f(Q)\, d\omega_{\Omega,R}^x(Q).
$$
\end{theorem}

\begin{proof}

First observe that for each $f \in C(\partial \Omega)$, there is a unique weak solution $u_f$ for the data $f$. Indeed, if there were two solutions, say $u_f$ and $v_f$, then their difference would weakly solve the homogeneous Robin problem and would be continuous up to the boundary. By Lemma \ref{l:maxprinc}, this implies that $u_f = v_f$.
The linearity of the map $f \mapsto u_f$ follows immediately from the weak formulation of the problem, and its boundedness follows from Corollary \ref{c:strongmax}.

	For any $x\in \Omega$, define the map $f\mapsto u_f(x)$. Since this map is well defined, linear and bounded with norm one from $C(\partial \Omega)$ to $\mathbb R$,
the Riesz Representation theorem gives the desired result. 
	
\end{proof}

We now are ready to show the main result of this section:

\begin{theorem}\label{t:varequiv}
Let $(\Omega, \sigma)$ be as in Theorem \ref{t:hmexistence}. 
	For any Borel measurable $E \subset \partial \Omega$ 
	we have that $u(x):= \omega^x_R(E)$ solves \begin{equation}\label{e:weakhm}
	\frac{1}{a}\int_{\Omega} A\nabla u \nabla \varphi + \int_{\partial \Omega} u\varphi = \int_{E}\varphi,\qquad  \forall \varphi \in C_c^\infty(\mathbb R^n).
	\end{equation}
\end{theorem}

\begin{proof}
	In view of Theorem~\ref{t:hmexistence}, we only need an approximation argument to pass from the characteristic function to continuous data. We follow the approach of \cite{kenigbook}, which proves a similar result for the 
harmonic measure. From the Riesz representation theorem we know that for each $x\in \Omega$ 
we have that $\omega^x_R$ is a Radon measure on $\partial \Omega$. 
In particular, fixing $x_0\in \Omega$, 
for each $i \in \mathbb N^\ast$ 
there exists $K_i \cap \partial \Omega \subset E \subset O_i \cap \partial \Omega$ 
where $K_i$ is compact and $O_i$ is open so that $\omega^{x_0}_R(O_i\backslash K_i) 
< \frac{1}{i}$. Note that since $\sigma$ is also a Radon measure,  
by perhaps shrinking $O_i$ or expanding $K_i$ we can also guarantee that 
$\sigma(O_i\backslash K_i) < \frac{1}{i}$.
	
	Let $f_i \in C(\partial \Omega)$ be such that $0 \leq f_i \leq 1$ and $f_i \equiv 1$ on $K_i$ and $\mathrm{spt} f_i \subset O_i$. We then define 
	$$u_i(x) = \int_{\partial \Omega} f_i(Q)\, d\omega_R^x(Q).$$ 
	We know from Theorem \ref{t:RobinOsc} and Lemma \ref{l:maxprinc} that 
$u_i \in C(\overline{\Omega})$ and $\sup |u_i| \leq 1$. 
By definition, the sequence $\{u_i(x_0)\}_i$ converges for every $x_0 \in \Omega$, 
and thus by the Harnack inequality 
 the $u_i$ converge uniformly on compact subsets of $\Omega$ to a limit $u_\infty$, 
 which solves $-\mathrm{div}(A \nabla u_\infty) = 0$ in $\Omega$ (because the $u_i$ do).
	
	We observe that the Harnack inequality tells us that for every $x\in \Omega$ 
and every $\varepsilon > 0$ there exists an $i_0 \geq 1$ large enough 
(depending on $x, x_0, \varepsilon$) 
such that $i \geq i_0$ implies that $\omega_R^x(O_i\backslash K_i) < \varepsilon$. 
Since $|u_i(x) - \omega^x_R(E)| \leq \omega_R^x(O_i\backslash K_i)$, 
we get that $u_\infty(x) = \omega_R^x(E)$ for every $x\in \Omega$.

	By definitions, each $u_i$ weakly solves 
	$$\frac{1}{a}\int_{\Omega} A\nabla u_i\nabla \varphi +\int_{\partial \Omega} u_i\varphi 
	= \int_{\partial \Omega} f_i\varphi,\quad \forall \varphi\in C_c^\infty(\mathbb R^n).$$ 
	By Theorem \ref{th:existencerobin} 
	and the construction of the $f$ we have that 
	$$
	\|u_i\|_{W^{1,2}(\Omega)} \leq C\|f_i\|_{L^2(\partial \Omega)} 
	\leq  C \sigma(E). 
	$$ 
	So $u_i$ has a weak accumulation point in $W^{1,2}$ and by the argument above we have $u_i \rightharpoonup u_\infty$ in $W^{1,2}(\Omega)$. Since the trace is continuous we also have $u_i\rightharpoonup u_\infty$ in $L^2(\partial \Omega)$. Thus taking limits above we get that $u_\infty$ weakly solves \eqref{e:weakhm} and we are done. 
		\end{proof}

We end this section by constructing the Robin Green function and showing a relationship between that Green function and the Robin harmonic measure.

Let $f,g\in W^{1,2}(\Omega)$. We define the natural  quadratic form associated to the Robin problem
\begin{equation}
    b(f,g):= \int_{\Omega}A\nabla f\nabla g+a\int_{\partial \Omega}fg d\sigma. 
\end{equation}
The latter is symmetric, if $A$ is. 

\begin{theorem}[Existence of the Robin Green function] 
\label{t:GreenfunctionExistence}
Let $(\Omega, \sigma)$ be a 1-sided NTA pair of mixed dimension, with $\Omega$ bounded.
Then for every $y\in \Omega$ there exists a non-negative function $G_R(\cdot,y)$ such that 
	\begin{equation*}
		G_R(\cdot,y)\in W^{1,2}(\Omega\setminus B(y,r))\cap W^{1,1}(\Omega)
		\quad \text{for any $r > 0$,}
	\end{equation*}
	and such that for all $\phi\in C^{\infty}(\Omega)$
	\begin{equation}\label{e:Greenidentity}
		b(G_R(\cdot,y),\phi)= \phi(y).
	\end{equation}
	Moreover, if $u$ is a weak solution to the Robin problem with data $f\in L^2(\partial \Omega)$ (i.e. it satisfies \eqref{weakform}), then 
	\begin{equation}\label{e:representationformula}
		u(x)= 
		a\int_{\partial \Omega}f(y)G_R(x,y)d\sigma(y)  
		\,\,\text{for every } x\in \Omega. 
	\end{equation}
	\end{theorem}

\ms
The representation formula \eqref{e:representationformula} also works when $f \in H^\ast$,
the dual of $H$ (see Theorem~\ref{th:existencerobin} and \eqref{2a8}), with the same proof.

More importantly, notice that \eqref{e:Greenidentity}, applied with $\phi$ supported away from $y$, 
says that $G_R(\cdot, y)$ is, locally on $\Omega \sm \{ y\}$ a weak solution of our equation
${\rm{div}} A \nabla u = 0$, with homogeneous Robin conditions
(compare to our equation \eqref{weakformloc} with $f=0$. 
In particular, away from $y$, it is locally H\"older continuous up to the boundary.

\begin{proof}
	We show the existence of this object by abstract functional analytic methods, see \cite[Chapter 3]{Davies} and \cite{AE95}.

Since $b$ is accretive in $W^{1,2}(\Omega)$ (by, e.g. Proposition \ref{prop:equivalence of norms}), 
we can apply the Stampacchia theorem and get that for any $f\in L^2(\Omega)$ 
there exists a unique  function $w_f \in W^{1,2}(\Omega)$ such that 
\begin{equation} \label{5d8}
b(w_f, \varphi) = \int_{\Omega} \varphi f \ \text{for all } \varphi \in C^\infty_c(\Omega)
\end{equation}
and $\|w_f\|_{W^{1,2}(\Omega)} \leq C\|f\|_{L^2(\Omega)}$. 
We call $\mathcal S : L^2(\Omega) \to W^{1,2}(\Omega)$ the solution operator; 
thus $w_f = \mathcal S f$. We test \eqref{5d8} with $\varphi \in C_c^\infty(\Omega)$, 
the boundary integral disappears and get that $-\mathrm{div}(A(x)\nabla w_f) = f$ (weakly) in $\Omega$. 

We follow \cite{GW}, let $y \in \Omega$ and $\rho > 0$ and define $f_\rho$ to be a smooth approximation of $\frac{1}{|B(y,\rho)|}\chi_{B(y,\rho)}$. Define $G_R^\rho(-, y)$ to be the solution (given in the paragraph above) with right hand side $f_\rho$. By elliptic regularity, $G_R^\rho(-, y)\in C(\Omega)$ and by the weak maximum principle $G_R^\rho(x, y) \geq 0$ for all $x\in \Omega$. Also note (by testing against $\varphi$ with support away from $y$) that $G_R^\rho(-,y) \in W^{1,2}_{loc}(\Omega\backslash \{y\})$ with a norm independent of $\rho$. 

Let $s\in (1, \frac{n}{n-1})$ and let $s'$ be the Sobolev conjugate. By Sobolev embedding, if $B_{4\epsilon}(y)\subset \Omega$ and $\phi \in W^{1,s'}(B_{4\epsilon}(y))$ then $\phi \in L^\infty(B_{4\epsilon}(y))$. Thus $f_\rho \in (W^{1,s'}(B_{4\epsilon}(y)))^*$ with a uniform bound on its operator norm. Calder\'on-Zygmund estimates then imply that $G_R^\rho(-, y) \in W^{1,s}(B_{2\epsilon}(y))$ again with a uniform bound on the operator norm. Taking limits (and using the increased regularity of $G_R^\rho$ away from $y$) we have that $G_R(x,y) = \lim_{\rho \downarrow 0} G_R(-,y)$ exists in $W^{1,1}(\Omega)$ and \eqref{e:Greenidentity} follows (since $\lim_{\rho \downarrow 0} \fint_{B(y,\rho)} \phi = \phi(y)$).

	We are left with showing \eqref{e:representationformula}; this would not be so difficult except for the fact that solutions to differential equations with merely bounded elliptic coefficients are not necessarily Lipschitz, and this lack of regularity makes it hard to justify certain limits. Thus we regularize. Let $A_\epsilon \in C^1(\Omega)$ be uniformly elliptic real matrix valued functions such that $A_\epsilon \rightarrow A$ in $L^2(\Omega)$ (i.e. each entry converges in $L^2(\Omega)$). We can choose them so that $\lambda I \leq A_\epsilon \leq \Lambda I$ for all $\epsilon > 0$. Let $G_R^\epsilon$ be the Green's function above associated to $A_\epsilon$ and let $u^\epsilon_f\in W^{1,2}(\Omega)$ (weakly) solve $-\mathrm{div}(A_\epsilon\nabla u) = 0$ in $\Omega$ with Robin data $f$ (i.e. $u^\epsilon_f$ satisfies \eqref{weakform} with the matrix $A_\epsilon$). Our first goal is to show that (up to a subsequence) $u^\epsilon_f\rightarrow u_f$ uniformly on compact subsets of $\Omega$ and that for each $x\in \Omega$, $G_R^\epsilon(x, -) \rightharpoonup G_R(x,-)$ in $L^2(\partial \Omega)$ as $\epsilon \downarrow 0$.

First note that $u^\epsilon_f$ are bounded in $W^{1,2}(\Omega)$ and thus have a weakly convergent subsequence. Furthermore for any $\varphi \in C^\infty(\Omega)$ we have that 
$$
\left|\frac{1}{a}\int_{\Omega} A \nabla u^\epsilon_f \nabla \varphi +\int_{\partial \Omega} (u^\epsilon_f-f) \varphi\right| \leq \frac{1}{a} \|A-A_\epsilon\|_{L^2} \|\nabla \varphi\|_{L^\infty}\|\nabla u^\epsilon\|_{L^2} \rightarrow 0.
$$ 
By uniqueness of weak solutions we have that $u^\epsilon_f$ converges weakly to $u_f$. But by De Giorgi-Nash-Moser the $u^\epsilon_f$ are bounded in $C^{2\alpha}(K)$ for each $K \subset \subset \Omega$ and some $2\alpha < 1$.  It follows that $u^\epsilon_f \rightarrow u_f$ in $C^\alpha_{loc}(\Omega)$. To show that the Green function converges in the appropriate way we note again that $G_R^\epsilon$ are uniformly bounded in $W^{1,2}(\Omega\backslash B(y,r))$ for any fixed $r > 0$. This implies that 
they 
converge weakly in $W^{1,2}(\Omega\backslash B(y,r))$.  Furthermore, if $g \in C_c^\infty(\Omega\backslash \{y\})$ and $w^\epsilon_g$ solves $-\mathrm{div}(A_\epsilon \nabla w^\epsilon_g) = g$ weakly with zero Robin boundary data, then the argument above shows that $w^\epsilon_g(y) \rightarrow w_g(y)$ and thus $$\int_\Omega G_R^\epsilon(x,y) g(y)\,dy \rightarrow \int_\Omega G_R(x,y)g(y)\, dy.$$ So indeed $G_R^\epsilon \rightharpoonup G_R$ in $W^{1,2}(\Omega\backslash B(y,r))$.

Let $\varphi \in C^\infty(\mathbb R^n\backslash \{y\})$ we have that $$\int_\Omega A_\epsilon \nabla G_R^\epsilon\nabla \varphi - \int_\Omega A\nabla G_R\nabla \varphi dx = a\int_{\partial \Omega} \varphi(G_R-G_R^\epsilon).$$ We can see that the integral on the left goes to zero by adding and subtracting $A_\epsilon\nabla G_R\nabla \varphi$ and using that $A_\epsilon \rightarrow A$ in $L^2$ and $G_R^\epsilon \rightharpoonup G_R$ in $W^{1,2}$. Since $\varphi|_{\partial \Omega}$ is essentially arbitrary, this gives the desired weak convergence. 
	
	As such it suffices to show \eqref{e:representationformula} when $A\in C^1(\Omega)$ as the general result follows because $G_R^\epsilon(x,y)\rightharpoonup G_R(x,y)$ in $L^2(\partial \Omega)$ and $u_f^\epsilon(y) \rightarrow u_f(y)$ uniformly on compact subsets of $\Omega$. To avoid the proliferation of epsilons we will drop them from now on. 	
	
Let $u_f \in W^{1,2}(\Omega)$ (weakly) solve $-\mathrm{div}(A\nabla u) = 0$ in 
$\Omega$ with Robin data $f$ (i.e. $u_f$, satisfies \eqref{weakform}). 
Recalling that $\Omega$ is an extension domain, $C^\infty(\Omega)$ is dense in $W^{1,2}(\Omega)$, so taking $C^\infty(\Omega) \ni \phi_i \rightarrow u_f\in W^{1,2}(\Omega)\cap W^{1, \infty}(B(y, \varepsilon) )$ (here we are using that $A$ is smooth and so $u_f$ is locally Lipschitz) in \eqref{e:Greenidentity} we have 
$$
u_f(y) = \int_{\Omega}A\nabla u_f \nabla G_R(\cdot, y)
+ a\int_{\partial \Omega}u_f  G_R(\cdot ,y) d\sigma.
$$
Write $G_R(\cdot, y) = \eta G_R(\cdot, y) 
+ (1-\eta)G_R(\cdot, y)$ where $\eta \in C_c^\infty(B(y, 2\theta))$, $\eta \equiv 1$ in $B(y, \theta)$ and $0 \leq \eta \leq 1$ always. Note we can ensure that $|\nabla \eta| \leq 3\theta^{-1}$ pointwise. 

By the density of $C^\infty(\Omega)$ in $W^{1,2}(\Omega)$ again, in \eqref{weakform} we can test against
 $\varphi := (1-\eta) G_R(\cdot,y) \in W^{1,2}(\Omega)$, to obtain
$$
\int_{\Omega}A\nabla u_f \nabla ((1-\eta) G_R(x, y))\, dx
=  - a \int_{\d\Omega} (u_f - f) (1-\eta) G_R(x, y) d\sigma(x).
$$
When $\theta > 0$ is small enough we have that $B(y, 2\theta) \subset \Omega$ and so 
$- a \int_{\d\Omega} (u_f - f) (1-\eta) G_R(x, y) d\sigma(x) = -a\int_{\partial \Omega} (u_f-f)G_R(x,y)d\sigma(x)$. Putting all of this together we have that $$u_f(y) = a\int_{\partial \Omega} f(x)G_R(x,y)\, d\sigma(x)+ \int_\Omega A\nabla u_f \nabla (\eta G_R(x,y))\, dx.$$

Thus for \eqref{e:representationformula} we just need to check that
$$\lim_{\theta \downarrow 0}  \int_\Omega A\nabla u_f \nabla (\eta G_R(x,y))\, dx 
= \lim_{\theta \downarrow 0} \Big\{ 
\int_\Omega G_R(x,y) A\nabla u_f \nabla \eta \,dx + \int_\Omega \eta A\nabla u_f \nabla G_R(x,y)\, dx
\Big\} 
= 0.
$$
Note that the second integral, $\int_\Omega \eta A\nabla u_f \nabla G_R(x,y)\, dx \rightarrow 0$ because $\eta\nabla G_R(x,y) \rightarrow 0$ in $L^1(dx)$ and $A\nabla u_f \in L^\infty$ (again we are using the extra smoothness assumption on $A$). 

For the first integral $$\int_\Omega G_R(x,y) A\nabla u_f \nabla \eta \,dx \leq C\theta^{-1} \|\nabla u_f\|_{L^\infty(B(y, 2\theta))} \int_{B(y, 2\theta)} G_R(x,y)\, dx.$$ 
By the 
Sobolev embedding we have that $G_R(x,y) \in L^{\frac{n}{n-1}}(B(y,2\theta))$ and H\"older gives us $$\int_\Omega G_R(x,y) A\nabla u_f \nabla \eta \,dx  \leq C\theta^{-1} \theta^{n/n} \|\nabla u_f\|_{L^\infty(B(y, 2\theta))}\|G_R\|_{L^{\frac{n}{n-1}}(B(y,2\theta))} \rightarrow 0,$$ as $\theta \downarrow 0$. This finishes the proof of the representation formula for smooth $A$ and thus for general elliptic $A$.

 \end{proof}
 
 We call the function $G_R(x,y)$ introduced in Theorem \ref{t:GreenfunctionExistence} the Robin Green function with pole at $y$. In the next Section we use the representation formula \eqref{e:representationformula} to prove some surprising properties about the mutual absolute continuity of the Robin harmonic measure and measure $\sigma$.

			\section{Absolute continuity of the 
			Robin Harmonic Measure}\label{s:abscont}

			The goal of this section is to prove the mutual absolute continuity of the 
Robin harmonic measure with 
surface measure (our main Theorems \ref{t:mainac}, \ref{t:ainfinitysmall}, \ref{t:ainfinitylarge}). We should mention that all the arguments in this section are inspired by Jill Pipher's suggestion that we use the representation formula to prove $\sigma(E)= 0\Rightarrow \omega^X_R(E) = 0$ in Theorem \ref{t:mainac}. 
			
				\begin{proof}[Proof of Theorem \ref{t:mainac}]
				Use $f$ equal to the characteristic function of the set $E$ in \eqref{e:representationformula} and the second statement in Theorem~\ref{t:varequiv} together with the uniqueness of solutions to write 
\begin{equation}\label{e:repformulae}\omega^y_R(E)= 
		a\int_{E}G_R(x,y)d\sigma(x)  
		\,\,\text{for every } y\in \Omega. 
 \end{equation}
Hence, if $\sigma(E)=0$ then automatically $\omega^y_R(E)=0$.

In the other direction if $\omega^y_R(E) = 0$ for some $y \in \Omega$, 
then by the non-negativity and continuity in $x$ 
of the Green function, 
either $G_R(x,y) = 0$ for all $x\in E$ or $\sigma(E) = 0$. If $G_R(x,y) = 0$ for any $x\in E$, then the Harnack inequality up to the boundary, Theorem \ref{th:robinharnack}, implies that $G_R(x,y) = 0$ for all $x\in \overline{\Omega}$ close enough to $E$. Repeated applications of the Harnack inequalities show that $G_R(\cdot,y) \equiv 0$ on $\Omega$. 
However, if we let $\phi \equiv 1$ in \eqref{e:Greenidentity} this gives us a contradiction. 
Thus $\sigma(E) = 0$ (and hence, by the first part $\omega^z_R(E) = 0$ for all $z \in \Omega$ as well). 
			\end{proof}

			Our quantitative refinements of Theorem \ref{t:mainac} also depend on \eqref{e:representationformula}. To simplify exposition, we separate out the following key estimate, which follows immediately from \eqref{e:repformulae} above, and where we set
$\Delta = \Delta(x_0, r) = B(x_0,r)\cap \partial \Omega$  for $x_0 \in \d\Omega$ and $r > 0$,
and $y$ is any point of $\Omega$:
\begin{equation}\label{e:consofrep}
			\frac{\inf_{x \in \Delta} G(x, y)}{\sup_{x\in \Delta} G(x, y)} \ \frac{\sigma(E)}{\sigma(B(x_0,r))} \leq \frac{\omega_R^y(E)}{\omega_R^y(B(x_0,r))} 
		\leq \frac{\sup_{x\in \Delta} G(x, y)}{\inf_{x\in \Delta} G(x, y)} \ \frac{\sigma(E)}{\sigma(B(x_0,r))}.
			\end{equation}

			We are now ready to prove Theorems \ref{t:ainfinitysmall} and \ref{t:ainfinitylarge}. 
			
	\begin{proof}[Proof of Theorem \ref{t:ainfinitysmall}]
Recall that we are given $x_0 \in \d\Omega$, $r > 0$ such that $ar^{2-n} \sigma(B(x_0,r)) \leq 1$,
$E\subset B(x_0,r)\cap \partial \Omega$, and $X \in \Omega$ such that 
$|X-x_0| \geq C r$. Thus $r \leq C^{-1} \diam(\Omega)$, which helps us apply the results above.
We want to show that \eqref{e:Ainfinitysmall} holds for $y = X$, and because of \eqref{e:consofrep} it is enough to check that 
$$
\sup_{x \in \Delta} G(x, X) \leq C \inf_{x \in \Delta} G(x, X),
$$
again with $\Delta= B(x_0,r)\cap \partial \Omega$. 
We will even check that
\begin{equation} \label{6a3}
\sup_{x \in \Omega \cap B(x,r_0)} G(x, X) \leq C \inf_{x \in \Omega \cap B(x, r_0)} G(x, X),
\end{equation}
which is more general (and easier to manipulate).

We know that $G(\cdot, X)$ is a solution on $B(x_0, Cr)$, with vanishing Robin data on the boundary. 
Theorem \ref{th:robinharnack} applies if $C$ is large enough 
(compared to $K^2$) and in addition $ar^{2-n} \sigma(B(x_0,r)) \leq c_0$ 
(the constant in that theorem); observe that \eqref{e:notgenuine2} always hold when $\beta=0$.
The theorem then yields \eqref{6a3}.

If  $c_0 \leq ar^{2-n} \sigma(B(x_0,r)) \leq 1$, we observe that for $r_0 = \lambda r$,
$\lambda$ small enough, the asymptotic condition \eqref{1n3}, applied to any $x \in \Delta$, yields
\begin{equation} \label{6a4}
\begin{split}
a r_0^{2-n} &\sigma(B(x,r_0)) \leq a  r_0^{2-n}c_d^{-1} \lambda^d \sigma(B(x,r)) 
= c_d^{-1} a\lambda^{d-2} r^{2-n} \sigma(B(x,r)) \\
&\leq c_d^{-1} a\lambda^{d-2} r^{2-n} \sigma(B(x_0,2r)) \leq C \lambda^{d-2} a r^{2-n} \sigma(B(x_0,r)) \leq C \lambda^{d-2} \leq c_0
\end{split}
\end{equation}
by the doubling property of $\sigma$, our assumption on $r$, 
and if we take $\lambda$ small enough.
That is, the various $B(x,r_0)$ satisfy the assumption above, \eqref{6a3} holds for them,
and we can deduce \eqref{6a3} for $B(x_0, r)$ from Harnack's inequality (we merely loose a constant depending on $c_0$).
\end{proof}

The proof of ``large scale" quantitative absolute continuity proceeds similarly; we just need to
	use a larger buffer region. 

		\begin{proof}[Proof of Theorem \ref{t:ainfinitylarge}]
	The setting is the same as above, but this time $A = ar^{2-n} \sigma(B(x_0,r))$ is larger than $1$.
We still want to apply Theorem \ref{th:robinharnack} to the balls 
$B(x, r_0)$, $x \in \d\Omega \cap B(x_0,r)$, and again we want to choose $r_0 = \lambda r$
so small that $ar_0^{2-n} \sigma(B(x,r)) \leq c_0$. As in \eqref{6a4}, 
\begin{equation} \label{6a5}
\begin{split}
a r_0^{2-n} &\sigma(B(x,r_0)) \leq a  r_0^{2-n} c_d^{-1} \lambda^d \sigma(B(x,r)) 
= c_d^{-1} a \lambda^{d-2} r^{2-n} \sigma(B(x,r)) \\
&\leq c_d^{-1} a \lambda^{d-2} r^{2-n} \sigma(B(x_0,2r)) 
\leq C \lambda^{d-2} a r^{2-n} \sigma(B(x_0,r)) = C A \lambda^{d-2} 
\end{split}
\end{equation}		
and we get less than $c_0$ if we choose $\lambda = c A^{\frac{1}{d-2}}$, $c$ small enough.
And then, by Theorem \ref{th:robinharnack},  \eqref{6a3} holds for each ball 
$B(x, r_0)$, $x\in \d\Omega \cap B(x_0,r)$. In addition, if $\xi_x$ denotes a corkscrew point
for $B(x, r_0)$ and $\xi$ a corkscrew point for $B(x_0, r)$, we can connect $\xi_x$ to $\xi$
by a Harnack chain of length at most $L = C \log(r/r_0) + C\leq \frac{C}{d-2} \log(A) + C$,
and repeated uses of Harnack's inequality yield $C^{-L} u(\xi) \leq u(\xi_x) \leq C^L$
for $x\in \d\Omega \cap B(x_0,r)$. Since \eqref{6a3} also gives that
$C^{-1} u(\xi_x) \leq u(x) \leq C u(x_\xi)$, we see that
$|\log(u(x)/u(\xi))| \leq C L + C \leq \frac{C}{d-2} \log(A) + C$, and now
\eqref{e:Ainfinitylarge}, with $\gamma \leq \frac{C}{d-2}$, follows from \eqref{e:consofrep} as before.
\end{proof}


\begin{thebibliography}{AAA}
			\bibitem[AE95]{AE95}  W. Arendt and A.F.M. ter Elst.{\it Gaussian estimates for second order elliptic operators with boundary
				conditions.}  J. Operator Theory 38 (1) (1997) 87–130.
			
			\bibitem[AR19]{AR} K. Arfi and A. Rozanova-Pierrat. {\it Dirichlet-to-Neumann or Poincaré-Steklov operator on fractals described by d-sets.} Discrete Contin. Dyn. Syst., Ser. S 12, No. 1, 1--26, 2019.
			
			\bibitem[AGMT23]{AGMT} J. Azzam, J. Garnett, M. Mourgoglou and X. Tolsa. {\it  Uniform Rectifiability, Elliptic Measure, Square Functions, and $\epsilon$-Approximability Via an ACF Monotonicity Formula} International Mathematics Research Notices, Volume 2023, Issue 13, July 2023, Pages 10837–10941
			
			\bibitem[AHMMT20]{AHMMT} J. Azzam, S. Hofmann, J.M. Martell, M. Mourgoglou and X. Tolsa. {\it Harmonic measure and quantitative connectivity: geometric characterization of the $L^p$-solvability of the Dirichlet problem} Invent. Math. 222 (2020), no. 3, 881–993.
			
			\bibitem[AHMMMTV16]{seven} J. Azzam, S. Hofmann, J. M. Martell, S. Mayboroda, M. Mourgoglou, X. Tolsa, and A. Volberg. {\it Rectifiability of harmonic measure} Geom. Funct. Anal. 26 (2016), no. 3, 703–728.
			
			\bibitem[Azz20]{Azzamdrop} J. Azzam, {\it Dimension drop for harmonic measure on Ahlfors regular boundaries}, Potential
Anal. 53 (2020), no. 3, 1025–1041.
			
			\bibitem[BaG23]{Badger1} M. Badger, A. Genschaw. {\it Lower bounds on Bourgain’s constant for harmonic measure.} Canad.
J. Math. First View, 1–20.
			
			\bibitem[BaG24]{Badger2} M. Badger, A. Genschaw. {\it Hausdorff dimension of caloric measure.} To appear in Amer. J.
Math.

\bibitem[Bat96]{Batakis} A. Batakis. {\it Harmonic measure of some Cantor type sets,} Ann. Acad. Sci. Fenn.
Math. 21 (1996), no. 2, 255–270. MR 1404086
			
			\bibitem[BBC08]{BBC} R. Bass, K. Burdzy and Z.-Q. Chen, 
			{\it On the Robin problem in fractal domains} Proc. Lond. Math. Soc. (3)96(2008), no.2, 273–311.
			
			\bibitem[B86]{B86} M. H. Bossel 
			{\it Membranes \'{e}lastiquement li\'{e}es: extension du th\'{e}or\`{e}me de Rayleigh-Faber-
				Krahn et de l’in\'{e}galit\'{e} de Cheeger} C. R. Acad. Sci. Paris S\'{e}r. I Math., 302(1):47–50, 1986.
				
				\bibitem[Bou87]{Bourgain} J. Bourgain, {\it On the Hausdorff dimension of harmonic measure in higher dimension,} Invent.
Math. 87 (1987), no. 3, 477–483. MR 874032
			
			
			\bibitem[BFK17]{BFK17} D. Bucur, P. Freitas, J. Kennedy {\it  The Robin problem } (book chapter, open access) Shape optimization and spectral theory Ed. Antoine Henrot, 78–119, De Gruyter Open, Warsaw, 2017. 
			
			\bibitem[BG15]{BG15}D. Bucur, A. Giacomini {\it A variational approach to the isoperimetric inequality for the
				Robin eigenvalue problem} Arch. Ration. Mech. Anal., 198(3):927–961, 2010.
			
			
			
			
			
			\bibitem[BGN22]{BGN} D. Bucur, 
			A. Giacomini, M. Nahon. {\it Boundary behavior of Robin problems in non-smooth domains}. arXiv:2206.09771 preprint (2022)
			
			\bibitem[BNNT22]{BNNT} D. Bucur, M. Nahon, C. Nitsch, C. Trombetti. {\it Shape optimization of a thermal insulation problem} Calc. Var. Partial Differential Equations 61 (2022), no. 5, Paper No. 186, 29 pp.  
			
			\bibitem[CFK14]{CFK} L.A. Caffarelli, E.B. Fabes, and C.E. Kenig. {\it Completely singular elliptic-harmonic measures.} Indiana Univ. Math. J. 30 (1981), no. 6, 917–924
			
			\bibitem[CaKr16]{kriv} L. A. Caffarelli, D. Kriventsov. {\it 
				A free boundary problem related to thermal insulation.}
			Comm. Partial Differential Equations41(2016), no.7, 1149–1182.
			
			\bibitem[Car85]{Carleson} Lennart Carleson, {\it On the support of harmonic measure for sets of Cantor type,} Ann. Acad.
Sci. Fenn. Ser. A I Math. 10 (1985), 113–123. MR 802473
			
			\bibitem[CF74]{CF} R. R. Coifman, C.\,Fefferman, {\it 
Weighted norm inequalities for maximal functions and singular integrals.}
Studia Math. 51 (1974), 241--250. 
			
			\bibitem[ChKi14]{CK} J. Choi,  and S. Kim, {\it Green's functions for elliptic and parabolic systems with
				{R}obin-type boundary conditions} J. Funct. Anal. 267 (2014), no. 9, 3205-3261.
			
			\bibitem[D06]{D06} D. Daners {\it A Faber-Krahn inequality for Robin problems in any space dimension} Math. Ann.,
			335(4):767–785, 2006.
			
			
			
			\bibitem[DEM21]{DEM} G. David, M. Engelstein and S. Mayboroda. {\it
				Square functions, nontangential limits, and harmonic measure in codimension larger than 1.} Duke Math. J. 170 (2021), no. 3, 455–501. 
				
			\bibitem[DFM2]{DFM2} G. David, J.  Feneuil, S. Mayboroda. 
{\em Elliptic theory for sets with higher co-dimensional boundaries}. 
Mem. Amer. Math. Soc. 274 (2021), no. 1346, vi+123 pp. 

			\bibitem[DFM23]{DFM20} G. David, J. Feneuil and S. Mayboroda. 
			{\it Elliptic Theory in Domains With Boundaries of Mixed Dimension} 
			Ast\'erisque (2023), no. 442, vi+139 pp. 
			
		\bibitem[DJJ]{NoDrop}	G. David, C. Jeznach, A. Julia, 
Cantor sets with absolutely continuous harmonic measure.
J. Éc. polytech. Math.10(2023), 1277–1298.
			
			\bibitem[DM21]{DM} G. David and S. Mayboroda 
			{\it Good elliptic operators on Cantor sets.} 
			Adv. Math. 383 (2021), Paper No. 107687, 21 pp.  
			
			\bibitem[DM23]{DM23} G. David and S. Mayboroda. 
			{\it  Harmonic measure is absolutely continuous with respect to the Hausdorff measure on all 				low-dimensional uniformly rectifiable sets.} 
			International Mathematics Research Notices, Vol. 23, No. 11 (2023), pp. 9319--9426. 
			
			\bibitem[Dav89]{Davies} E. B. Davies 
			{\it Heat kernels and spectral theory.}
			Cambridge Tracts in Math., 92
			Cambridge University Press, Cambridge, 1989. x+197 pp.
			
			\bibitem[DSS20]{BoundaryHarnack} D. De Silva and O. Savin. 
			{\it A short proof of boundary Harnack principle.}  
			J. Differential Equations 269 (2020), no. 3, 2419--2429.
			
			\bibitem[DL21]{DL} H. Dong, Z. Li. 
	{\it The conormal and Robin boundary value problems in nonsmooth domains satisfying a measure condition.}
			J. Funct. Anal. 281 (2021), no. 9, Paper No. 109167, 32 pp. 
			
			\bibitem[FKS82]{fks} E. Fabes, C. E. Kenig, and R. Serapioni. 
			{\it The local regularity of solutions of degenerate elliptic equations.}
			Comm. Partial Differential Equations 7(1982), no.1, 77–116.
			
			\bibitem[F-RR-O22]{FR} X. Fern\'andez-Real and X. Ros-Oton. {\it Regularity theory for Elliptic PDE}. Zurich Lectures in Advances Mathematics. EMS book, 2022. 

\bibitem[FGAS08]{FGAS} M. Filoche  D. S. Grebenkov, J. S. Andrade, Jr., and B. Sapoval. {\it Passivation of irregular surfaces accessed by diffusion}. PNAS USA 105 2008, 7636-7640.
   
			\bibitem[GR85]{GaRu} J.\,Garcia-Cuerva, J.\,Rubio de Francia, 
{\it Weighted norm inequalities and related topics.}
North-Holland Mathematics Studies, 116. Notas de Matem\'atica [Mathematical Notes], 104. North-Holland Publishing Co., Amsterdam, 1985. 
			
			\bibitem[GMN14]{GMN} F. Gesztesy, M. Mitrea, R. Nichols, {\it
				Heat kernel bounds for elliptic partial differential operators in divergence form with Robin-type boundary conditions.}
			J. Anal. Math. 122 (2014), 229–287.
=
			
			\bibitem[GFS03]{GFS} D.S. Grebenkov, M. Filoche, and B. Sapoval. {\it Spectral properties of the Brownian self-transport operator} Eur. Phys. J. B 36, 221–231 (2003). 
			
			\bibitem[GFS06]{GFS2} D.S. Grebenkov, M. Filoche, and B. Sapoval. {\it Mathematical basis for a general theory of Laplacian transport towards irregular interfaces.} Physical Review E 73, 021-103 
(2006).
			
			\bibitem[GW-82]{GW} M. Gr\"uter and K.-O. Widman. {\it The Green function for uniformly elliptic equations} Manuscripta Math., 37, 303–342 (1982).
			\bibitem[HKT]{HKT}P. Hajłasz, P. Koskela and H. Tuominen. {\it Sobolev embeddings, extensions and measure density condition} J. Funct. Anal. 254, No. 5, 1217--1234, 2008.
			\bibitem[HKM06]{finnish} J. Heinonen, T. Kilpel\"ainen, and O. Martio. {\it Nonlinear potential theory of degenerate elliptic equations.}
			Unabridged republication of the 1993 original
			Dover Publications, Inc., Mineola, NY, 2006. xii+404 pp.
			\bibitem[HM14]{HM}S. Hofmann and J.M. Martell. {\it Uniform rectifiability and harmonic measure I: Uniform rectifiability implies Poisson kernels in $L^p$}. Ann. Sci ENS 47 (2014), no. 3, 577-654.
			
			\bibitem[HMM16]{HMM} S. Hofmann, J.M. Martell, and S. Mayboroda. 
			{\it Uniform rectifiability, Carleson measure estimates,
		and approximation of harmonic functions.} Duke Math. J. 165 (2016), no. 12, 2331–2389.
			
			\bibitem[HMMTZ21]{HMMTZ} S. Hofmann, Steve, J. M. Martell, S. Mayboroda, T. Toro, and Z. Zhao. {\it Uniform rectifiability and elliptic operators satisfying a Carleson measure condition.} Geom. Funct. Anal.31(2021), no.2, 325–401.
			
			\bibitem[JK82]{JKNTA} D. Jerison and C.E. Kenig. {\it Boundary behavior of harmonic functions in nontangentially accessible domains.} Adv. in Math. 46 (1982), no. 1, 80–147. 
			\bibitem[J]{J} P. W. Jones. {\it Quasiconformal mappings and extendability of functions in Sobolev spaces} Acta Math. 147, 71--88, 1981.
			\bibitem[JW88]{JonesWolff} Peter W. Jones and Thomas H. Wolff, Hausdorff dimension of harmonic measures in the
plane, Acta Math. 161 (1988), no. 1-2, 131–144. MR 962097
			\bibitem[Jo83]{Jou} J.-L. Journ\'e. {\it Calder\'on-Zygmund operators, pseudodifferential operators and the Cauchy integral of Calder\'on}. Lecture Notes in Mathematics, 994. Springer-Verlag, Berlin, 1983. 
			\bibitem[K94]{kenigbook}C. E. Kenig. {\it Harmonic analysis techniques for second order elliptic boundary value problems.} CBMS Regional Conference Series in Mathematics, 83. Published for the Conference Board of the Mathematical Sciences, Washington, DC; by the American Mathematical Society, Providence, RI, 1994. xii+146 pp.
            \bibitem[KKPT]{KKPT} C. Kenig, H. Koch, J. Pipher and T. Toro, {\it A New Approach to Absolute Continuity of Elliptic Measure, with Applications to Non-symmetric Equations},
            Adv. Math. 153, 2 (2000), 231-298.

			\bibitem[KP93]{KP2} C. Kenig and J. Pipher, {\it The Neumann problem for elliptic equations with nonsmooth coefficients}, Invent. Math. 113, 3 (1993), 447–509.
			\bibitem[KP01]{KP} C. Kenig, and J. Pipher. {\it The Dirichlet problem for elliptic equations with drift terms.} Publ. Mat. 45 (2001), 199–217.
			\bibitem[Kim15]{Kim} S. Kim {\it Note on local boundedness for weak solutions of {N}eumann problem for second-order elliptic equations} J. Korean Soc. Ind. Appl. Math 19 (2015), no. 2, 189--195.
			\bibitem[LS04]{LS} L. Lanzani, Z. Shen.
			{\it On the Robin boundary conditions for Laplace's equation in Lipschitz domains.}
			Commun. Partial Differ. Equations, 29(1-2): 91-109, 2004.
			
			\bibitem[Mak85]{Makarov1} N. G. Makarov, On the distortion of boundary sets under conformal mappings, Proc. London
Math. Soc. (3) 51 (1985), no. 2, 369–384. MR 794117

\bibitem[Mak85b]{Makarov2} N.G. Makarov, Harmonic measure and the Hausdorff measure. (Russian) Doklady Akademii Nauk SSSR (3)280 (1985), 545–548.


			
			\bibitem[MP21]{MP} S. Mayboroda and B. Poggi.
	{\it Carleson perturbations of elliptic operators on domains with lower dimensional boundaries.}
			J. Funct. Anal. 280, no. 8 (2021): 108930. 
			
			\bibitem[Maz11]{Mazya} V. Maz'ya, {\it Sobolev spaces with applications to elliptic partial differential equations.}
			Grundlehren Math. Wiss., 342 [Fundamental Principles of Mathematical Sciences]
			Springer, Heidelberg, 2011, xxviii+866 pp.
			
			\bibitem[MM81]{MM} L. Modica, and S. Mortola. {\it  Construction of a singular elliptic-harmonic measure.} Manuscripta Math. 33 (1980/81), no. 1, 81–98.
			
			\bibitem[Pi14]{JillSurvey} J. Pipher {\it 
				Carleson measures and elliptic boundary value problems.} Proceedings of the International Congress of Mathematicians—Seoul 2014. Vol. III, 387–400.
			Kyung Moon Sa, Seoul, 2014.
			
			\bibitem[RR]{RR} F. and M. Riesz, {\it \"Uber die randwerte einer analtischen funktion}.
Compte Rendues du Quatri\`eme Congr\`es des Math\'ematiciens Scandinaves, Stockholm 1916,
Almqvists and Wilksels, Upsala, 1920.
			
			\bibitem[St65]{Sta} G. Stampacchia {\it Le probl\`eme de Dirichlet pour les \'equations elliptiques du second ordre \`a coefficients discontinus.} Annales de l'Institut Fourier 15, 1, 189--257, 1965.
			
			
			\bibitem[To10]{TatianaSurvey} T. Toro {\it Potential analysis meets geometric measure theory.} Proceedings of the International Congress of Mathematicians. Volume III, 1485–1497.
			Hindustan Book Agency, New Delhi; distributed by, 2010.
			
			\bibitem[Tol24]{Tolsadrop} X. Tolsa, The dimension of harmonic measure on some AD-regular flat sets of fractional dimension.
Int. Math. Res. Not. IMRN(2024), no. 8, 6579–6605.
			
			\bibitem[V-S12]{VS} A. V\'elez-Santiago, {\it Solvability of linear local and nonlocal Robin problems over $C(\overline{\Omega})$}. J. Math. Anal. Appl.386(2012), no.2, 677–698.
			
			\bibitem[Vol92]{Volberg1} Volberg, A. L. ``On the harmonic measure of self-similar sets on the plane." Harmonic Analysis and Discrete Potential theory. 267–80. Springer, 1992.  MR1222465

\bibitem[VOl93]{Volberg2} Volberg, A. ``On the dimension of harmonic measure of cantor repellers." Michigan Math. J. 40, no. 2 (1993): 239–58. https://doi.org/10.1307/mmj/1029004751.  MR1226830 
			
			\bibitem[W84]{bioref} E. R. Weibel, {\it The Pathway for Oxygen. Structure and Function in the Mammalian Respiratory System} Harvard University Press, Cambridge, MA, 1984.
			
			\bibitem[Wol91]{Wolff} Thomas H. Wolff, Counterexamples with harmonic gradients in R3, Essays on Fourier analysis
in honor of Elias M. Stein (Princeton, NJ, 1991), Princeton Math. Ser., vol. 42, Princeton
Univ. Press, Princeton, NJ, 1995, pp. 321–384. MR 1315554
			
			\bibitem[YYY18]{YYY} S. Yang, D. Yang, W. Yuan, Wen. {\it 
				The Lp Robin problem for Laplace equations in Lipschitz and (semi-)convex domains.}
			J. Differential Equations 264 (2018), no. 2, 1348–1376. 
			
			\bibitem[Z89]{ziemer} W. Ziemer. {\it Weakly differentiable functions.
				Sobolev spaces and functions of bounded variation} Grad. Texts in Math., 120 Springer-Verlag, New York, 1989. xvi+308 pp.
			
		\end{thebibliography}
	\end{document}